\documentclass[12pt,twoside]{amsart}
\usepackage{amssymb,amsmath,amsthm, amscd, enumerate, mathrsfs}
\usepackage{graphicx, hhline}
\usepackage[all]{xy}
\usepackage[usenames]{color}
\usepackage{hyperref}
\usepackage{fancyhdr}
\usepackage{bbm}
\usepackage[top=30truemm,bottom=30truemm,left=30truemm,right=30truemm]{geometry}

\hypersetup{colorlinks=true}

\title[Boundedness and moduli]{On boundedness and moduli spaces of K-stable Calabi-Yau fibrations over curves}
\author{Kenta Hashizume and Masafumi Hattori}
\date{\today}
\address{Kenta Hashizume \\ Department of 
Mathematics, Faculty of Science, Niigata University, Niigata 950-2181, Japan}
\address{Institute for Research Administration, Niigata University, Niigata 950-2181, Japan}
\email{hkenta@math.sc.niigata-u.ac.jp}
\address{Masafumi Hattori \\ Department of 
Mathematics, Graduate School of Science, 
Kyoto University, Kyoto 606-8502, Japan}
\email{hattori.masafumi.47z@st.kyoto-u.ac.jp}

\newtheorem{thm}{Theorem}[section]

\newtheorem{lem}[thm]{Lemma}
\newtheorem{cor}[thm]{Corollary}
\newtheorem{prop}[thm]{Proposition}
\newtheorem{conj}[thm]{Conjecture}

\newtheorem{claim}{Claim}

\theoremstyle{definition}
\newtheorem{defn}[thm]{Definition}
\newtheorem{ex}[thm]{Example}
\newtheorem{rem}[thm]{Remark}
\newtheorem{note}[thm]{Notation}

\newtheorem{step}{Step}

\newtheorem*{claim*}{Claim}
\begin{document}

\maketitle

\begin{abstract}
    We show boundedness of polarized Calabi--Yau fibrations over curves only with fixed volumes of general fibers and Iitaka volumes.
    As its application, we construct a separated coarse moduli space of K-stable Calabi-Yau fibrations over curves in an adiabatic sense \cite{Hat} and show that all members (resp.~smooth members) of the moduli are simultaneously uniformly K-stable (resp.~have cscK metrics) for a certain choice of polarizations.
\end{abstract}


\section{Introduction}\label{intro}
\subsection{Moduli problem}
Classification of higher-dimensional algebraic varieties by their geometries is one of the most important problems in algebraic geometry. 
Moduli spaces, that are parameter spaces of specific classes of varieties, are effective tools to classify varieties. 

The moduli spaces of stable curves were constructed by Deligne--Mumford \cite{deligne-mumford} as Deligne-Mumford stacks, and they are compactifications of the moduli spaces of smooth curves of general type of fixed genus (cf.~\cite{GIT}). 
After \cite{deligne-mumford}, the moduli spaces of stable curves and the moduli spaces of canonically polarized surfaces with only canonical singularities have been constructed by Mumford \cite{M} and Gieseker \cite{gieseker}, respectively. 
For the construction, Mumford's geometric invariant theory (GIT, for short, see \cite{GIT}) was used, and the GIT-stability of those varieties was studied to apply the GIT. 
However, it is very difficult to detect the GIT-stability (more precisely, the asymptotic Chow stability, cf.~\cite{M}) of other kinds of polarized varieties. 
As an other strategy, Koll\'{a}r and Shepherd-Barron \cite{KSB} (see also \cite{kollar-moduli-stable-surface-proj} by Koll\'ar and \cite{ale94} by Alexeev) used the minimal model theory to construct the moduli spaces of stable surfaces. 
By their works and \cite{ale96}, semi log canonical models turned out to be a suitable higher-dimensional analog of stable curves to construct the moduli space. 
Their moduli spaces, called KSBA-moduli, have been completed as a full generalization of the moduli of stable curves.  
For details, see \cite{kollar-moduli} by Koll\'ar and the references therein. 
The recent developments of the minimal model theory (\cite{BCHM}, \cite{HX}, \cite{hmx-boundgentype}) are indispensable for the theory of KSBA-moduli.
The construction in \cite{gieseker} by using GIT does not work for the KSBA-moduli because there is a klt variety with the ample canonical divisor that is asymptotically Chow unstable (see \cite{O} for example). 

We need a polarization when we discuss moduli theory of varieties whose canonical divisor is neither ample nor anti-ample. 
See \cite{S} by Seiler and \cite{viehweg95} by Viehweg for the study of the moduli theory of good minimal models with polarizations. 
In general, the moduli theory for non-canonically polarized varieties is much more complicated.
For example, we cannot directly apply the theories as above to construct separated moduli of all polarized rational elliptic surfaces.
The GIT-stability of rational Weierstrass fibrations (\cite{Mi}) and Halphen pencils (\cite{Mi2}, \cite{Zan}, \cite{hz}) was investigated to consider the moduli of rational elliptic surfaces from the viewpoint of GIT.
Seiler \cite{S} constructed the moduli space of some polarized elliptic surfaces by applying the GIT. 
He treated not only elliptic surfaces with nef canonical divisors but also rational elliptic surfaces whose fibers are reduced or of $_mI_n$-type. 
However, Seiler did not study the Chow stability of all polarized rational elliptic surfaces. 
By \cite[Corollary 5.7]{Hat},  the moduli constructed by Seiler does not contain a polarized smooth rational elliptic surface $(X,L)$ of index two with a unique constant scalar curvature K\"{a}hler (cscK, for short) metric. 
This $(X,L)$ is asymptotically Chow stable (cf.~\cite{Dn}). 
Thus, we naturally expect the existence of a moduli space parametrizing more polarized rational elliptic surfaces.

\subsection{K-stability and K-moduli}\label{sec-K-stk-mod}
K-stability was introduced by Tian \cite{T} and Donaldson \cite{Dn2} in the context of the K\"{a}hler geometry to detect the existence of cscK metrics, and the notion is closely related to the GIT. 
In \cite{O}, Odaka found a relationship between the K-stability and the minimal model theory, and he proved that semi log canonical models are K-stable. 
This implies that the KSBA-moduli is a kind of moduli of K-polystable varieties. 
A moduli space parametrizing all K-polystable varieties is called a {\it K-moduli}. 
Odaka proposed the following conjecture. 

\begin{conj}[K-moduli conjecture, {cf.~\cite[Conjecture 5.2]{O2}}]\label{conj--K-moduli}
    There exists a quasi-projective moduli scheme parametrizing all polarized K-polystable varieties with fixed some numerical data (e.g. genera of curves, or volumes of polarizations).
\end{conj}

This conjecture was motivated by the work of Fujiki--Schumacher \cite{FS} on the construction and a partial projectivity of moduli spaces of some projective manifolds with unique cscK metrics. 
On the other hand, Dervan--Naumann \cite{DN} constructed the moduli spaces of projective manifolds admitting cscK metrics and non-discrete automorphism groups. 
As the Yau-Tian-Donaldson conjecture predicts that the K-polystability is equivalent to the existence of cscK metrics, the K-moduli can be thought of an algebro-geometric generalization of the moduli in \cite{FS} and \cite{DN}.

For the K-stability of log Fano pairs, algebraic geometers and differential geometers have made remarkable progress and they constructed the projective moduli space of all K-polystable log Fano pairs (cf.~\cite{LWX}, \cite{ABHLX}, \cite{XZ1}, \cite{LXZ}). 
Moreover, quasi projective moduli schemes of polarized K-polystable Calabi-Yau varieties have already been constructed in several cases (cf.~\cite{O4}). Thus, it seems that we can make use of the K-stability to construct moduli spaces of polarized varieties. 
The following problems are keys to construct the desired moduli spaces as Deligne-Mumford stacks. 

\begin{enumerate}[(I)]
    \item\label{intro-(2)} 
    Boundedness: Are polarized K-stable varieties parametrized by a scheme of finite type over $\mathbb{C}$?
    \item\label{intro-(3)}
     Openness: Do polarized K-stable varieties form an open subset of the Hilbert scheme?
    \item\label{intro-(4)}
     Separatedness: Let $C$ be a smooth curve, $0\in C$ and $(\mathcal{X}^\circ,\mathcal{L}^\circ)\to C^\circ$ be a family of polarized K-stable varieties over $C^{\circ}=C\setminus\{0\}$. 
    Let $(\mathcal{X},\mathcal{L})\to C$ and $(\mathcal{X}',\mathcal{L}')\to C$ be two extensions of this family over $C$. If $(\mathcal{X}_0,\mathcal{L}_0)$ and $(\mathcal{X}'_0,\mathcal{L}_0')$ are K-stable, then $(\mathcal{X},\mathcal{L})\cong(\mathcal{X}',\mathcal{L}')$?
\end{enumerate}
In the case of (uniformly) K-stable $\mathbb{Q}$-Fano varieties, these problems had been settled. 
More precisely, (\ref{intro-(2)}) was solved by Jiang \cite{Ji} (see also \cite{XZ}), (\ref{intro-(3)}) was solved by Blum--Liu \cite{BL}, and (\ref{intro-(4)}) was solved by Blum--Xu \cite{BX}. 
Their proofs are based on the work of Blum--Jonsson \cite{BlJ}, which shows that the $\delta$-invariant introduced by Fujita--Odaka \cite{FO} completely detects uniform K-stability of log Fano pairs.
However, there are few criteria of the K-stability for other kinds of polarized varieties, and we do not know whether (\ref{intro-(2)})--(\ref{intro-(4)}) hold or not.

\subsection{Adiabatic K-stability and moduli}\label{subsec-intro-adkmod}
Adiabatic K-stability was introduced by the second author \cite{fibst} and it was inspired by the works of Fine \cite{Fine}, \cite{Fine2} and Dervan-Sektnan \cite{DS1}, \cite{DS2} on the existence problem of cscK metrics of fibrations.
Frankly speaking, uniform adiabatic K-stability (\cite[Definition 2.6]{Hat}) is designed to be ``uniform" K-stability of fiber spaces when their polarizations are very close to ample line bundles on the base spaces. Such K-stability and cscK metrics on fiber spaces when their polarizations are very close to ample line bundles on the base are studied in \cite{Fine}, \cite{Fine2}, \cite{DS1}, \cite{DS2}. On the other hand, Dervan--Ross \cite{DR2} points out there is a relationship between adiabatic K-stability and ``K-stability" of the base. More precisely, they show that adiabatic K-semistability implies twisted K-semistability of the base. Recently, by replacing log twisted K-stability with twisted K-stability, the second author \cite{Hat} showed for klt--trivial fibrations over curves that uniform adiabatic K-stability are equivalent to log-twisted K-stability of the base. Moreover, he showed the existence of cscK metrics corresponding to  the uniform adiabatic K-stability for klt--trivial fibrations over curves. By using this criterion, 
the uniform adiabatic K-stability of elliptic surfaces is closely related to the GIT-stability of rational Weierstrass fibrations and Halphen pencils (cf.~\cite[\S5]{Hat}, \cite[Remark 4.3]{hz}).
Moreover, elliptic surfaces treated by Seiler are uniformly adiabatically K-stable, and the result in \cite{Hat} (cf.~Definition \ref{unifdef}) gave a useful characterization of the uniform adiabatic K-stability for klt-trivial fibrations over curves. 
Quite recently, (\ref{intro-(4)}) was also proved by the second author \cite{CM} over $\mathbb{C}$, thus we may expect a variant of Conjecture \ref{conj--K-moduli} for the uniform adiabatic K-stability in an appropriate formulation. 

The main purpose of this paper is to prove the following result. 
\begin{thm}\label{thm--main-rough-statement}
There exists a moduli space parametrizing uniformly adiabatically K-stable klt-trivial fibrations over curves as a separated algebraic space of finite type.
\end{thm}
To state the result more precisely, we prepare some notations. 
Let $d$ be a positive integer, $v$ a positive rational number, and $u$ a rational number. 
We set
$$\mathfrak{Z}_{d, v,u}:=\left\{
\begin{array}{l}
f\colon (X,\Delta=0,A) \to C
\end{array}
\;\middle|
\begin{array}{rl}
(i)&\text{$f$ is a uniformly adiabatically}\\
&\text{K-stable polarized klt-trivial} \\
&\text{fibration over a curve $C$,}\\
(ii)&\text{${\rm dim}X=d$,}\\
(iii)&\text{$K_X\equiv uf^*H$ for some line bundle} \\
&\text{$H$ on $C$ such that $\mathrm{deg}\,H=1$,} \\
(iv)&\text{$A$ is an $f$-ample line bundle on}\\
&\text{$X$ such that $(H\cdot A^{d-1})=v$.}
\end{array}\right\}.$$ 
When $u\ne0$, the boundedness result by Birkar (\cite{birkar-geometry-moduli}) implies the effectivity of the klt-trivial fibrations (see Lemma \ref{lem--Cartierindex}).
More precisely, there exists a positive integer $r$, depending only on $d$, $u$, and $v$, such that for any element $f\colon (X,0,A) \to C$ of $\mathfrak{Z}_{d, v,u}$, $erK_{X}$ is a base point free Cartier divisor and the linear system $|erK_{X}|$ defines $f$, where $e:=\frac{u}{|u|}$. 
We can write the precise statement of Theorem \ref{thm--main-rough-statement} with these notations. 

\begin{thm}\label{quesmain}
We fix $d\in\mathbb{Z}_{>0}$, $u\in\mathbb{Q}_{<0}$, $v\in\mathbb{Q}_{>0}$, and $r \in\mathbb{Z}_{>0}$ as above. 
For any locally Noetherian scheme $S$ over $\mathbb{C}$, we define $\mathscr{M}_{d,v,u,r}(S)$ to be a groupoid whose objects are 
 $$\left\{
 \vcenter{
 \xymatrix@C=12pt{
(\mathcal{X},\mathscr{A})\ar[rr]^-{f}\ar[dr]_{\pi_{\mathcal{X}}}&& \mathcal{C} \ar[dl]\\
&S
}
}
\;\middle|
\begin{array}{rl}
(i)&\text{$\pi_{\mathcal{X}}$ is a flat projective morphism and $\mathcal{X}$ is a scheme,}\\
(ii)&\text{$\mathscr{A}\in\mathrm{Pic}_{\mathcal{X}/S}(S)$ (see Subsection \ref{subsec-Hilb}) such that $\mathscr{A}_{\bar{s}}$ is}\\
&\text{$f_{\bar{s}}$-ample for any geometric point $\bar{s}\in S$,}\\
(iii)&\text{$\omega_{\mathcal{X}/S}^{[r]}$ (see Definition \ref{defn-q-gor}) exists as a line bundle,}\\
(iv)&\text{$\pi_{\mathcal{X}*}\omega_{\mathcal{X}/S}^{[-lr]}$ is locally free and it generates}\\
&\text{$H^0(\mathcal{X}_{s}, \mathcal{O}_{\mathcal{X}_{s}}(-lrK_{\mathcal{X}_{s}}))$ for any point $s\in S$ and any}\\
&\text{$l\in\mathbb{Z}_{>0}$,}\\
(v)&\text{$f$ is the ample model of $\omega_{\mathcal{X}/S}^{[-r]}$ over $S$ and}\\
&\text{$(\mathcal{X}_{\overline{s}},0,\mathscr{A}_{\overline{s}})\to \mathcal{C}_{\overline{s}} \in \mathfrak{Z}_{d, v,u}$ for any geometric}\\
&\text{point $\overline{s}\in S$}
\end{array}\right\}.$$
Here, we define an isomorphism $\alpha\colon (f \colon (\mathcal{X},\mathscr{A})\to\mathcal{C})\to (f' \colon (\mathcal{X}',\mathscr{A}')\to\mathcal{C}')$ of any two objects of $\mathscr{M}_{d,v,u,r}(S)$ to be an $S$-isomorphism $\alpha\colon\mathcal{X}\to\mathcal{X}'$ such that there exists $\mathscr{B}\in\mathrm{Pic}_{\mathcal{C}/S}(S)$ satisfying that $\alpha^*\mathscr{A}'=\mathscr{A}\otimes f^*\mathscr{B}$ as elements of $\mathrm{Pic}_{\mathcal{X}/S}(S)$.

Then $\mathscr{M}_{d,v,u,r}$ is a separated Deligne-Mumford stack of finite type over $\mathbb{C}$ with a coarse moduli space.
\end{thm}

We emphasize that $\mathscr{A}_{\bar{s}}$ are not assumed to be ample or the volumes of $\mathscr{A}_{\bar{s}}$ in $\mathscr{M}_{d,v,u,r}(S)$ are not bounded from above. 
Theorem \ref{quesmain} is the combination of the conditions (\ref{intro-(2)})--(\ref{intro-(4)}) for the uniform adiabatic K-stability (\cite[Theorem B]{Hat}, Corollary \ref{cor--hilbertpolynomial}, Theorem \ref{op}, and Theorem \ref{sep2}) and Theorems \ref{modulimain} and \ref{mainbound} below, which are also key ingredients.

The first ingredient (Theorem \ref{modulimain} below) is the existence of a separated coarse moduli space that parametrizes $f\colon(X,0,A)\to C\in\mathfrak{Z}_{d, v,u}$ such that $f$ is uniformly adiabatically K-stable and $A$ is an ample line bundle whose volume is bounded from above.
We set
$$\mathfrak{Z}_{d, v,u,w}:=\left\{
\begin{array}{l}
f\colon (X,0,A) \to C\in\mathfrak{Z}_{d, v,u}\!
\end{array}
\;\middle|
\begin{array}{l}
\text{$A$ is (globally) ample and $\mathrm{vol}(A)\le w$}
\end{array}\right\}$$ 
for any positive rational number $w$. 
Then the following holds.

\begin{thm}[see Theorem \ref{mod2}]\label{modulimain}
We fix $d\in\mathbb{Z}_{>0}$, $u\in\mathbb{Q}_{\ne0}$ with $e:=\frac{u}{|u|}$, $v\in\mathbb{Q}_{>0}$, $w\in\mathbb{Q}_{>0}$, and $r \in\mathbb{Z}_{>0}$ as above. 
For any locally Noetherian scheme $S$ over $\mathbb{C}$, we define $\mathscr{M}_{d,v,u,w,r}(S)$ to be a groupoid whose objects are 
 $$\left\{
 \vcenter{
 \xymatrix@C=12pt{
(\mathcal{X},\mathscr{A})\ar[rr]^-{f}\ar[dr]_{\pi_{\mathcal{X}}}&& \mathcal{C} \ar[dl]\\
&S
}
}
\;\middle|
\begin{array}{rl}
(i)&\text{$\pi_{\mathcal{X}}$ is a flat projective morphism and $\mathcal{X}$ is a scheme,}\\
(ii)&\text{$\mathscr{A}\in\mathrm{Pic}_{\mathcal{X}/S}(S)$ such that $\mathscr{A}_{\bar{s}}$ is ample for any}\\
&\text{geometric point $\bar{s}\in S$,}\\
(iii)&\text{$\omega_{\mathcal{X}/S}^{[r]}$ exists as a line bundle,}\\
(iv)&\text{$\pi_{\mathcal{X}*}\omega_{\mathcal{X}/S}^{[ler]}$ is locally free and it generates}\\
&\text{$H^0(\mathcal{X}_{s}, \mathcal{O}_{\mathcal{X}_{s}}(lerK_{\mathcal{X}_{s}}))$ for any point $s\in S$ and any}\\
&\text{$l\in\mathbb{Z}_{>0}$,}\\
(v)&\text{$f$ is the ample model of $\omega_{\mathcal{X}/S}^{[er]}$ over $S$ and}\\
&\text{$(\mathcal{X}_{\overline{s}},0,\mathscr{A}_{\overline{s}})\to \mathcal{C}_{\overline{s}} \in \mathfrak{Z}_{d, v,u,w}$ for any geometric}\\
&\text{point $\overline{s}\in S$}
\end{array}\right\}.$$
Here, we define an isomorphism $\alpha\colon (f \colon (\mathcal{X},\mathscr{A})\to\mathcal{C})\to (f' \colon (\mathcal{X}',\mathscr{A}')\to\mathcal{C}')$ of any two objects of $\mathscr{M}_{d,v,u,w,r}(S)$ to be an $S$-isomorphism $\alpha\colon\mathcal{X}\to\mathcal{X}'$ such that $\alpha^*\mathscr{A}'=\mathscr{A}$ as elements of $\mathrm{Pic}_{\mathcal{X}/S}(S)$.

Then $\mathscr{M}_{d,v,u,w,r}$ is a separated Deligne-Mumford stack of finite type over $\mathbb{C}$ with a coarse moduli space.
\end{thm}

When $u>0$, Theorem \ref{modulimain} shows the existence of the moduli of the Iitaka fibrations from klt good minimal models of Iitaka dimension one (see \cite{birkar-moduli-stable-min-model} for the related topic).

We note that the isomorphisms in $\mathscr{M}_{d,v,u,w,r}$ are those in $\mathscr{M}_{d,v,u,r}$, but the converse is not necessarily true. 
The choice of $r$ in Lemma \ref{lem--Cartierindex} is not unique and the stacks $\mathscr{M}_{d,v,u,r}$ in Theorem \ref{quesmain} and $\mathscr{M}_{d,v,u,w,r}$ in Theorem \ref{modulimain} depend on the choice of $r$, however, their reduced structures are independent of $r$. For details, see Remark \ref{rem-samemoduli}. 

The second ingredient (Theorem \ref{mainbound} below) is the boundedness of $\mathscr{M}_{d,v,u,r}(\mathrm{Spec}\,\mathbb{C})$.
In fact, we prove the following much stronger assertion.
Let $d$ be a positive integer, $\Theta \subset \mathbb{Q}$ a DCC set, $v$ a positive rational number, and $u$ a rational number. 
We set
$$\mathfrak{D}_{d, \Theta, v,u}:=\left\{
\begin{array}{l}
f\colon (X,\Delta,A) \to C
\end{array}
\;\middle|
\begin{array}{rl}
(i)&\text{$f\colon (X,\Delta) \to C$ is a klt-trivial}\\
&\text{fibration over a curve $C$ such}\\
&\text{that $K_{X}+\Delta\equiv uf^{*}H$ with a}\\
&\text{line bundle $H$ of degree one,}\\
(ii)&\text{${\rm dim}X=d$,}\\
(iii)&\text{the coefficients of $\Delta$ belong to $\Theta$,}\\
(iv)&\text{$A$ is an $f$-ample $\mathbb{Q}$-Cartier Weil}\\
&\text{divisor such that $(H\cdot A^{d-1})=v$.}
\end{array}\right\}/\sim,$$
where $\sim$ means relative linear equivalence of $A$ over $C$. We denote the equivalence class by $[f\colon (X,\Delta,A) \to C]$. Consider also for any $w>0$,
$$\mathfrak{G}_{d, \Theta, v,u,w}:=\left\{
\begin{array}{l}
f\colon (X,\Delta,A) \to C
\end{array}
\;\middle|
\begin{array}{l}
\text{$f$ satisfies the conditions (i)--(iv) such that}\\
\text{$A$ is (globally) ample and $\mathrm{vol}(A)\le w$.}
\end{array}\right\}.$$

\begin{thm}[Boundedness]\label{mainbound}
Fix $d\in\mathbb{Z}_{>0}$, a DCC set $\Theta\subset \mathbb{Q}$, $v \in \mathbb{Q}_{>0}$, and $u\in\mathbb{Q}$. 
With notation as above, the following hold.
\begin{enumerate}
\item The set of klt pairs $(X,\Delta)$ appearing in $\mathfrak{D}_{d, \Theta, v,u}$ is log bounded, and\label{bound-(1)}
    \item there exists $w\in\mathbb{Q}_{>0}$, depending only on $d$, $\Theta$, $v$, and $u$, such that the natural map
    \[
    \mathfrak{G}_{d, \Theta, v,u,w}\longrightarrow\mathfrak{D}_{d, \Theta, v,u}
    \]
    is surjective.\label{bound-(2)}
\end{enumerate}
\end{thm}

After the paper has been completed, Birkar informed the authors that he and Hacon obtained Theorem \ref{mainbound} (1) independently. 
Theorem \ref{mainbound} not only shows the boundedness of $\mathfrak{Z}_{d, v,u}$ but also assures that $\mathscr{M}_{d,v,u,r}$ is of finite type in Theorem \ref{quesmain}. 
For the other statement of the boundedness, see Proposition \ref{bou6}. 

We also study special K-stability, which was introduced by the second author \cite{CM}. 
This is a stronger condition than the uniform K-stability. 
By \cite{CM}, there exists an explicit criterion of the special K-stability without using test configurations, and the CM minimization conjecture, a numerical and stronger assertion than (\ref{intro-(4)}), holds for the spacial K-stability.
We note that a uniformly adaiabatically K-stable klt-trivial fibration over a curve is specially K-stable for a certain polarization (\cite[Theorem 3.12]{CM}). 
We show that all members of $\mathfrak{G}_{d, \Theta, v,u,w}$ are simultaneously specially K-stable for a certain choice of polarizations as follows.

\begin{thm}[Uniformity of adiabatic K-stability]\label{mainunif}
Let $d \in \mathbb{Z}_{>0}$, $\Theta\subset \mathbb{Q}$, $u \in \mathbb{Q}$, and $v \in \mathbb{Q}_{>0}$ be as in Theorem \ref{mainbound} and $w$ be a positive rational number.
Then, there exists an $\epsilon_0\in\mathbb{Q}_{>0}$, depending only on $d$, $\Theta $, $u$, $w$, and $v$, such that $(X,\Delta,\epsilon A+f^*H)$ is specially K-stable for any rational number $\epsilon\in(0,\epsilon_0)$, line bundle $H$ on $C$ of $\mathrm{deg}\,H=1$, and $f\colon (X,\Delta,A) \to C\in\mathfrak{G}_{d, \Theta, v, u,w}$. 

Furthermore, there exists $\alpha>0$ such that 
\[
M^{\mathrm{NA}}_{\Delta}(\mathcal{X},\mathcal{M})\ge\alpha(\mathcal{J}^{\epsilon A+f^*H})^{\mathrm{NA}}(\mathcal{X},\mathcal{M})
\]
for any $f\colon (X,\Delta,A) \to C\in\mathfrak{G}_{d, \Theta, v, u,w}$ with a line bundle $H$ on $C$ of $\mathrm{deg}\,H=1$, normal semiample test configuration $(\mathcal{X},\mathcal{M})$ for $(X,\epsilon A+f^*H)$, and rational number $\epsilon\in(0,\epsilon_0)$.
\end{thm}

It is known by \cite{Z} that every specially K-stable smooth polarized manifold $(X,L)$ has a cscK metric in the first Chern class $\mathrm{c}_1(L)$.
 By Theorems \ref{mainbound} and \ref{mainunif}, we have the following corollary on the ``uniform" existence of cscK metrics.

 \begin{cor}\label{cor-main}
     Let $d \in \mathbb{Z}_{>0}$, $\Theta\subset \mathbb{Q}$, $u \in \mathbb{Q}$, and $v \in \mathbb{Q}_{>0}$ be as in Theorem \ref{mainbound}. 
     Then there exists a $w>0$, depending only on $d$, $\Theta$, $u$, and $v$, satisfying the following:
     For any representative $f\colon (X,\Delta,A)\to C$ of any element of $\mathfrak{D}_{d,\Theta,v,u}$ with a general fiber $F$ of $f$, if $\mathrm{vol}(A+tF)\ge w$ for some $t\in\mathbb{Q}$ then $(X,\Delta,A+tF)$ is specially K-stable.
     
     Furthermore, if $X$ is smooth and $\Delta=0$, then $X$ has a cscK metric $\omega$ in $\mathrm{c}_1(A+tF)$. 
 \end{cor}
 
Furthermore, Corollary \ref{cor-main} states that there is a ``universal" family $\mathscr{U}'$ over $\mathscr{M}_{d,v,u,r}$ in Theorem \ref{quesmain} with a polarization $\mathscr{A}_{\mathscr{U}'}$ whose geometric fibers are specially K-stable varieties. 
Here, the word ``universal'' comes from the construction (see Remark \ref{rem-final}).

 \subsection{Structure of this paper and overview of proof}\label{structureintro}
The contents of this paper are as follows.
 
 In \S\ref{Pre}, we collect notations and definitions in birational geometry, Hilbert schemes and stacks.
To discuss the $\mathbb{Q}$-Gorensteinness of families, we explain the universal hull of coherent sheaves introduced by Koll\'ar \cite{kollar-moduli}.
We also collect basic facts of K-stability and some results on J-stability and uniform adiabatic K-stability (\cite{Hat}), and we introduce a characterization of the uniform adiabatic K-stability of klt-trivial fibrations over curves (Definition \ref{unifdef}). We make use of this characterization to construct our moduli spaces.
 
In \S\ref{Bousec}, we prove Theorem \ref{mainbound}. 
The idea is as follows: 
With notations in Theorem \ref{mainbound}, we first give an upper bound $n$ of the Cartier indices of the log canonical divisors (see Lemma \ref{lem--Cartierindex}) and we reduce Theorem \ref{mainbound} to the case where $\Theta=\frac{1}{n}\mathbb{Z}\cap[0,1]$. 
We also know that $(X,\Delta)$ are $\frac{1}{n}$-lc.
By using the boundedness of singularities, we next find a lower bound of the $\alpha$-invariants of $A|_F$ with respect to $(F,\Delta|_F)$ for the general fibers $F$ of $f$ (Lemma \ref{lem--lct}). 
Since $(F,\Delta|_{F})$ are $\frac{1}{n}$-lc pairs polarized by $A|_{F}$, the existence of the lower bound is a consequence of Birkar's result \cite{birkar-bab}. 
By using this lower bound and the semi-positivity theorem by Fujino \cite[Theorem 1.11]{fujino-semi-positivity}, we find an $m \in \mathbb{Z}$ such that $A+mF$ is ample and ${\rm vol}(A+mF)$ is universally bounded from above (Proposition \ref{prop--nefthreshold-volume}, which is a special case of Theorem \ref{mainbound} (\ref{bound-(2)})). 
Here $m$ can be negative.  
From this result and \cite{birkar-geometry-moduli}, we obtain Theorem \ref{mainbound}. 
The boundedness problem (\ref{intro-(2)}) will be solved in this section. 
For the construction of our moduli, we also prove a result on the finiteness of the Hilbert polynomials (Corollary \ref{cor--hilbertpolynomial}). 

In \S\ref{toolsec}, we prove two important properties, i.e., the openness (Theorem \ref{op}) and the separatedness (Theorem \ref{sep2}) of uniformly adiabatically K-stable klt-trivial fibrations over curves. 
The openness is a direct consequence of the lower semi-continuity of the $\delta$-invariants of the log twisted bases, which will be proved in Theorem \ref{op}.
Note that we cannot apply \cite{BL} since the case of families of polarized log pairs was studied in their paper. 
The separatedness has already been known by the second author \cite{CM} when the varieties are over $\mathbb{C}$. 
In Theorem \ref{sep2}, we will give an alternative proof of the separatedness, which is an enhancement of \cite[Corollary 3.22]{CM} and works for any algebraically closed field of characteristic zero.
These two results directly imply (\ref{intro-(3)}) and (\ref{intro-(4)}) respectively.
Thus, we can obtain all the key conditions (\ref{intro-(2)})--(\ref{intro-(4)}) for the uniform adiabatic K-stability in \S\ref{Bousec} and \S\ref{toolsec}.
We also discuss the invariance of  (anti-)plurigenera used in the construction of our moduli spaces (see Theorem \ref{thm--inv-pluri}).

In \S\ref{Consec}, we prove Theorem \ref{modulimain}, in other words, we construct the moduli space by using tools proved in \S\ref{Bousec} and \S\ref{toolsec}.
We also show Theorem \ref{quesmain} by applying Proposition \ref{prop--nefthreshold-volume}.

In \S\ref{seckst}, we prove Theorem \ref{mainunif} and Corollary \ref{cor-main}. 
For this, we first show that there exist finitely many log $\mathbb{Q}$-Gorenstein families parametrizing polarized klt-trivial fibrations over curves (Proposition \ref{bou6}). 
Compared to \S\ref{Consec}, we deal with klt-trivial fibrations whose boundary divisors are not necessarily zero. 
However, these log $\mathbb{Q}$-Gorenstein families can be constructed by a similar argument to the proof of Theorem \ref{modulimain}. 
In the case of nef log canonical divisors, Theorem \ref{mainunif} follows from a simple observation of J-stability in \cite{Hat2} (Theorem \ref{unif+}).
For other case, we prove that the uniform ``convergence of the $\delta$-invairant" (cf.~\cite[Theorem D]{Hat}) holds for members of a family of klt-trivial fibrations (Proposition \ref{lem-delta-conv}). 
Theorem \ref{mainunif} follows from these results, and Corollary \ref{cor-main} follows from Theorems \ref{mainbound} and \ref{mainunif}. 

\subsection*{Acknowledgements}
This work was partially supported by JSPS KAKENHI 22K13887 for K.H. and JSPS KAKENHI 22J20059 (Grant-in-Aid for JSPS Fellows DC1) for M.H. 
The authors are grateful to Professors Makoto Enokizono and Yuji Odaka and Doctors Dai Imaike, Kentaro Inoue, and Kazuhiro Ito for fruitful discussions which improved Theorem \ref{modulimain}.
The authors thank Professors J\'{a}nos Koll\'{a}r and Shou Yoshikawa for answering questions on the book \cite{kollar-moduli}. 
The authors thank Professors Caucher Birkar, Yoshonori Hashimoto, and Gang Tian and Doctor Eiji Inoue for useful comments.

\section{Preliminaries}\label{Pre}

Throughout this paper, we work over an algebraically closed field $\mathbbm{k}$ of characteristic zero unless otherwise stated.

\subsection*{Notation and Convention}
We collect notations and conventions used in this paper.

\begin{enumerate}[(1)]
\item\label{notation-(1)}
A {\em scheme} means a locally Noetherian scheme over $\mathbbm{k}$. 
For a scheme $X$, we denote the induced reduced scheme by $X_{\mathrm{red}}$.
A {\em variety} means an integral separated scheme of finite type over $\mathbbm{k}$. 
A {\em curve} means a smooth variety of dimension one. 
 
A {\em geometric point} of $X$ is a morphism $\mathrm{Spec}\,\Omega\to X$ where $\Omega$ is an algebraically closed field. 
For a point $x \in X$, $\overline{x}$ denotes the geometric point of $X$ which maps the unique point of $\mathrm{Spec}\,\Omega$ to $x$. 
We simply denote it by $\overline{x}\in X$.

\item\label{notation-(2)}
For any scheme $S$ and positive integer $d$, we denote $\mathbb{P}^{d}_{\mathbbm{k}}\times_{\rm Spec\,\mathbbm{k}} S$ by $\mathbb{P}_{S}^{d}$. 
We simply write $\mathbb{P}^{d}$ if there is no risk of confusion, for example, $S={\rm Spec}\,\mathbbm{k}$ or $S$ is a geometric point of a scheme. 
Let $p\colon \mathbb{P}_{S}^{d} \to \mathbb{P}^{d}_{\mathbbm{k}}$ be the projection. 
For any $m \in \mathbb{Z}$, we often denote $p^{*}\mathcal{O}_{\mathbb{P}^{d}_{\mathbbm{k}}}(m)$ by $\mathcal{O}(m)$, and we sometimes think $\mathcal{O}(m)$ of a Cartier divisor on $\mathbb{P}_{S}^{d}$ if there is no risk of confusion. 

\item\label{notation-(3)}
A morphism $f \colon X \to Y$ of schemes is called a {\em contraction} if $f$ is projective and $f_{*}\mathcal{O}_{X}\cong \mathcal{O}_{Y}$. 
For a morphism $f \colon X \to Y$ of schemes and a (geometric) point $y \in Y$, the fiber of $f$ over $y$ is denoted by $X_{y}$. 
For a $\mathbb{Q}$-divisor $D$ on $X$, we denote the restriction of $D$ to $f^{-1}(y)$ by $D_{y}$ when it is well-defined, for example, $D$ is $\mathbb{Q}$-Cartier at every codimension one point of $f^{-1}(y)$ and ${\rm Supp}\,D \not\supset f^{-1}(y)$. 
We note that $D_y$ does not coincide with the scheme-theoretic fiber in general.

\item\label{notation-(4)}
Let $X$ be a smooth variety, $D$ an snc divisor on $X$, and $f \colon X \to Z$ a morphism to a scheme $Z$. 
We say that $(X,D)$ is {\em log smooth over $Z$} or $f \colon (X,D) \to Z$ is {\em log smooth} if $f$ is a smooth surjective morphism and for any stratum $T$ of $(X,D)$, the restriction $f|_{T} \colon T \to Z$ is also a smooth surjective morphism. 

\item\label{notation-(5)}
We say that a subset of $\mathbb{R}$ satisfies the {\em descending chain condition} ({\em DCC}, for short) if the subset does not contain any strictly decreasing infinite sequence. 
We say that a subset of $\mathbb{R}$ satisfies the {\em ascending chain condition} ({\em ACC}, for short) if the subset does not contain any strictly increasing infinite
sequence. 
A subset of $\mathbb{R}$ is called a {\em DCC set} (resp.~an {\em ACC set}) if the subset satisfies the DCC (resp.~ACC).

\item\label{notation-(6)}
Let $a$ be a real number. 
Then we define $\lceil a \rceil$ to be the unique integer satisfying $ \lceil a \rceil -1< a \leq \lceil a \rceil$.
Let $X$ be a normal variety and let $D$ be an $\mathbb{R}$-divisor on $X$. 
Let $D=\sum_{i} d_{i}D_{i}$ be the prime decomposition. 
Then we define 
$\lceil D \rceil:=\sum_{i} \lceil d_{i} \rceil D_{i}$. 
We say that $D$ is a {\em Weil divisor} if every coefficients of $D$ is an integer, in other words, $D=\lceil D \rceil$ holds.
We define the reduced divisor $D_{\mathrm{red}}$ of $D$ to be $\sum_iD_i$.

\item\label{notation-(7)}
Let $X$ be a normal variety.
For a line bundle (resp.~$\mathbb{Q}$-line bundle, $\mathbb{R}$-line bundle) $L$ on $X$, we often think $L$ of a Cartier (resp.~$\mathbb{Q}$-Cartier, $\mathbb{R}$-Cartier) divisor on $X$. 
When $L$ is a line bundle on $X$, we often denote $L^{\otimes m}\otimes \mathcal{O}_X(D)$ by $\mathcal{O}_X(mL+D)$ for every Weil divisor $D$ on $X$. 

\item\label{notation-(8)}
Let $f \colon X \to Y$ be a morphism of schemes. 
Let $L_{1}$ and $L_{2}$ be line bundles on $X$.
We say that $L_{1}$ and $L_{2}$ are {\em linearly equivalent over $Y$}, denoted by $L_{1}\sim_{Y}L_{2}$, if there is a line bundle $L$ on $Y$ such that $L_{1}\cong L_{2}\otimes f^{*}L$. 
When $Y$ is a point, we simply say that $L_{1}$ and $L_{2}$ are {\em linearly equivalent} and we denote it by $L_{1}\sim L_{2}$. 

Suppose that $X$ is a normal variety. 
Let $D_1$ and $D_2$ be $\mathbb{Q}$-Cartier $\mathbb{Q}$-divisors on $X$.
We say that $D_{1}$ and $D_{2}$ are {\em $\mathbb{Q}$-linearly equivalent over $Y$}, denoted by $D_{1}\sim_{\mathbb{Q},Y}D_{2}$, if there exists a positive integer $m$ such that both $mD_{1}$ and $mD_{2}$ are Cartier and $\mathcal{O}_{X}(mD_{1}) \sim_{Y} \mathcal{O}_{X}(mD_{2})$. 
This definition is not standard.
However, the definition coincides with the usual definition of the relative $\mathbb{Q}$-linear equivalence when $Y$ is a variety (e.g. $f$ is a contraction). 
When $Y$ is a point, we simply say that $D_{1}$ and $D_{2}$ are {\em $\mathbb{Q}$-linearly equivalent} and we denote it by $D_{1}\sim_{\mathbb{Q}}D_{2}$. 
 
\item\label{notation-(9)} 
Let $X$ be a projective scheme over $\mathbbm{k}$, $A$ a Cartier divisor of $X$, and $\phi_{|A|}$ a rational map $X\dashrightarrow \mathbb{P}^{h^0(X,\mathcal{O}_X(A))-1}$ defined by the linear system $|A|$. If $A$ is semiample, $\phi_{|mA|}$ induces a contraction for every sufficiently large and divisible $m>0$, and this is a kind of ample model defined in \cite[Lemma 3.6.5 (3)]{BCHM}.
 Similarly, for a projective morphism $\pi \colon \mathcal{X}\to S$ of schemes and a $\pi$-semiample line bundle $\mathscr{A}$ on $\mathcal{X}$, we call a morphism $f \colon \mathcal{X}\to\mathbf{Proj}_S(\bigoplus_{l\ge0}\pi_*\mathscr{A}^{\otimes l})$ the {\it ample model} of $\mathscr{A}$ over $S$. 

\item\label{notation-(10)} 
Let $f_1 \colon X_1\to Y_1$ and $f_2 \colon X_2\to Y_2$ be morphisms of schemes over a scheme $S$. Then  the induced morphism $X_1\times_SX_2\to Y_1\times_SY_2$ from $f_1$ and $f_2$ is denoted by $f_1\times_Sf_2$. 
When $S=\mathrm{Spec}\,\mathbbm{k}$, we simply write $f_1\times f_2 \colon X_1\times X_2\to Y_1\times Y_2$.

\item\label{notation-(11)}
For any morphisms $f \colon \mathcal{X}\to S$ and $h \colon T\to S$, we denote $\mathcal{X}\times_ST$ by $\mathcal{X}_T$ and the base change $\mathcal{X}_T\to T$ by $f_T$. 
For any coherent sheaf $\mathscr{A}$ on $\mathcal{X}$, we denote $(h\times_S\mathrm{id}_{\mathcal{X}})^*\mathscr{A}$ by $\mathscr{A}_T$. 
For an $f$-ample line bundle $H$ on $\mathcal{X}$ and a polynomial $p$, if $\chi(\mathscr{A}_s(tH_s))=p(t)$ for every $t\in\mathbb{Z}$, then we say that {\it $\mathscr{A}_s$ has the Hilbert polynomial $p$ with respect to $H$}. 
\item\label{defn-hor-vert}
     Let $f \colon Y\to C$ be a contraction from a normal variety to a curve and $D$ a $\mathbb{Q}$-divisor on $Y$.
     Then we can decompose $D$ into $D_{\mathrm{vert}}+D_{\mathrm{hor}}$ where the support of $D_{\mathrm{hor}}$ is flat over $C$ and the support of $D_{\mathrm{vert}}$ has the zero-dimensional image in $C$. 
\end{enumerate}

\begin{defn}
     Let $S$ be a Noetherian scheme and let $S_{1},\,\cdots,\,S_{l}$ be locally closed subschemes of $S$ that are disjoint in each other and $\sqcup_{1=1}^{l}S_{i} = S$ set-theoretically. 
     Then we call the natural inclusion $\sqcup_{1=1}^{l}S_{i}\to S$ a {\it locally closed decomposition}. A subset $F \subset S$ is called a {\em constructible subset} if $F$ is a finite union of locally closed subsets. 
  \end{defn}
  
\begin{lem}\label{const--lem}
Let $S$ be a scheme of finite type over $\mathbbm{k}$. Suppose that $F\subset S$ is a constructible subset. Then $F$ is closed if and only if the following holds.
\begin{itemize}
    \item For any morphism $\varphi \colon C\to S$ from an affine curve $C$, if $\varphi^{-1}(F)$ is dense, then $\varphi(C)\subset F$. 
\end{itemize}
\end{lem}
\begin{proof}
The assertion is local and we may assume that $S$ is affine. 
Suppose that the condition holds. Let $\overline{F}$ be the Zariski closure and assume that there exists a point $s\in \overline{F}\setminus F$.
Take an irreducible component $Z$ of $\overline{F}$ containing $s$. 
It is easy to see that $F$ contains a non-empty open subset of $Z$ (cf.~\cite[6.C]{Mat}). 
On the other hand, since $\overline{\{s\}}\cap F$ is not dense in $\overline{\{s\}}$, there is a closed point $s_0\in \overline{\{s\}}\setminus F$. 
By \cite[\S6, Lemma]{Ab}, there exists a morphism $\varphi \colon C\to Z$ from an affine curve such that $\varphi^{-1}(F)$ is dense and $s_0\in \varphi(C)$. 
Thus, $s_0\in F$ by the condition and this is a contradiction.
\end{proof}

 \subsection{Birational geometry}
 In this subsection, we collect definitions concerned with singularities of pairs, klt-trivial fibration, and boundedness.
 \begin{defn}[Singularities of pairs]
A {\em subpair} $(X,\Delta)$ consists of a proper normal variety $X$ and a $\mathbb{Q}$-divisor $\Delta$ on $X$ such that $K_X+\Delta$ is $\mathbb{Q}$-Cartier. A subpair is called a {\em pair} if the coefficients of $\Delta$ are positive. Let $F$ be a prime divisor over $X$ and take $\pi \colon Y\to X$ a proper birational morphism from a normal variety such that $F$ appears as a divisor on $Y$. Then we define the {\em log discrepancy} of $F$ with respect to $(X,\Delta)$ by
 $$A_{(X,\Delta)}(F):=1+\mathrm{ord}_F(K_{Y}-\pi^*(K_X+\Delta)),$$
where $\mathrm{ord}_F$ is the divisorial valuation associated to $F$ with $\mathrm{ord}_F(F)=1$.
It is easy to see that $A_{(X,\Delta)}(F)$ is independent  of $\pi$. A (sub)pair $(X,\Delta)$ is called ({\em sub}){\em klt} (resp.~({\em sub}){\em lc}, $\epsilon$-({\em sub}){\em lc}) if $A_{(X,\Delta)}(F)>0$ (resp. ~$\geq 0$, $\geq \epsilon$)  for every prime divisor $F$ over $X$. 
We say that a proper normal variety $V$ is a {\em klt variety} if $(V,0)$ is a klt pair. 

For an effective $\mathbb{Q}$-Cartier $\mathbb{Q}$-divisor $M$ on a normal variety $X$, the {\em log canonical threshold} of $M$ with respect to a subpair $(X,\Delta)$, denoted by $\mathrm{lct}(X,\Delta;M)$, is defined as follows: 
If there exists a $t$ such that $(X,B+tM)$ is sublc, then
 \[
 \mathrm{lct}(X,\Delta;M):=\sup\{t\in\mathbb{Q}\,|\,(X,\Delta+tM)\, \textrm{is sublc}\},
 \]
and otherwise we set $\mathrm{lct}(X,\Delta;M):=-\infty$.
\end{defn}

\begin{defn}[Iitaka volume]\label{defn--iitakavol}Let $X$ be a normal projective variety and let $D$ be a $\mathbb{Q}$-Cartier divisor on $X$ such that the Iitaka dimension $\kappa(X,D)$ is nonnegative. Then the {\em Iitaka volume} of $D$, denoted by ${\rm Ivol}(D)$, is defined by $${\rm Ivol}(D):=\underset{m\to \infty}{\rm lim\,sup}\frac{{\rm dim}H^{0}(X,\mathcal{O}_{X}(\lfloor mD \rfloor))}{m^{\kappa(X,D)} \slash \kappa(X,D)!}.$$\end{defn}  

When $D$ is big, the Iitaka volume of $D$ coincides with the usual volume. 
By definition, we can easily check that ${\rm Ivol}(rD)=r^{\kappa(X,D)}\cdot {\rm Ivol}(D)$ for every $r \in \mathbb{Z}_{>0}$.

\begin{defn}[Klt-trivial fibration]\label{defn--klttrivialfib}
Let $(X,\Delta)$ be a klt pair, and let $f \colon X \to C$ be a contraction of normal projective varieties. 
Then $f \colon (X,\Delta) \to C$ is called a {\it klt-trivial fibration} if $K_{X}+\Delta \sim_{\mathbb{Q},\,C}0$. 

For a klt-trivial fibration $f \colon (X,\Delta) \to C$, we define the {\em discriminant $\mathbb{Q}$-divisor $B_{C}$} and the {\em moduli $\mathbb{Q}$-divisor} $M_{C}$ on $C$ as follows: 
For every prime divisor $P$ on $C$, let $b_{P}$ be the largest real number such that after shrinking $C$ around the generic point $\eta$ of $P$, the pair $(X,\Delta+b_Pf^{*}P)$ is lc. 
Note that $b_{P}$ is well-defined since $P$ is Cartier at $\eta$. 
Then we define the discriminant $\mathbb{Q}$-divisor $B_{C}$ by 
\[
B_{C}:=\sum_{P}(1-b_{P})P,
\]
where $P$ runs over prime divisors on $C$. 
Next, we fix a $\mathbb{Q}$-Cartier $\mathbb{Q}$-divisor $L$ on $C$ such that $K_{X}+\Delta \sim_{\mathbb{Q}}f^{*}L$. 
Then the moduli $\mathbb{Q}$-divisor $M_{C}$ is defined by 
$$M_C:=L-(K_C+B_C).$$ 
Note that $M_{C}$ is only defined up to $\mathbb{Q}$-linear equivalence class. 
We call 
\[
K_X+\Delta\sim_{\mathbb{Q}}f^*(K_C+B_C+M_C)
\]
the {\it canonical bundle formula}. 
\end{defn}

In general, we can define klt-trivial fibrations for subpairs and contractions (cf.~\cite{A}). 
However, for simplicity we always assume that $(X,\Delta)$ in klt-trivial fibrations are klt pairs.

We make use of the following fundamental fact.
\begin{thm}[{\cite[Theorem 0.1]{A}}]\label{Bsemi}
If ${\rm dim}\,C=1$, then $M_C$ is a semiample $\mathbb{Q}$-Cartier $\mathbb{Q}$-divisor.
\end{thm}

\begin{defn}[Discriminant $\mathbb{Q}$-divisor with respect to contraction]\label{defn--discriminant-morphism}
By extending the notion of the discriminant $\mathbb{Q}$-divisors in Definition \ref{defn--klttrivialfib}, for every sublc pair $(X,\Delta)$ with a contraction $f \colon X \to Z$ of normal varieties, we define the {\em discriminant $\mathbb{Q}$-divisor with respect to $f \colon (X,\Delta) \to Z$} as follows: For each prime divisor $P$ on $Z$, we define
$$\mu_{P}:={\rm sup}\left\{\gamma \in \mathbb{R}\,\middle|\, \text{$(X, \Delta+\gamma f^{*}P)$ is sublc over the generic point of $P$}\right\}.$$
Note that we may assume that $f^{*}P$ is well-defined since we may shrink $Z$ around the generic point of $P$.
We define
$$B:=\sum_{P}(1-\mu_{P})P,$$
where $P$ runs over prime divisors on $Z$. 

It is easy to see that this definition coincides with the discriminant $\mathbb{Q}$-divisor in Definition \ref{defn--klttrivialfib} when $f \colon (X,\Delta) \to Z$ is a klt-trivial fibration.
\end{defn}

\begin{defn}[Boundedness]
We say a set $\mathfrak{Q}$ of normal projective varieties is {\em bounded} if there exist finitely many projective morphisms $V_{i} \to T_{i}$ of varieties such that for each $X \in \mathfrak{Q}$ there exist an index $i$, a closed point $t \in T_{i}$, and an isomorphism $\phi \colon (V_i)_{t} \to X$.

A {\em couple} $(X, S)$ consists of a normal projective variety $X$ and a reduced divisor $S$ on $X$. 
We use the term couple because $K_X + S$ is not assumed to be $\mathbb{Q}$-Cartier. 
We say that a set $\mathfrak{P}$ of couples is {\em bounded} if there exist finitely many projective morphisms $V_{i} \to T_{i}$ of varieties and a reduced divisor $C$ on each $V_{i}$ such that for any $(X, S) \in \mathfrak{P}$ there exist an index $i$, a closed point $t \in T_{i}$, and an isomorphism $\phi \colon (V_i)_{t} \to X$ such that $((V_i)_{t},C_{t})$ is a couple and $C_{t}=\phi^{-1}_{*}S$. 

Finally, we say that a set $\mathfrak{R}$ of projective pairs $(X,\Delta)$ is {\em log bounded} if the set of the corresponding couples $(X, {\rm Supp}\,\Delta)$ is bounded. 
 \end{defn}

 \subsection{Hilbert schemes}\label{subsec-Hilb}
Let $f \colon \mathcal{X}\to S$ be a proper morphism of schemes and $\mathscr{A}$ an $f$-ample line bundle on $\mathcal{X}$. 
Then $\mathrm{Hilb}_{\mathcal{X}/S}^{p,\mathscr{A}}$ denotes the scheme representing the following functor $\mathfrak{Hilb}_{\mathcal{X}/S}^{p,\mathscr{A}}$.
For any morphism $T\to S$, we attain
 $$\mathfrak{Hilb}_{\mathcal{X}/S}^{p,\mathscr{A}}(T):=\left\{
\begin{array}{l}
\mathcal{Z}\subset \mathcal{X}_T
\end{array}
\;\middle|
\begin{array}{rl}
&\text{$\mathcal{Z}$ is a closed subscheme of $\mathcal{X}_T$ flat over $T$ whose}\\
&\text{fibers have the same Hilbert polynomial $p$ with}\\
&\text{respect to $\mathscr{A}$.}
\end{array}\right\}.$$
We remark that $\mathrm{Hilb}_{\mathcal{X}/S}^{p,\mathscr{A}}$ exists as a locally projective scheme over $S$. 
Indeed, it is well-known that if $\mathscr{A}$ is further $f$-very ample, then $\mathrm{Hilb}_{\mathcal{X}/S}^{p,\mathscr{A}}$ is projective over $S$ entirely (cf.~\cite[\S5]{FGA}). Therefore, for any quasi-compact open subset $U\subset S$, by taking $m>0$ such that $\mathscr{A}^{\otimes m}|_{\mathcal{X}_U}$ is $f_U$-very ample, we see that $\mathrm{Hilb}_{\mathcal{X}_U/U}^{p,\mathscr{A}|_{\mathcal{X}_U}}=\mathrm{Hilb}_{\mathcal{X}_U/U}^{q,\mathscr{A}^{\otimes m}|_{\mathcal{X}_U}}$ exists as a projective scheme over $S$, where $q(n)=p(mn)$ for any $n\in\mathbb{Z}$. By patching $\mathrm{Hilb}_{\mathcal{X}_U/U}^{p,\mathscr{A}|_{\mathcal{X}_U}}$ together over $S$, we obtain a unique locally projective scheme $\mathrm{Hilb}_{\mathcal{X}/S}^{p,\mathscr{A}}$ over $S$ up to isomorphism. 
In this paper, we call $\mathrm{Hilb}_{\mathcal{X}/S}^{p,\mathscr{A}}$ the {\em Hilbert scheme}. When $S=\mathrm{Spec}\,\mathbbm{k}$, we simply write $\mathrm{Hilb}_{\mathcal{X}}^{p,\mathscr{A}}$.
We denote $\sqcup_{p}\mathrm{Hilb}_{\mathcal{X}/S}^{p,\mathscr{A}}$ by $\mathrm{Hilb}_{\mathcal{X}/S}$, where $p$ runs over polynomials. 

Next, we assume that $f$ is flat and it has geometrically connected and normal fibers. 
Let $g \colon \mathcal{Y}\to S$ be another proper morphism of schemes and $\mathscr{B}$ a $g$-ample line bundle on $\mathcal{Y}$ such that the fibers of $g$ have the Hilbert polynomial $p$ with respect to $\mathscr{B}$. 
We set
\begin{align*}
\mathfrak{Isom}_S(\mathcal{X},\mathcal{Y})(T)&=\{\text{$h \colon \mathcal{X}_T\to\mathcal{Y}_T$ is a $T$-isomorphism}\},\quad\text{and}\\
\mathfrak{Isom}_S((\mathcal{X},\mathscr{A}),(\mathcal{Y},\mathscr{B}))(T)&=\{h\in\mathfrak{Isom}_S(\mathcal{X},\mathcal{Y})(T)|\,\text{$h^*\mathscr{B}_T\sim_T\mathscr{A}_T$}\}
\end{align*}
for any $S$-scheme $T$. 
By \cite[Theorem 5.23]{FGA} and Corollary \ref{cor--hako}, which we will treat later, the functor $\mathfrak{Isom}_S(\mathcal{X},\mathcal{Y})$ (resp.~$\mathfrak{Isom}_S((\mathcal{X},\mathscr{A}),(\mathcal{Y},\mathscr{B}))$) is represented by a locally closed subscheme $\mathrm{Isom}_S(\mathcal{X},\mathcal{Y})$ (resp.~$\mathrm{Isom}_S((\mathcal{X},\mathscr{A}),(\mathcal{Y},\mathscr{B}))$) of $\mathrm{Hilb}_{\mathcal{X}\times_S\mathcal{Y}/S}$ (resp.~$\mathrm{Hilb}^{q,\,p_1^*\mathscr{A}\otimes p_2^*\mathscr{B}}_{\mathcal{X}\times_S\mathcal{Y}/S}$), where $q$ is the polynomial defined by $q(m)=p(2m)$ and $p_1\colon\mathcal{X}\times_S\mathcal{Y}\to\mathcal{X}$ (resp.~$p_2\colon\mathcal{X}\times_S\mathcal{Y}\to\mathcal{Y}$) is the first (resp.~second) projection. 
Thus, we see that $\mathrm{Isom}_S((\mathcal{X},\mathscr{A}),(\mathcal{Y},\mathscr{B}))$ is locally quasi-projective over $S$. If $(\mathcal{X},\mathscr{A})=(\mathcal{Y},\mathscr{B})$, we set 
$$\mathrm{Aut}_S(\mathcal{X},\mathscr{A}):=\mathrm{Isom}_S((\mathcal{X},\mathscr{A}),(\mathcal{X},\mathscr{A})).$$ 
For details, we refer to \cite[\S5.6]{FGA}. 
We note that if we define
\[
\mathfrak{Aut}(\mathbb{P}^n_\mathbb{Z})(T):=\{T\text{-automorphisms of $\mathbb{P}^n_T$}\}
\]
for any locally Noetherian scheme $T$, then it is well-known that the functor $\mathfrak{Aut}(\mathbb{P}^n_\mathbb{Z})$ is represented by $PGL(n+1,\mathbb{Z})$ (cf.~\cite[\S0.5 (b)]{GIT}).

On the other hand, we define a presheaf $\mathfrak{Pic}_{\mathcal{X}/S}$ as follows:  For any morphism $T\to S$, we attain the following set
\[
\mathfrak{Pic}_{\mathcal{X}/S}(T)=\{\text{$L$: a line bundle on $\mathcal{X}_T$}\}/\sim_T.
\]
In other words, $\mathfrak{Pic}_{\mathcal{X}/S}(T)$ is the set of all relative linear equivalence classes of line bundles on $\mathcal{X}_T$ over $T$.
In general, $\mathfrak{Pic}_{\mathcal{X}/S}$ is not an \'{e}tale sheaf.
However, it is well-known (cf.~\cite[\S0.5]{GIT}, \cite[\S9]{FGA}) that under the same assumption on $\mathcal{X}\to S$ as the previous paragraph, there exists a separated scheme $\mathrm{Pic}_{\mathcal{X}/S}$ locally of finite type over $S$ such that there exist the following maps for all $T\to S$,
\[
\mathfrak{Pic}_{\mathcal{X}/S}(T)\hookrightarrow \mathrm{Hom}_S(T,\mathrm{Pic}_{\mathcal{X}/S})
\]
that is injective and they induce the \'{e}tale sheafification $\mathfrak{Pic}_{\mathcal{X}/S}\to \mathrm{Hom}_S(\bullet,\mathrm{Pic}_{\mathcal{X}/S})$ of $\mathfrak{Pic}_{\mathcal{X}/S}$.
Moreover, if $\mathcal{X}\to S$ has a section, then $\mathfrak{Pic}_{\mathcal{X}/S}$ coincides with $\mathrm{Hom}_S(\bullet,\mathrm{Pic}_{\mathcal{X}/S})$. 

  \subsection{Stacks}
 We refer to \cite[\S3, \S5]{Ols} and \cite{stacksproject-chap98} for the notations of fibered categories and algebraic spaces.
 Let $\mathbf{Sets}$ be the category of sets and $\mathbf{Sch}_{/S}$ the category of (locally Noetherian) schemes over $S$.
 If $S=\mathrm{Spec}\,\mathbbm{k}$, we denote $\mathbf{Sch}_{/\mathbbm{k}}$.
 For any scheme $S$, we endow $\mathbf{Sch}_{/S}$ with the \'{e}tale topology. 
 We recall the definition of stacks (see \cite[Proposition 4.6.2]{Ols}).
 \begin{defn}[Stacks]\label{defn--stacks}
     Let $p\colon\mathscr{C}\to\mathbf{Sch}_{/S}$ be a category fibered in groupoids.
     $\mathscr{C}$ is called a {\em stack} over $S$ if the following two conditions hold (cf.~\cite[Definition 4.6.1]{Ols}).
     \begin{enumerate}
             \item \label{defn--stack-(1)}
         For any $S$-scheme $X$ and any two objects $x,y\in\mathscr{C}(X):= p^{-1}(X)$, the presheaf $\mathfrak{Isom}_X(x,y)$, defined by $$\mathfrak{Isom}_X(x,y)(f\colon Y\to X):=\mathrm{Isom}_Y(f^*x,f^*y),$$ 
                   where the right hand side is the set of all isomorphisms $g$ such that $p(g)=\mathrm{id}_{Y}$, is an \'{e}tale sheaf, and
         \item \label{defn--stack-(2)} 
         for any \'{e}tale covering in $\mathbf{Sch}_{/S}$, all descent data with respect to the covering are effective (cf.~\cite[Definition 4.2.6]{Ols}). 
     \end{enumerate}
     \begin{rem}\label{rem-descent}
In the situation of Definition \ref{defn--stacks}, we consider the following condition.
     \begin{itemize}
             \item 
             For any set of $S$-schemes $\{X_i\}_{i\in I}$, the natural functor 
             $$\mathscr{C}\left(\bigsqcup_{i\in I}X_i\right)\to \prod_{i\in I}\mathscr{C}(X_i)$$ 
             is an equivalence of categories. 
         \end{itemize}
     \end{rem}
     We note that all stacks we treat in this paper satisfy this condition.
     Here, we explain the definition of descent data when $\mathscr{C}$ satisfies the above condition.
     We say that a surjective morphism $f\colon X'\to X$ is an {\it \'{e}tale covering} if $X':=\sqcup_{i\in I}X_i$ and $f|_{X_i}\colon X_i\to X$ is \'{e}tale for any $i\in I$.
     For any \'{e}tale covering $f$, let
     \[
     p_1,p_2\colon X'\times_XX'\to X'\text{ and }p_{12},p_{23},p_{13}\colon X'\times_XX'\times_XX'\to X'\times_XX'
     \]
     be the projections.
     A pair $(u'\in\mathscr{C}(X'),\sigma)$ is called a {\em descent datum} with respect to $f\colon X'\to X$ if $\sigma\in\mathrm{Isom}_{X'\times_XX'}(p_1^*u',p_2^*u')$ such that $p_{23}^*\sigma\circ p_{12}^*\sigma=p_{13}^*\sigma$.
     Note that for any $u\in\mathscr{C}(X)$, there exists a canonical isomorphism $\sigma_{\mathrm{can}}\in\mathrm{Isom}_{X'\times_XX'}(p_1^*f^*u,p_2^*f^*u)$ such that $(f^*u,\sigma_{\mathrm{can}})$ is a descent datum.
     If there is $u\in\mathscr{C}(X)$ such that  $\sigma\circ p_1^*g=p_2^*g\circ\sigma_{\mathrm{can}}$ for some $g\in\mathrm{Isom}(f^*u,u')$, then we call $(u',\sigma)$ an {\em effective descent datum}.
     We see that our definition and the original definition (\cite[Definition 4.2.6]{Ols}) of descent data are the same in this situation by \cite[Lemma 4.2.7]{Ols}.
 \end{defn}

 \begin{defn}[Artin stacks, Deligne-Mumford stacks]
     Let $\mathscr{C}$ be a stack over $\mathbbm{k}$. 
      $\mathscr{C}$ is called a {\it Deligne-Mumford (resp.~Artin) stack} if the following hold.
      \begin{enumerate}[(i)]
          \item The diagonal $\Delta\colon\mathscr{C}\to\mathscr{C}\times\mathscr{C}$ is representable, i.e.~ for any morphism $U\to \mathscr{C}\times\mathscr{C}$ from a scheme, $U\times_{ \mathscr{C}\times\mathscr{C}}\mathscr{C}$ is an algebraic space.
          \item There exists an \'{e}tale (resp.~smooth) surjectve morphism $\pi \colon X\to\mathscr{C}$ from a scheme.
      \end{enumerate}
      If $\mathscr{C}$ is a Noetherian Artin stack, we can define coherent sheaves on $\mathscr{C}$ in the way of \cite[9.1.14]{Ols}.
      If $\mathscr{L}$ is a coherent sheaf on $\mathscr{C}$ and there exists a smooth surjection $g\colon T\to\mathscr{C}$ such that $g^*\mathscr{L}$ is a line bundle, we say that $\mathscr{L}$ is a {\it line bundle} on $\mathscr{C}$ (see also \cite[\S9.3]{Ols}).
 \end{defn}

 \begin{ex}
     It is well-known that $\mathbf{Sch}_{/S}$ has the natural stack structure over $\mathbbm{k}$ for any scheme $S$ (see \cite[p.~97]{deligne-mumford}).
     We simply denote this stack by $S$. 
     For any scheme $T$, we know that $\mathbf{Sch}_{/S}(T)=\mathrm{Hom}(T,S)$.
     We denote this by $S(T)$.
 \end{ex}
 
  \begin{ex}\label{ex--quotient-stack}
     Let $X$ be a scheme of finite type over $\mathbbm{k}$ and $G$ be a  linear algebraic group over $\mathbbm{k}$.
     Then there exists a quotient stack $[X/G]$ defined as \cite[Example 8.1.12]{Ols}. 
     We remark that $[X/G]$ is an Artin stack of finite type over $\mathbbm{k}$ (cf.~\cite[Tag 036O]{stacksproject-chap98}). 
     Note that $X$ is quasi-compact.
 Moreover, for any $G$-equivariant line bundle $L$ on $X$, we can find a line bundle $\mathscr{L}$ on $[X/G]$ such that $\pi^*\mathscr{L}=L$, where $\pi\colon X\to [X/G]$ is the canonical projection (cf.~\cite[Exercise 9.H]{Ols}). 
 This $\mathscr{L}$ is unique up to isomorphism. 
     \end{ex}

The following category will be used in \S\ref{Consec}.
 
\begin{defn}\label{defn--pol}
Let $\mathfrak{Pol}$ be the category such that the collection of objects is
\[
\left\{
\begin{array}{l}
f\colon(\mathcal{X},\mathscr{A})\to S
\end{array}
\;\middle|
\begin{array}{rl}
&\text{$f$ is a surjective proper flat morphism of schemes whose}\\
&\text{geometric fibers are normal and connected, and}\\
&\text{$\mathscr{A}\in \mathrm{Pic}_{\mathcal{X}/S}(S)$ such that there exists an \'{e}tale covering}\\
&\text{$S'\to S$ by which the pullback of $\mathscr{A}$ to $\mathcal{X}\times_SS'$ is}\\
&\text{represented by a relatively ample line bundle over $S'$.}
\end{array}\right\},
\]
and an arrow $(g,\alpha)\colon(f\colon(\mathcal{X},\mathscr{A})\to S)\to (f'\colon(\mathcal{X}',\mathscr{A}')\to S')$ is defined in the way that $\alpha\colon S\to S'$ is a morphism and $g\colon\mathcal{X}\to \mathcal{X}'\times_{S'}S$ is an isomorphism such that $g^*\alpha_{\mathcal{X}'}^*\mathscr{A}'=\mathscr{A}$ as elements of $\mathrm{Pic}_{\mathcal{X}'/S'}(S)$. 
\end{defn}
It is easy to see that there exists a natural functor $p \colon \mathfrak{Pol}\to\mathbf{Sch}_{/\mathbbm{k}}$ such that $\mathfrak{Pol}$ is a category fibered in groupoids. 
 We further show the following.

 \begin{lem}\label{lem--descent}
$\mathfrak{Pol}$ is a stack over $\mathbbm{k}$.
\end{lem}

\begin{proof}
It suffices to check the conditions (\ref{defn--stack-(1)}) and (\ref{defn--stack-(2)}) of Definition \ref{defn--stacks} for $\mathfrak{Pol}$.
We first treat (\ref{defn--stack-(1)}). 
Take objects $f\colon(\mathcal{X},\mathscr{A})\to S$ and $f'\colon(\mathcal{X}',\mathscr{A}')\to S$.
Then $\mathfrak{Isom}_S(\mathcal{X},\mathcal{X}')$ is represented by a locally closed subscheme of $\mathrm{Hilb}_{\mathcal{X}\times_S\mathcal{X}'/S}$ (\cite[\S5.6]{FGA}).
Since $\mathrm{Pic}_{\mathcal{X}/S}$ is separated, we see that $\mathfrak{Isom}_S(f,f')\hookrightarrow\mathfrak{Isom}_S(\mathcal{X},\mathcal{X}')$ is a closed immersion. 
Therefore $\mathfrak{Isom}_S(f,f')$ is represented by a scheme. In particular, it is an \'{e}tale sheaf. 
Hence, (\ref{defn--stack-(1)}) holds.

Next, we treat (\ref{defn--stack-(2)}).
One can check that for any set of schemes $\{X_i\}_{i\in I}$, the natural functor 
\[
\mathfrak{Pol}\left(\bigsqcup_{i\in I} X_i\right)\to\prod_{i\in I} \mathfrak{Pol}(X_i)
\]
is an equivalence of categories.
By Remark \ref{rem-descent}, it suffices to show the following:
For any \'{e}tale covering $S'\to S$ with the projections
\[
    p_{1}, \, p_{2}\colon S'\times_SS'\to S' \quad \text{and}\quad p_{12}, \, p_{23}, \, p_{13}\colon S'\times_SS'\times_SS'\to S'\times_SS',
    \]
any descent datum $(f'\colon(\mathcal{X}',\mathscr{A}')\to S',\sigma)$ is effective.  
Here, $\sigma\in\mathrm{Isom}_{S'\times_SS'}(p_1^*f',p_2^*f')$.
If the pullback of $(f',\sigma)$ by an \'{e}tale covering $T \to S$ is effective, then so is $(f',\sigma)$ by the condition (\ref{defn--stack-(1)}). From this fact, by replacing $S'$ with a scheme $T$ admitting an \'{e}tale covering $T\to S'$, we may assume that $\mathscr{A}'$ is an $f'$-ample line bundle

By the $f'$-ampleness, there exists $m\in\mathbb{Z}_{>0}$ such that $H^i(\mathcal{X}'_s,\mathscr{A}'^{\otimes m}_s)=0$ for every $s\in S'$ and $i>0$ and the natural morphism $\mathcal{X}'\to\mathbb{P}_{S'}(f'_{*}\mathscr{A}'^{\otimes m})$ is a closed immersion.
We note that for any flat morphism $g\colon T\to S$, we have the natural isomorphism
    \[
    f'_{T*}g_{\mathcal{X}}^*\mathscr{A}'^{\otimes m}\cong g^*f'_*\mathscr{A}'^{\otimes m}
    \]
by \cite[III, Proposition 9.3]{Ha}. 
Thus, we may identify $f'_{T*}g_{\mathcal{X}}^*\mathscr{A}'^{\otimes m}$ with $g^*f'_*\mathscr{A}'^{\otimes m}$.
    Furthermore, $f'_*\mathscr{A}'^{\otimes m}$ is locally free by \cite[III, Theorem 12.11]{Ha}.
    On the other hand, there exist a line bundle $\mathcal{M}$ and an isomorphism 
    \[
    h\colon\sigma^{*}p_{2,\mathcal{X}'}^{*}\mathscr{A}'\cong p_{1,\mathcal{X}'}^{*}\mathscr{A}'\otimes(p_{1}^{*}f')^{*}\mathcal{M},
    \]
    where $p_{1,\mathcal{X}'}\colon\mathcal{X}'\times_SS'\to \mathcal{X}'$ (resp.~$p_{2,\mathcal{X}'}\colon\mathcal{X}'\times_SS'\to \mathcal{X}'$) is the morphism induced from base change of $p_{1}$ (resp.~$p_{2}$) by the canonical morphism $\mathcal{X}' \to S'$, and $p_{1}^*f'$ (resp.~$p_{2}^*f'$) is the base change of $f'$ by $p_1$ (resp.~$p_2$).
    Then 
    $h$ induces the following isomorphism 
    \[
    \varphi\colon\mathbb{P}_{S'\times_SS'}(p_1^*f'_*\mathscr{A}'^{\otimes m})=\mathbb{P}_{S'\times_SS'}(p_1^*f'_*\mathscr{A}'^{\otimes m}\otimes \mathcal{M}^{\otimes m})\overset{\cong}{\longrightarrow}\mathbb{P}_{S'\times_SS'}(p_{2}^*f'_*\mathscr{A}'^{\otimes m}).
    \]
    Here, the first equality is via the canonical isomorphism in \cite[II, Lemma 7.9]{Ha}. 
    
    The following easy claim implies that $\varphi$ is independent of the choice of $h$.
    \begin{claim*}
        Let $p \colon \mathcal{Y}\to T$ be a proper flat surjective morphism whose geometric fibers are connected and normal, and let $\mathscr{L}$ be a line bundle on $\mathcal{X}$.
        Suppose that $p_*\mathscr{L}$ is locally free.
        For any isomorphism $\tau \colon\mathscr{L}\to\mathscr{L}$, 
        $\tau$ induces the identity morphism of $\mathbb{P}_{T}(p_*\mathscr{L})$.
    \end{claim*}
    \begin{proof}[Proof of Claim]
    Now $\tau\in\mathrm{Hom}(\mathscr{L},\mathscr{L})$, and $\mathrm{Hom}(\mathscr{L},\mathscr{L})\cong H^0(\mathcal{Y},\mathcal{O}_{\mathcal{Y}})\cong H^0(T,\mathcal{O}_T)$ by \cite[9.3.11]{FGA}.
        Since $\tau$ is an isomorphism, we may regard $\tau$ as an element of $H^0(T,\mathcal{O}_T^*)$.
        Then the argument in the proof of \cite[II, Lemma 7.9]{Ha} works without any change.
    \end{proof}
    We continue to prove Lemma \ref{lem--descent}.
We will show that $\varphi$ defines a descent datum, i.e.~$p_{23}^*\varphi\circ p_{12}^*\varphi=p_{13}^*\varphi$. 
Clearly, this is equivalent to that $(p_{13}^*\varphi)^{-1}\circ p_{23}^*\varphi\circ p_{12}^*\varphi$ is the identity. 
We note that if the base change of the morphism $\sigma$ in the decent datum $(f',\sigma)$ by $p_{12}$ (resp.~$p_{23}$, $p_{13}$) is denoted by $p_{12}^*\sigma$ (resp.~$p_{23}^*\sigma$, $p_{13}^*\sigma$), then $p_{23}^*\sigma\circ p_{12}^*\sigma=p_{13}^*\sigma$ holds. 
This follows from the definition of descent data. 
We also note that the relative linear equivalence $p_{ij}^*h$ induces $p_{ij}^*\varphi$ for any $1\le i<j\le 3$. 
Thus, $(p_{13}^*\varphi)^{-1}\circ p_{23}^*\varphi\circ p_{12}^*\varphi$ is induced from the linear equivalence over $S'\times_SS'\times_SS'$
\begin{align*}
p_{12,\mathcal{X}'\times_SS'}^*p_{1,\mathcal{X}'}^*\mathscr{A}'&=(p_{13}^*\sigma)^*(( p_{23}^*\sigma)^{-1})^*(( p_{12}^*\sigma)^{-1})^*p_{12,\mathcal{X}'\times_SS'}^*p_{1,\mathcal{X}'}^*\mathscr{A}'\\
&\sim_{S'\times_SS'\times_SS'}(p_{13}^*\sigma)^*(( p_{23}^*\sigma)^{-1})^*p_{12,S'\times_S\mathcal{X}'}^*p_{2,\mathcal{X}'}^*\mathscr{A}'\\
&=(p_{13}^*\sigma)^*(( p_{23}^*\sigma)^{-1})^*p_{23,\mathcal{X}'\times_SS'}^*p_{1,\mathcal{X}'}^*\mathscr{A}'\\
&\sim_{S'\times_SS'\times_SS'}(p_{13}^*\sigma)^*p_{13,S'\times_S\mathcal{X}'}^*p_{2,\mathcal{X}'}^*\mathscr{A}'\\
&\sim_{S'\times_SS'\times_SS'}p_{12,\mathcal{X}'\times_SS'}^*p_{1,\mathcal{X}'}^*\mathscr{A}',
\end{align*}
where $p_{12,\mathcal{X}'\times_SS'}\colon\mathcal{X}'\times_SS'\times_SS'\to \mathcal{X}'\times_SS'$ is the base change of $p_{12}$ by the canonical morphism $\mathcal{X}'\times_SS'\to S'\times_SS'$, and $p_{23,\mathcal{X}'\times_SS'}\colon\mathcal{X}'\times_SS'\times_SS'\to \mathcal{X}'\times_SS'$ and $p_{13,\mathcal{X}'\times_SS'}\colon\mathcal{X}'\times_SS'\times_SS'\to \mathcal{X}'\times_SS'$ are defined similarly. 
By Claim, it immediately follows that $(p_{13}^*\varphi)^{-1}\circ p_{23}^*\varphi\circ p_{12}^*\varphi$ is the identity morphism. Thus $p_{23}^*\varphi\circ p_{12}^*\varphi=p_{13}^*\varphi$. 

    On the other hand, $-K_{\mathbb{P}_{S'}( f'_*\mathscr{A}'^{\otimes m})/S'}$ is relatively ample over $S$. 
    Hence, applying \cite[Proposition 4.4.12]{Ols} to $\mathbb{P}_{S'}(f'_*\mathscr{A}'^{\otimes m})$ and $-K_{\mathbb{P}_{S'} (f'_*\mathscr{A}'^{\otimes m})/S'}$, we may find a scheme $\mathcal{P}$ and a projective flat surjective morphism $\mathcal{P}\to S$ that canonically defines a descent datum isomorphic to $(\mathbb{P}_{S'}(f'_{*}\mathscr{A}'^{\otimes m}),\varphi)$.
Note that $\mathcal{P}$ is not a projective bundle but every geometric fiber over $S$ is a projective space.
    By applying \cite[Proposition 4.4.3]{Ols} to the closed immersion $\mathcal{X}'\hookrightarrow\mathbb{P}_{S'} (f'_*\mathscr{A}'^{\otimes m})$, we obtain a closed immersion $\mathcal{X}\hookrightarrow \mathcal{P}$ whose base change by $S'\to S$ coincides with $\mathcal{X}'\hookrightarrow\mathbb{P}_{S'} (f'_*\mathscr{A}'^{\otimes m})$.
On the other hand, by the definition of the Picard scheme, there exists a unique element $\mathscr{A}\in\mathrm{Pic}_{\mathcal{X}/S}(S)$ such that the pullback of $\mathscr{A}$ to $\mathcal{X}\times_SS'$ coincides with $\mathscr{A}'$. 

From the above facts, $(f',\sigma)$ is effective.
We finish the proof of Lemma \ref{lem--descent}.
\end{proof}

 \begin{rem}
 Let $f\colon\mathcal{X}\to S$ be a proper surjective flat morphism of schemes whose geometric fibers are normal and connected. We fix $\mathscr{A} \in \mathrm{Pic}_{\mathcal{X}/S}$.
    Then $(\mathcal{X},\mathscr{A}) \to S$ is an object of $\mathfrak{Pol}$ if and only if $\mathscr{A}_{\bar{s}}$ is ample for any geometric point $\bar{s}\in S$.
     Indeed, we may replace $f$ by $\mathcal{X}\times_{S}S' \to S'$ for some \'{e}tale covering $S' \to S$, thus we may assume that $\mathscr{A}$ is a line bundle. 
     Then $\mathscr{A}$ is $f$-ample if and only if $\mathscr{A}_{\bar{s}}$ is ample for any geometric point $\bar{s}\in S$ (cf.~\cite[Proposition 1.41]{KM}). 
     The converse is easy.
 \end{rem}
 
The following theorem is well-known to experts and holds since we assume that $\mathrm{char}(\mathbbm{k})=0$ (cf.~{\cite[\S11, Theorem]{Ab}}).

 \begin{thm}[{\cite[Remark 8.3.4]{Ols}}, \cite{KeM}] \label{dm}
 Let $\mathscr{C}$ be an Artin stack of finite type over $\mathbbm{k}$. If the diagonal morphism $\Delta \colon \mathscr{C}\to\mathscr{C}\times\mathscr{C}$ is finite, then $\mathscr{C}$ is a separated Deligne-Mumford stack. Furthermore, there exists a separated coarse moduli space of finite type over $\mathbbm{k}$. 
 \end{thm}

\begin{rem}
     The authors in \cite{Ols} and \cite{stacksproject-chap98} treat the category of schemes that are not necessarily locally Noetherian, but our theory works even if we treat $\mathbf{Sch}_{/\mathbbm{k}}$. 
     For example, we can extend $\mathfrak{Pol}$ to a stack over the category of all schemes, including schemes that are not locally Noetherian (cf.~\cite[28.2.12]{vakil22}). One can also see that we can apply Theorem \ref{dm} to $[N/PGL(d_1)\times PGL(d_2)\times PGL(d_3)]$, which is defined on the category of all schemes, in the proof of Theorem \ref{mod2}.
\end{rem}

\subsection{Universal hull and $\mathbb{Q}$-Gorenstein family}\label{subsec2.5}
For any scheme $X$ and coherent sheaf $\mathscr{F}$ on $X$ of pure dimension, we can define the $S_{2}$-{\em hull} of $\mathscr{F}$, which we denote by $\mathscr{F}^{[**]}$. 
For details, we refer to \cite[\S1.1]{huybrechts2010geometry}. If $X$ is a normal variety of dimension $d$ and $\mathscr{F}$ is of pure dimension $d$, then $\mathscr{F}^{[**]}=\mathcal{H}om_{\mathcal{O}_X}(\mathcal{H}om_{\mathcal{O}_X}(\mathscr{F},\mathcal{O}_X),\mathcal{O}_X)$.

Let $f \colon \mathcal{X} \to S$ be a flat projective surjective morphism between locally Noetherian schemes such that the relative dimension of $f$ is $d$ and all geometric fibers of $f$ are normal and connected. Then there is a closed reduced subscheme $Z\subset \mathcal{X}$ such that $f$ is smooth on $\mathcal{X}\setminus Z$ and the fiber $Z_s$ over any $s \in S$ satisfies ${\rm codim}_{\mathcal{X}_s}(Z_{s})\geq 2$.  Let $\mathscr{F}$ be a coherent sheaf on $X$ such that $\mathscr{F}|_{\mathcal{X}\setminus Z}$ is an invertible sheaf. We define a {\it (universal) hull} of $\mathscr{F}$, which we also denote by $\mathscr{F}^{[**]}$, to be a coherent sheaf with the following properties (cf.~\cite{Ko} or \cite[\S9]{kollar-moduli}).
 \begin{itemize}
 \item $\mathscr{F}^{[**]}$ is flat over $S$,
 \item there exists a morphism $q \colon \mathscr{F}\to\mathscr{F}^{[**]}$ that is an isomorphism outside $Z$, and
 \item for any point $s\in S$, the morphism $\mathscr{F}^{[**]}|_{\mathcal{X}_s}\to \mathscr{F}_s^{[**]}$ induced by $q$ is an isomorphism, where $\mathscr{F}_s^{[**]}$ is the $S_2$-hull of $\mathscr{F}_s :=\mathscr{F}|_{\mathcal{X}_s}$. 
 \end{itemize}
 A universal hull does not always exist for the sheaf $\mathscr{F}$ as above, but if it exists then $\mathscr{F}^{[**]}\cong j_*(\mathscr{F}|_{\mathcal{X}\setminus Z})$, where $j \colon \mathcal{X}\setminus Z\hookrightarrow \mathcal{X}$ is the inclusion. 
 Indeed, for any $p\in\mathcal{X}$ and any affine open neighborhood $U\subset \mathcal{X}$ of $p$, let $I_{Z}\subset\mathcal{O}_{\mathcal{X}}$ be the ideal sheaf corresponding to $Z$. 
 Take a regular sequence $\overline{a},\overline{b}\in I_Z\otimes\mathcal{O}_{U\cap\mathcal{X}_s}(U\cap\mathcal{X}_s)$ of $\mathscr{F}_s^{[**]}$ for $s:= f(p)$, i.e.~$\overline{a}$ is a non-zero divisor in $\mathscr{F}_s^{[**]}$ and $\overline{b}$ is a non-zero divisor in $\mathscr{F}_s^{[**]}/\overline{a}\mathscr{F}_s^{[**]}$. 
 If $\overline{a},\overline{b}$ are the restrictions of $a, \, b\in I_Z(U)$, then $a,b$ is also a regular sequence of $\mathscr{F}^{[**]}$ around $\mathcal{X}_s\cap U$ by \cite[(20.E)]{Mat}. Thus, shrinking $U$ if necessary, we may assume that there exists a regular sequence $a, \, b\in I_Z(U)$ of $\mathscr{F}^{[**]}$.
Now the natural map $\mathscr{F}^{[**]}\to j_*(\mathscr{F}|_{\mathcal{X}\setminus Z})$ is injective over $U$, and the surjectivity can be proved as follows: Let $m\in j_*(\mathscr{F}|_{\mathcal{X}\setminus Z})$ be a local section over $U$. By the assumption, there exists two sections $m_a,\,m_b\in\mathscr{F}^{[**]}(U)$ such that $m=\frac{m_a}{a}=\frac{m_b}{b}$. Here, we applied \cite[Theorem 27]{Mat} and assumed that $b$ is also a non-zero divisor by shrinking $U$. Thus, $bm_a=am_b$ as elements of $\mathscr{F}^{[**]}(U)$. From this and the fact that $b$ is a non-zero divisor in $(\mathscr{F}^{[**]}/a\mathscr{F}^{[**]})(U)$, we have $m\in\mathscr{F}^{[**]}(U)$. 

Hence, if a universal hull of $\mathscr{F}$ exists, then $\mathscr{F}^{[**]}\cong j_*(\mathscr{F}|_{\mathcal{X}\setminus Z})$, and furthermore we have $(\mathscr{F}^{[**]})_T= (\mathscr{F}_T)^{[**]}$ for any morphism $g \colon T\to S$. We denote this by $\mathscr{F}^{[**]}_T$.

By applying the Koll\'ar's theory \cite{kollar-moduli} to our setup, we obtain the following theorem. 

 \begin{thm}[cf.~{\cite[Theorem 9.40]{kollar-moduli}}]\label{hullsdecomp}
Let $f \colon \mathcal{X} \to S$ be a flat projective surjective morphism between schemes of finite type over $\mathbbm{k}$ such that the relative dimension of $f$ is $d$ and the geometric fibers of $f$ are normal and connected. 
Let $Z\subset \mathcal{X}$ be a closed subset such that $f$ is smooth on $\mathcal{X}\setminus Z$ and the fiber $Z_s$ over any $s \in S$ satisfies ${\rm codim}_{\mathcal{X}_s}(Z_{s})\geq 2$.  Let $\mathscr{F}$ be a coherent sheaf on $X$ such that $\mathscr{F}|_{\mathcal{X}\setminus Z}$ is an invertible sheaf on $\mathcal{X}\setminus Z$. 
Let $H$ be an $f$-ample line bundle on $\mathcal{X}$. 

Then there exist finitely many distinct polynomials $p_1,\,\cdots,\,p_l$ with corresponding locally closed subschemes $S_1,\,\cdots,\,S_l$ of $S$ satisfying the following. 
\begin{itemize}
    \item $S=\sqcup_{i=1}^lS_i $ set-theoretically,
    \item
     for each $1 \leq i \leq l$, there exists the universal hull $\mathscr{F}_{S_i}^{[**]}$ of $\mathscr{F}_{S_i}$ such that the Hilbert polynomial of $\mathscr{F}_{s}^{[**]}$ with respect to $H$ is $p_{i}$ for all $s\in S_i$, and
    \item for any morphism $g \colon T\to S$ from a locally Noetherian scheme $T$, if $\mathscr{F}_T$ has a universal hull $\mathscr{F}_T^{[**]}$ such that all fibers $\mathscr{F}_t^{[**]}$ have the same Hilbert polynomial $p$ with respect to $H_{t}$, then $p=p_i$ and $g$ factors through $S_i$ for some $i$.
\end{itemize}
 \end{thm}

The following result was proved by Hassett--Kov\'{a}cs \cite[3.11]{HK} when the fibers are Cohen-Macaulay, and Koll\'ar \cite[Proposition 9.42]{kollar-moduli} proved a more general statement. Thus, we omit the proof. 

\begin{cor}\label{cor--hako}
Let $f \colon \mathcal{X}\to S$, $\mathscr{F}$, and $H$ be as in Theorem \ref{hullsdecomp}. For any line bundle $L$ on $\mathcal{X}$, there exists a locally closed subscheme $S^u\subset S$ such that a morphism $g \colon T\to S$ from a locally Noetherian scheme $T$ factors through $S^{u} \hookrightarrow S$ if and only if the universal hull $\mathscr{F}_T^{[**]}$ exists and $L_T\otimes f_T^*M\cong \mathscr{F}_T^{[**]}$ for some line bundle $M$ on $T$.
\end{cor}

From now on, we deal with the relative dualizing sheaf. 
Let $f \colon \mathcal{X} \to S$ be a flat projective surjective morphism of schemes of finite type over $\mathbbm{k}$ whose geometric fibers are normal and connected, and let $U \subset \mathcal{X}$ be the largest open subscheme such that $f$ is smooth at every point of $U$. 
Let $\omega_{\mathcal{X}/S}$ be the relative dualizing sheaf. Then $\omega_{\mathcal{X}/S}^{\otimes m}$ is a coherent sheaf and $\omega_{\mathcal{X}/S}^{\otimes m}|_{U}$ is an invertible sheaf for every $m\in\mathbb{Z}$. 
Hence, we may use the previous results to $\omega_{\mathcal{X}/S}^{\otimes m}$. 
For each $m\in\mathbb{Z}$, if the universal hull of $\omega_{\mathcal{X}/S}^{\otimes m}$ exists, then $\omega_{\mathcal{X}/S}^{[m]}$ denotes the (universal) hull. We also have  $\omega_{\mathcal{X}_T/T}^{[m]}=(h\times_{S}\mathrm{id}_{\mathcal{X}})^*\omega_{\mathcal{X}/S}^{[m]}$ for every morphism $h \colon T\to S$ since $\omega_{U/S}=\omega_{\mathcal{X}/S}|_{U}$ is a line bundle that commutes with the base change. 

\begin{defn}[$\mathbb{Q}$-Gorenstein family, log $\mathbb{Q}$-Gorenstein family]\label{defn-q-gor}
Let $f \colon \mathcal{X} \to S$ be a flat projective surjective morphism of schemes of finite type over $\mathbbm{k}$ whose geometric fibers are normal and connected. 
We say that $f \colon \mathcal{X}\to S$ is a {\it $\mathbb{Q}$-Gorenstein family} over $S$ if there exists $m\in\mathbb{Z}_{>0}$ such that $\omega_{\mathcal{X}/S}^{[m]}$ exists as a line bundle. 

For any $f \colon \mathcal{X} \to S$ as above, if $S$ is normal, then $\mathcal{X}$ is also normal, $\omega_{\mathcal{X}/S}$ is reflexive and $\omega_{\mathcal{X}/S} = \mathcal{O}_{\mathcal{X}}(K_{\mathcal{X}/S})$ for some Weil divisor $K_{\mathcal{X}/S}$ on $\mathcal{X}$ (cf.~\cite[Prop.~A10]{PSZ}, \cite[\S2]{CP}). 

Let $f \colon \mathcal{X} \to S$ be as above. 
Suppose that $S$ is normal. 
Let $\Delta$ be an effective $\mathbb{Q}$-divisor on $\mathcal{X}$ such that the support of $\Delta$ contains no fiber of $f$. 
We say that $f \colon (\mathcal{X},\Delta)\to S$ is a {\it log $\mathbb{Q}$-Gorenstein family} if $K_{\mathcal{X}/S}+\Delta$ is $\mathbb{Q}$-Cartier. 
\end{defn}

\begin{rem}
Let $f \colon \mathcal{X} \to S$ be as above. 
Suppose that $S$ is normal. 
\begin{itemize}
\item
Let $U\subset\mathcal{X}$ be the open locus on which $f$ is smooth. 
If $f \colon (\mathcal{X},\Delta)\to S$ is a log $\mathbb{Q}$-Gorenstein family, then $\omega_{\mathcal{X}/S}|_U=\omega_{U/S}$ is an invertible sheaf (cf.~\cite[0E9Z]{stacksproject-chap98}), and thus $\Delta|_U$ is $\mathbb{Q}$-Cartier. 
 For any morphism $h \colon T\to S$ from a normal variety $T$ and the induced morphism $\sigma \colon \mathcal{X}_T \to \mathcal{X}$, we define $\Delta_T$ as a unique extension of $\sigma^*(\Delta|_U)$ on $U\times_ST$. 
 Then we can check that \[
 K_{\mathcal{X}_T/T}+\Delta_T=\sigma^*(K_{\mathcal{X}/S}+\Delta)
 \]
 by applying \cite[Thm 3.6.1]{Con} to $U\times_ST\to U$. 
 See also \cite[\S2]{CP}. 
\item
Let $\mathcal{D}$ be an effective Weil divisor on $\mathcal{X}$ such that $\mathcal{D}$ is flat over $S$ as a scheme and it has only geometrically integral fibers over $S$. 
 Then the scheme-theoretic fiber $\mathcal{D}_s$ for any $s\in S$ is also a Weil divisor, $\mathcal{O}_{\mathcal{X}}(-\mathcal{D})$ is also flat and the restriction $\mathcal{O}_{\mathcal{X}}(-\mathcal{D})|_{U}$ is locally free by \cite[Lemma 2.1.7]{huybrechts2010geometry}.
\item
Let $\Delta$ be an effective $\mathbb{Q}$-divisor on $\mathcal{X}$ such that the support of $\Delta$ contains no fiber of $f$. 
Here, we do not assume that $f \colon (\mathcal{X},\Delta)\to S$ is a log $\mathbb{Q}$-Gorenstein family. 
Let $j \colon U \hookrightarrow \mathcal{X}$ be the open immersion. 
We fix $m\in\mathbb{Z}_{>0}$ such that $m\Delta$ is a Weil divisor on $\mathcal{X}$. 
If a universal hull of $\mathcal{O}_{\mathcal{X}}(m(K_{\mathcal{X}/S}+\Delta))$ exists, then the $S_2$ condition of $\mathcal{O}_{\mathcal{X}}(m(K_{\mathcal{X}/S}+\Delta))$ implies
 $$\mathcal{O}_{\mathcal{X}}(m(K_{\mathcal{X}/S}+\Delta))=j_*\mathcal{O}_{U}(m(K_{U/S}+\Delta|_U))=\mathcal{O}_{\mathcal{X}}(m(K_{\mathcal{X}/S}+\Delta))^{[**]}.$$  
Moreover, if any irreducible component of $\Delta$ is flat over $S$ as a reduced scheme and it has only geometrically integral fibers over $S$, then $\mathcal{O}_{\mathcal{X}}(m(K_{\mathcal{X}/S}+\Delta))|_U$ is locally free. Then, we can apply  Corollary \ref{cor--hako} to $\mathcal{O}_{\mathcal{X}}(m(K_{\mathcal{X}/S}+\Delta))$ and any line bundle $L$ on $\mathcal{X}$ to construct a locally closed subscheme $S^{u}\subset S$ satisfying the property of Corollary \ref{cor--hako}.
\end{itemize}
\end{rem}

 \subsection{K-stability}\label{seckstdef}
 In this subsection, we collect some definitions and known results on K-stability.
 
A {\em polarized variety} $(X,L)$ consists of a proper normal variety $X$ and an ample $\mathbb{Q}$-line bundle $L$ on it. 
The notation of polarized varieties and subpairs are the same, however, we adopt these notations because both are standard. 
We will mainly deal with subpairs in \S\ref{Bousec}, \S\ref{toolsec}, and \S\ref{seckst}, whereas we will deal with polarized varieties in \S\ref{Consec}. 

Let $\Delta$ be a $\mathbb{Q}$-divisor such that $(X,\Delta)$ is a pair. We call $(X,\Delta,L)$ a {\em polarized pair}. We denote the algebraic group 
$$\{g\in\mathrm{Aut}(X)\,|\,g_*\Delta=\Delta, \,g^*L\sim_{\mathbb{Q}} L\}$$
 by $\mathrm{Aut}(X,\Delta,L)$. 
This is a closed subscheme of $\{g\in\mathrm{Aut}(X)\,|\,g^*L\sim_{\mathbb{Q}} L\}$, which is a group scheme of finite type over $\mathbbm{k}$ since $\chi(X,L^{\otimes m}\otimes g^*L^{\otimes n})=\chi(X,L^{\otimes m+n})$ for every sufficiently divisible $m$ and $n\in\mathbb{Z}_{>0}$ (see \cite[\S 5.6]{FGA}). 
Hence, the above algebraic group is also of finite type over $\mathbbm{k}$. 
We can check that $\mathrm{Aut}(X,\Delta,L)$ is a linear algebraic group.
 Indeed, for any sufficiently divisible $m\in\mathbb{Z}_{>0}$, since there exists a well-defined closed immersion $G_m\hookrightarrow PGL(h^0(X,L^{\otimes m}))$, the group scheme
 $$G_m:=\{g\in\mathrm{Aut}(X)\,|\,g_*\Delta=\Delta, \,g^*L^{\otimes m}\sim L^{\otimes m}\}$$ is affine.  
 Since $\mathrm{Aut}(X,\Delta,L)$ is an algebraic group and $$\mathrm{Aut}(X,\Delta,L)=\bigcup_{m:\,\text{sufficiently divisible}}G_m$$
as sets, we have $\mathrm{Aut}(X,\Delta,L)=G_m$ for some $m$.
 Hence, $\mathrm{Aut}(X,\Delta,L)$ is affine.

We say that $f \colon (X,\Delta,A)\to C$ is a {\it polarized klt-trivial fibration over a curve} if $f \colon (X,\Delta)\to C$ is a klt-trivial fibration over a proper curve and $A$ is an $f$-ample $\mathbb{Q}$-line bundle on $X$. 

We give the following ad hoc definition of uniform adiabatic K-stability of $f$.

\begin{defn}[Uniform adiabatic K-stability]\label{unifdef}
A polarized klt-trivial fibration over a curve $f \colon (X,\Delta,A)\to C$ is called {\it uniformly adiabatically K-stable} if one of the following hold.
\begin{itemize}
    \item $K_X+\Delta\sim_{\mathbb{Q}}f^*(K_C+B_C+M_C)$ is nef, or
    \item $C=\mathbb{P}^1$, $K_X+\Delta\sim_{\mathbb{Q}}uf^*(\mathcal{O}(1))$ for some $u<0$, and $\max_{p\in\mathbb{P}^1}\mathrm{ord}_p(B_C)<1+\frac{u}{2}$, where $B_C$ is the discriminant $\mathbb{Q}$-divisor with respect to $f$.
\end{itemize}
Here, $B_{C}$ and $M_{C}$ are $\mathbb{Q}$-divisors defined in Definition \ref{defn--klttrivialfib}. 
\end{defn}
 
We note that the uniform adiabatic K-stability is a condition of $(C,B_C,M_C)$, which we call a {\it log-twisted pair}, rather than $f$. 

Next, we recall the definition of K-stability but we do not need it except in \S\ref{seckst}.

\begin{defn}[K-stability]
Let $(X,\Delta,L)$ be a polarized log pair of dimension $d$. 
$\pi \colon (\mathcal{X},\mathcal{L})\to\mathbb{A}^1$ 
is called a {\em (semi)ample test configuration} if the following hold.
\begin{itemize}
\item $\pi\colon\mathcal{X}\to\mathbb{A}^1$ is a proper and flat morphism of schemes.
\item $\mathcal{L}$ is a (semi)ample $\mathbb{Q}$-line bundle on $\mathcal{X}$.
\item $\mathbb{G}_m$ acts on $(\mathcal{X},\mathcal{L})$ so that $\pi$ is $\mathbb{G}_m$-equivariant where $\mathbb{G}_m$ acts on $\mathbb{A}^1$ by multiplication.
\item $(\pi^{-1}(1),\mathcal{L}|_{\pi^{-1}(1)})\cong(X,L)$.
\end{itemize}
We will denote $\pi \colon (\mathcal{X},\mathcal{L})\to\mathbb{A}^1$ by $(\mathcal{X},\mathcal{L})$ for simplicity. In this paper, we only treat test configurations $(\mathcal{X},\mathcal{L})$ such that $\mathcal{X}$ is normal. 
A test configuration $(\mathcal{X},\mathcal{L})$ is {\em trivial} if $\mathcal{X}$ is $\mathbb{G}_m$-equivariantly isomorphic to $X\times\mathbb{A}^1$ and we denote $\mathcal{X}$ by $X_{\mathbb{A}^1}$ in this case.
Let $p\colon X_{\mathbb{A}^1}\to X$ be the canonical projection.
It is well-known that for any semiample test configuration $(\mathcal{X},\mathcal{L})$, there is a normal semiample test configuration $(\mathcal{Y},\sigma^*\mathcal{L})$ together with two $\mathbb{G}_m$-equivariant morphisms $\sigma \colon \mathcal{Y}\to\mathcal{X}$ and $\rho \colon \mathcal{Y}\to X_{\mathbb{A}^1}$ that are the identity morphisms over $\mathbb{A}^1\setminus\{0\}$. 
Let $H$ be an $\mathbb{R}$-line bundle on $X$ and $\mathcal{D}$ be the closure of $\Delta \times\mathbb{G}_m\subset \mathcal{X}$. 
Then we define the non-Archimedean Mabuchi functional and the non-Archimedean J$^H$-functional by
\begin{equation*}
\begin{split}
M^\mathrm{NA}_{\Delta}(\mathcal{X},\mathcal{L}):=&(K_{\overline{\mathcal{X}}/\mathbb{P}^1}+\overline{\mathcal{D}}+\mathcal{X}_{0,\mathrm{red}}-\mathcal{X}_0)\cdot\overline{\mathcal{L}}^d-\frac{d(K_{X}+\Delta)\cdot L^{d-1}}{(d+1)L^d}\cdot\overline{\mathcal{L}}^{d+1},\quad {\rm and}\\
(\mathcal{J}^H)^\mathrm{NA}(\mathcal{X},\mathcal{L}):=&(p\circ\rho)^*H\cdot\sigma^*\overline{\mathcal{L}}^d-\frac{dH\cdot L^{d-1}}{(d+1)L^d}\cdot\overline{\mathcal{L}}^{d+1}.
\end{split}
\end{equation*}
Here the overline denotes the canonical compactification (cf.~\cite[\S3, \S7]{BHJ}). It is easy to see that $M^\mathrm{NA}_{\Delta}(\mathcal{X},\mathcal{L})=M^\mathrm{NA}_{\Delta}(\mathcal{Y},\sigma^*\mathcal{L})$ and $(\mathcal{J}^H)^\mathrm{NA}(\mathcal{X},\mathcal{L})=(\mathcal{J}^H)^\mathrm{NA}(\mathcal{Y},\sigma^*\mathcal{L})$. 
Hence, the functionals are well-defined. 
We say that $(X,B,L)$ is {\em uniformly K-stable} (resp.~$(X,L)$ is {\em uniformly {\rm J$^H$}-stable}) if there exists a positive constant $\epsilon>0$ such that 
\[
M^\mathrm{NA}_{\Delta}(\mathcal{X},\mathcal{L})\ge\epsilon (\mathcal{J}^L)^\mathrm{NA}(\mathcal{X},\mathcal{L}),\quad (\mathrm{resp.}\, (\mathcal{J}^H)^\mathrm{NA}(\mathcal{X},\mathcal{L})\ge\epsilon (\mathcal{J}^L)^\mathrm{NA}(\mathcal{X},\mathcal{L}))
\]
for any normal semiample test configuration.

We note that $(\mathcal{J}^L)^\mathrm{NA}(\mathcal{X},\mathcal{L})\ge 0$, and $(\mathcal{J}^L)^\mathrm{NA}(\mathcal{X},\mathcal{L})=0$ if and only if $(\mathcal{X},\mathcal{L})$ is trivial for any ample normal test configuration (cf.~\cite[Proposition 7.8]{BHJ}).
In \cite{BHJ}, $(\mathcal{J}^L)^\mathrm{NA}$ is introduced and denoted by $I^\mathrm{NA}-J^\mathrm{NA}$. 
This coincides with the minimum norm independently introduced in \cite{De2}.
\end{defn}

\begin{defn}
Let $(X,\Delta,L)$ be a klt polarized pair. 
Let $r$ be a positive integer such that $rL$ is a line bundle. 
For any $m\in\mathbb{Z}_{>0}$, a $\mathbb{Q}$-divisor $D_{mr}$ is called a {\em $mr$-basis type divisor} of $L$ if $D_{mr}=\frac{1}{mrh^0(X,\mathcal{O}_X(mrL))}\sum_{i=1}^{h^0(X,\mathcal{O}_X(mrL))}E_i$ such that $E_i$'s form a basis of $H^0(X,\mathcal{O}_X(mrL))$.
We define $\delta_{mr}$ and $\delta$-invariants as follows (cf.~\cite{FO}, \cite{BlJ}).
\begin{align*} 
\delta_{mr,(X,\Delta)}(L)&:=\inf_{D_{mr}}\mathrm{lct}(X,\Delta;D_{mr}),\\\delta_{(X,\Delta)}(L)&:=\lim_{m\to\infty}\delta_{mr,(X,\Delta)}(L),
\end{align*}
where $D_{mr}$ runs over all $mr$-basis type divisors. 
By \cite{BlJ}, the above limit exists. 
\end{defn}

For any prime divisor $F$ over $X$ with a projective birational morphism $\pi \colon Y\to X$ such that $F$ appears as a prime divisor on $Y$, we define
\begin{align*}
S_{L}(F):=\frac{1}{L^n}\int_0^\infty\mathrm{vol}(L-tF)dt,
\end{align*}
where $\mathrm{vol}(L-tF)$ denotes $\mathrm{vol}(\pi^*L-tF)$ by the abuse of notations. 
We set $$S_{mr,L}(F):=\max_{D_{mr}}\mathrm{ord}_F(D_{mr})=\frac{\sum_{i\ge1}h^0(Y,\mathcal{O}_Y(mr\pi^*L-iF))}{mrh^0(X,\mathcal{O}_X(mrL))},$$ where $D_{mr}$ runs over all $mr$-basis type divisors (cf.~\cite[Lemma 2.2]{FO}). 
It is well-known (cf.~\cite[Lemma 2.9]{BlJ}) that 
$$\lim_{m\to\infty}S_{mr,L}(F)=S_L(F).$$ 
Furthermore, we have
\[
\delta_{(X,\Delta)}(L)=\inf_F\frac{A_{(X,\Delta)}(F)}{S_L(F)}
\]
 by \cite{BlJ}, where $F$ runs over all prime divisors over $X$. 
 
\begin{defn}[$\alpha$-invariant] 
Let $(X,\Delta,L)$ be a klt polarized pair. 
We define the {\em $\alpha$-invariant}, denoted by $\alpha_{(X,\Delta)}(L)$, by
\begin{equation*}
\begin{split}
\alpha_{(X,\Delta)}(L):=&\inf\{\mathrm{lct}(X,\Delta;D)\,|\, D \in |L|_{\mathbb{Q}}\}\\
=&\inf\{\mathrm{lct}(X,\Delta;D)\,|\, D \in |L|_{\mathbb{R}}\}.
\end{split}
\end{equation*}
\end{defn}
This notion was introduced by Tian \cite{T3} (and restated in \cite{T2}) to obtain a sufficiency condition for the existence of K\"{a}hler--Einstein metrics on Fano manifolds.
 
The following fact is well-known. 
 \begin{lem}[cf.~{\cite[Theorem 9.14]{BHJ}}, {\cite[Proposition 2.1, Lemma 2.2]{Fjtb}}]\label{lem-delta-alpha}
 Let $(X,\Delta,L)$ be a $d$-dimensional klt polarized pair. Then
\[
0<\frac{d+1}{d}\alpha_{(X,\Delta)}(L)\le \delta_{(X,\Delta)}(L)\le (d+1)\alpha_{(X,\Delta)}(L).
\]
\end{lem}
\begin{ex}\label{ex--calculation}
When $X$ is a curve, we can easily compute $\delta_{(X,\Delta)}(L)$ as follows: Since every prime divisor over $X$ is a point $P\in X$, we have \[
S_L(P)=\frac{1}{\mathrm{deg}\,L}\int^{\mathrm{deg}\,L}_0(\mathrm{deg}\,L-t)dt=\frac{\mathrm{deg}\,L}{2}.\]
Thus, we have
\[
\delta_{(X,\Delta)}(L)=\frac{2}{\mathrm{deg}\,L}\inf_PA_{(X,\Delta)}(P)=2\frac{(1-\max_{P\in X}\mathrm{ord}_P(\Delta))}{\mathrm{deg}\,L}.
\]
In this case we have $$\alpha_{(X,\Delta)}(L)=\frac{1-\max_{P\in X}\mathrm{ord}_P(\Delta)}{\mathrm{deg}\,L}=\frac{1}{2}\delta_{(X,\Delta)}(L).$$
\end{ex}

The following notion will also be used in this paper.

\begin{defn}[Special K-stability, \cite{CM}]\label{defn-special}
We say that a klt polarized pair $(X,\Delta,L)$ is {\em specially K-stable} if $\delta_{(X,\Delta)}(L)L+K_X+\Delta$ is ample and $(X,L)$ is uniformly J$^{\delta_{(X,\Delta)}(L)L+K_X+\Delta}$-stable. 
\end{defn}
Note that the special K-stability depends only on the numerical class of $L$ since so do $\delta_{(X,\Delta)}(L)$ and the uniform J$^{\delta_{(X,\Delta)}(L)L+K_X+\Delta}$-stability.

By the following, we know that the special K-stability implies the uniform K-stability. 
\begin{thm}[{\cite[Corollary 3.21]{CM}}]\label{A.13}
Let $(X,\Delta,L)$ be a klt polarized variety and $(\mathcal{X},\mathcal{L})$ be a normal semiample test configuration for $(X,L)$. 
Then,
\[
M^{\mathrm{NA}}_{\Delta}(\mathcal{X},\mathcal{L})\ge (\mathcal{J}^{\delta_{(X,\Delta)}(L)L+K_X+\Delta})^{\mathrm{NA}}(\mathcal{X},\mathcal{L}).
\]
\end{thm}
Over $\mathbb{C}$, there exists an intrinsic criterion for J-stability and special K-stability without using test configurations. 

\begin{thm}\label{thm-jst}
Let $(X,L)$ be a polarized variety over $\mathbb{C}$ of dimension $d$, and let $H$ be an ample $\mathbb{R}$-line bundle on $X$. 
Then $(X,L)$ is uniformly $\mathrm{J}^H$-stable if and only if there exists $\epsilon>0$ such that
\[
\left(d\frac{H\cdot L^{d-1}}{L^d}L-pH\right)\cdot L^{p-1}\cdot V\ge \epsilon(d-p) L^p\cdot V
\] for any $p$-dimensional subvariety $V\subset X$ with $0<p<d$. 
In particular, if $(X,\Delta,L)$ is a polarized klt pair and $H:=\delta_{(X,\Delta)}(L)L+K_X+\Delta$ is ample, then the specially K-stablity of $(X,\Delta,L)$ is equivalent to the existence of $\epsilon>0$ such that the above inequality holds for any subvariety $V\subset X$.
\end{thm}

The above theorem has firstly been proved in the case of K\"{a}hler manifolds by \cite{G}, but currently the theorem holds for all polarized varieties by \cite[Theorem 8.12]{Hat2}. 
For polarized (resp.~K\"{a}hler) manifolds, it was shown by \cite{DP} (resp.~\cite{So}) that uniform J$^H$-stability is equivalent to a certain weaker condition.

Roughly speaking, the uniform adiabatic K-stability of $f \colon (X,\Delta,A)\to C$ was originally defined to be the uniform K-stability of $(X,\Delta,\epsilon A+L)$ with fixed some ample $\mathbb{Q}$-line bundle $L$ on $C$ for any sufficiently small $\epsilon>0$.
The original definition \cite[Definition 2.6]{Hat} and the ad hoc definition coincide by the following theorems.  
Furthermore, we do not have to fix $L$ on $C$ by what we stated after Definition \ref{defn-special}.

\begin{thm}[{\cite[Theorem 8.15]{Hat2}}]\label{k-plus}
Let $(X,\Delta,A)$ be a klt polarized pair over $\mathbb{C}$. If $K_X+\Delta$ is nef, then there exists a real number $C>0$, depending only on the intersection numbers $A^i\cdot(K_X+\Delta)^{{\rm dim}\,X-i}$ for $0\leq i \leq {\rm dim}\,X$, such that $(X,\epsilon A+K_X+\Delta)$ is uniformly $\mathrm{J}^{K_X+\Delta+C\epsilon(\epsilon A+K_X+\Delta)}$-stable for every $\epsilon>0$.
Furthermore, there exist real numbers $\epsilon_0>0$ and $\alpha>0$ such that $(X,\Delta,\epsilon A+K_X+\Delta)$ is specially K-stable and
\[
M^\mathrm{NA}_\Delta(\mathcal{X},\mathcal{M})\ge\alpha (\mathcal{J}^{\epsilon A+K_X+\Delta})^\mathrm{NA}(\mathcal{X},\mathcal{M})
\]
 for any $0<\epsilon<\epsilon_0$ and normal semiample test configuration $(\mathcal{X},\mathcal{M})$ for $(X,\epsilon A+K_X+\Delta)$. 
\end{thm}

\begin{thm}[{\cite[Theorem B]{Hat}}]\label{k-minus}
Let $f \colon (X,\Delta,A)\to (\mathbb{P}^1,\mathcal{O}(1))$ be a polarized klt trivial fibration over $\mathbb{C}$ such that $K_X+\Delta\sim_{\mathbb{Q}}uf^*(\mathcal{O}(1))$ for some $u<0$. 

Then, $f$ is uniformly adiabatically K-stable if and only if there exist real numbers $\epsilon_0>0$ and $\alpha>0$ such that $(X,\Delta,\epsilon A-(K_X+\Delta))$ is specially K-stable and 
\[
M^\mathrm{NA}_\Delta(\mathcal{X},\mathcal{M})\ge\alpha (\mathcal{J}^{\epsilon A-(K_X+\Delta)})^\mathrm{NA}(\mathcal{X},\mathcal{M})
\]
 for any $0<\epsilon<\epsilon_0$ and for any normal semiample test configuration $(\mathcal{X},\mathcal{M})$ for $(X,\epsilon A-(K_X+\Delta))$. 

\end{thm}

We remark that Theorem \ref{k-plus} follows from the proof of \cite[Theorem 8.15]{Hat2}.
On the other hand, the following equalities 
\[
\lim_{\epsilon\to0}\delta_{(X,\Delta)}(\epsilon A-(K_X+\Delta))= 2\frac{(\max_{p\in\mathbb{P}^1}\mathrm{ord}_p(B_{\mathbb{P}^1})-1)}{u}=\delta_{(\mathbb{P}^1,B_{\mathbb{P}^1})}(-K_{\mathbb{P}^1}-B_{\mathbb{P}^1}-M_{\mathbb{P}^1})
\]
are key steps to show Theorem \ref{k-minus} (cf.~\cite[Theorem D]{Hat}). 
We will show Theorem \ref{mainunif} in \S\ref{seckst} with them in mind.

\section{Boundedness}\label{Bousec}

In this section, we prove results of the boundedness of certain classes.

Let $d$ be a positive integer, $v$ a positive rational number, $u\neq0$ a rational number, and $\Theta \subset \mathbb{Q}_{\geq 0}$ a DCC set. 
We set $e:=\frac{u}{|u|}$ (thus $eu=|u|$). 
We consider the following sets:
\begin{equation*}
\begin{split}
\mathfrak{F}_{d, \Theta, v}:=&\left\{ f\colon (X,\Delta) \to C \;\middle|
\begin{array}{rl}
(i)&\text{$f$ is a klt-trivial fibration over a curve $C$,}\\
(ii)&\text{${\rm dim}X=d$,}\\
(iii)&\text{the coefficients of $\Delta$ belong to $\Theta$,}\\
(iv)&\text{there is an $f$-ample $\mathbb{Q}$-Cartier Weil}\\
&\text{divisor $A$ on $X$ such that ${\rm vol}(A|_{F})=v$, }\\
& \text{where $F$ is a general fiber of $f$.}
\end{array}\right\}\end{split}
\end{equation*}
and
\begin{equation*}
\begin{split}
\mathfrak{G}_{d, \Theta, v, u}:=&\left\{ f\colon (X,\Delta) \to C \in \mathfrak{F}_{d, \Theta, v} \;\middle|
\begin{array}{l}
\text{$K_{X}+\Delta \equiv uf^{*}H$ for some Cartier}\\
\text{divisor $H$ with ${\rm deg}\,H=1$}
\end{array}\right\}.
\end{split}
\end{equation*}
Lemma \ref{lem--Cartierindex} below is crucial for the boundedness and the lemma will be used in \S\ref{Consec}. 

\begin{lem}\label{lem--Cartierindex}
There exists a positive integer $r$, depending only on $d$, $\Theta$, $v$, and $u$, such that for any element $f\colon (X,\Delta) \to C$ of $\mathfrak{G}_{d, \Theta, v, u}$, we have $er(K_X+\Delta)\sim f^{*}D$ for some very ample Cartier divisor $D$ on $C$. 
In particular, $er(K_{X}+\Delta)$ is a base point free Cartier divisor and the linear system $|er(K_X+\Delta)|$ defines $f$. 
Furthermore, there are only finitely many possibilities of ${\rm dim}\,H^0(X,\mathcal{O}_{X}(er(K_X+\Delta)))$. 
\end{lem}

\begin{proof}
We fix an element $f\colon (X,\Delta) \to C$ of $\mathfrak{G}_{d, \Theta, v, u}$, and we pick a Weil divisor $A$ on $X$ as in (iv) of $\mathfrak{F}_{d, \Theta, v}$. 

Let $F$ be a general fiber of $f$, and pick a Cartier divisor $H$ on $C$ with ${\rm deg}\,H=1$. 
By applying \cite[Corollary 1.4]{birkar-geometry-moduli} to $(F,\Delta|_{F})$ and $A|_{F}$, we can find $m \in \mathbb{Z}_{>0}$, depending only on $d$ and $\Theta$, such that 
$H^{0}(F, \mathcal{O}_{X}(mA|_{F}))\neq 0$. 
Then there is a sufficiently large positive integer $t$ such that 
$$E \sim mA+tf^{*}H$$
for some effective Weil divisor $E$ on $X$. 
By construction, we have ${\rm vol}(E|_{F})=m^{d-1}v$. 
By applying \cite[Lemma 7.4]{birkar-nefpart} to $(X,\Delta)\to C$ and $E$, we can find a positive integer $q$, depending only on $d$, $\Theta$, and $v$, such that we can wirte
$$q(K_{X}+\Delta)\sim qf^{*}(K_{C}+B+M),$$
where $B$ (resp.~$M$) is the discriminant part (resp.~moduli part) of the canonical bundle formula, such that $qM$ is Cartier. 
Then we have ${\rm deg}(K_{C}+B+M)\leq eu$. 

By definition of the discriminant part of the canonical bundle formula and the ACC for lc thresholds \cite[Theorem 1.1]{hmx-acc}, we see that the coefficients of $B$ belong to a DCC set of $\mathbb{Q}_{> 0}$ depending only on $d$ and $\Theta$, which we denote by $\Psi$. 
Let $q'$ be the smallest positive integer such that $q'u$ is an integer and $q$ divides $q'$.  
Since ${\rm deg}\,K_{C}\geq-2$ and ${\rm deg}\,M\in \frac{1}{q}\mathbb{Z}_{\geq0}$ by Theorem \ref{Bsemi}, we see that 
$${\rm deg}(q'B) \in \{0,1,\cdots, eq'u+2q'\}.$$
We define $\delta:={\rm inf}\,\Psi$, which is a positive rational number because $\Psi$ satisfies the DCC. 
Since ${\rm deg}\,B\leq eu+2$, the number of components of $B$ is not greater than $\frac{eu+2}{\delta}$. 
Thus, all the coefficients of $B$ belong to the following set:
$$\Psi':=\left\{a_{0}-\sum_{i=1}^{l}a_{i}\,\middle |\, a_{0} \in \frac{1}{q'}\mathbb{Z}\cap [0, eu+2], \;a_{i}\in \Psi,\; l\leq \frac{eu+2}{\delta}\right\}.$$
We can easily check that $\Psi'$ satisfies the ACC because $\frac{1}{q'}\mathbb{Z}\cap [0, eu+2]$ is a finite set and $\Psi$ satisfies the DCC. 
Hence $\Psi \cap \Psi'$ satisfies the ACC and the DCC, which implies that $\Psi \cap \Psi'$ is a finite set. 

From these facts, we can find $q''$, depending only on $\Psi \cap \Psi'$, such that $q''B$ is a Weil divisor. 
By construction, $\Psi \cap \Psi'$ depends only on $d$, $\Theta$, $v$, and $u$.  
Since $q$ divides $q'$ by construction, we have 
 $$q'q''(K_{X}+\Delta)\sim q'q''f^{*}(K_{C}+B+M)$$
and the right hand side is Cartier. 
The integer $q'q''$ depends only on $d$, $\Theta$, $v$, and $u$.  

Since $eq'q''(K_{C}+B+M)$ is ample and Cartier, there is a positive integer $r$, depending only on $d$, $\Theta$, $v$, and $u$, such that $r$ is divided by $q'q''$ and $er(K_{C}+B+M)$ is very ample. 
Then this $r$ is the desired positive integer. 
We put $D:=er(K_{C}+B+M)$. 
The finiteness of ${\rm dim}\,H^0(X,\mathcal{O}_{X}(er(K_X+\Delta)))$ follows from 
\begin{equation*}
\begin{split}
0\leq &{\rm dim}\,H^0(X,\mathcal{O}_{X}(er(K_X+\Delta)))={\rm dim}\,H^0(C,\mathcal{O}_{C}(D))\\
=& {\rm dim}\,H^1(C,\mathcal{O}_{C}(D))+{\rm deg}\, D+\chi(C,\mathcal{O}_{C})\\
=& eru+1+\bigl({\rm dim}\,H^0(C,\mathcal{O}_{C}(K_{C}-D))- {\rm dim}\,H^0(C,\mathcal{O}_{C}(K_{C}))\bigr)\\
\leq & eru+1
\end{split}
\end{equation*}
by the Riemann--Roch theorem. 
\end{proof}

Let $n$ be a positive integer. 
We define 
$$\mathfrak{F}_{d, n, v}:=\left\{ f\colon (X,\Delta) \to C \;\middle|
\begin{array}{rl}
(i)&\text{$f$ is a klt-trivial fibration over a curve $C$,}\\
(ii)&\text{${\rm dim}X=d$,}\\
(iii)&\text{$n\Delta$ is a Weil divisor,}\\
(iv)&\text{there is an $f$-ample $\mathbb{Q}$-Cartier Weil}\\
&\text{divisor $A$ on $X$ such that ${\rm vol}(A|_{F})=v$, }\\
& \text{where $F$ is a general fiber of $f$.}
\end{array}\right\}.$$
Then Lemma \ref{lem--Cartierindex} shows the existence of an $n \in \mathbb{Z}_{>0}$ such that $\mathfrak{G}_{d, \Theta, v, u}$ is a subset of the following set
\begin{equation*}
\begin{split}
\mathfrak{G}_{d, n, v, u}:=&\left\{ f\colon (X,\Delta) \to C \in \mathfrak{F}_{d, n, v} \;\middle|
\begin{array}{l}
\text{$K_{X}+\Delta \equiv uf^{*}H$ for some}\\
\text{Cartier divisor $H$ with ${\rm deg}\,H=1$}
\end{array}\right\}.
\end{split}
\end{equation*}
Moreover, there exists a positive integer $r$, depending only on $d$, $n$, $v$, and $u$, such that for any element $f\colon (X,\Delta) \to C$ of $\mathfrak{G}_{d, n, v, u}$, the divisor $er(K_{X}+\Delta)$ is a base point free Cartier divisor and the linear system $|er(K_{X}+\Delta)|$ defines $f$. 

In the rest of this section, we will deal with $\mathfrak{F}_{d, n, v}$ and $\mathfrak{G}_{d, n, v, u}$ for the fixed $d, n \in \mathbb{Z}_{>0}$, $v \in \mathbb{Q}_{>0}$, and $u \in \mathbb{Q}_{\neq 0}$. 

The following lemma gives a lower bound of the $\alpha$-invariants for general fibers of the elements of $\mathfrak{F}_{d,n,v}$.

\begin{lem}\label{lem--lct}
There exists a positive integer $N$, depending only on $d$, $n$, and $v$, such that for any element $f\colon (X,\Delta) \to C$ of $\mathfrak{F}_{d, n, v}$ and any $\mathbb{Q}$-Cartier Weil divisor $A$ on $X$ as in (iv) of $\mathfrak{F}_{d, n, v}$, we have the inequality 
$$\alpha_{(F,\Delta|_{F})}(A|_F)\geq \frac{1}{N},$$
where $F$ is a general fiber of $f$. 
\end{lem}

\begin{proof}
We fix $f\colon (X,\Delta) \to C\in \mathfrak{F}_{d, n, v}$ and $A$ as in the final condition of $\mathfrak{F}_{d, n, v}$. 
Let $F$ be a general fiber of $f$. 
We put $\Delta_{F}=\Delta|_{F}$. 

By applying \cite[Corollary 1.4]{birkar-geometry-moduli} to $(F,\Delta_{F})$ and $A|_{F}$, we can find an $m \in \mathbb{Z}_{>0}$, depending only on $d$ and $n$, such that 
$$mA|_{F} \sim E_{F}$$
 for some effective Weil divisor $E_{F}$ on $F$. 
Then we have ${\rm vol}(E_{F})=m^{d-1}v$. 
By applying \cite[Corollary 1.6]{birkar-geometry-moduli} to $(F,\Delta_{F})$ and $E_{F}$, we see that $(F,{\rm Supp}(\Delta_{F}+E_{F}))$ belongs to a bounded family depending only on $d$, $n$, $v$, and $m$. 
By \cite[Lemma 2.24]{birkar-compl}, there exists $m' \in \mathbb{Z}_{>0}$, depending only on $d$, $n$, $v$, and $m$, such that $m'E_{F}\sim m'mA|_{F}$ is Cartier. 
Then $m'm$ depends only on $d$, $n$, and $v$. 

Because the divisor $m'mA|_{F}-(K_{F}+\Delta_{F})$ is ample, we may apply the effective base point freeness \cite[Theorem 1.1]{kollar-eff-basepoint-free} and the effective very ampleness \cite[Lemma 7.1]{fujino-eff-slc}. 
Hence, there exists $m'' \in \mathbb{Z}_{>0}$, depending only on $d$, $n$, and $v$, such that $m''A|_{F}$ is very ample. 
Taking a small $\mathbb{Q}$-factorialization of $X$ and applying the length of extremal rays, we see that $K_{F}+3dm''A|_{F}$ is the pushdown of a big divisor (cf.~\cite[Lemma 2.46]{birkar-compl}). 
Because we have
$$3dm''A|_{F}-\Delta_{F}\sim_{\mathbb{Q}}3dm''A|_{F}+K_{F},$$
we see that $3dm''A|_{F}-\Delta_{F}$ is the pushdown of a big divisor. 
We also have 
$${\rm vol}(3dm''A|_{F})=(3dm'')^{d-1}v.$$ 
By the ACC for numerically trivial pairs \cite[Theorem 1.5]{hmx-acc}, we can find a positive real number $\delta>0$, depending only on $d-1$ and $n$, such that $(F,\Delta_{F})$ is $\delta$-lc. 

From the above discussion, we may apply \cite[Theorem 1.8]{birkar-bab} to $(F,\Delta_{F})$ and $A|_{F}$, and we can find $\epsilon \in \mathbb{R}_{>0}$, depending only on $d-1$, $\delta$, and $(3dm'')^{d-1}v$, such that 
$$\alpha_{(F,\Delta|_{F})}(A|_F)\geq \epsilon.$$
Construction of $m''$ and $\delta$ implies that $\epsilon$ depends only on $d$, $n$, and $v$. 
Finally, we define $N$ to be the minimum positive integer satisfying $\epsilon >\frac{1}{N}$. 
Then $N$ satisfies the condition of Lemma \ref{lem--lct}. 
\end{proof}

\begin{defn}\label{defn--boundedvol}
Let $N$ be the positive integer in Lemma \ref{lem--lct}. 
We define 
$$\alpha:=d(4N+\lceil euN \rceil)v.$$
Note that $\alpha$ depends only on $d$, $n$, $v$, and $u$. 
\end{defn}

The following result is a crucial step for the boundedness and a special form of Theorem \ref{mainbound} (\ref{bound-(2)}). 

\begin{prop}\label{prop--nefthreshold-volume}
For any element $f\colon (X,\Delta) \to C$ of $\mathfrak{G}_{d, n, v, u}$ and any $\mathbb{Q}$-Cartier Weil divisor $A$ on $X$ as in (iv) of $\mathfrak{F}_{d, n, v}$, there exists a Cartier divisor $D$ on $C$ such that $A+f^{*}D$ is ample and  ${\rm vol}(A+f^{*}D)\leq \alpha$. 
\end{prop}

\begin{proof}
We fix an element $f\colon (X,\Delta) \to C$ of $\mathfrak{G}_{d, n, v, u}$, and we pick a Weil divisor $A$ on $X$ as in (iv) of $\mathfrak{F}_{d, n, v}$. 
Let $H$ be a Cartier divisor on $C$ such that ${\rm deg}\,H=1$.  

We define $\tau$ to be the smallest integer such that $A+\tau f^{*}H$ is big. 
Note that $\tau$ is well-defined since $H$ is not numerically trivial and $A+tf^{*}H$ is ample for all $t\gg0$. 
We fix an effective $\mathbb{Q}$-divisor 
$$A' \sim_{\mathbb{Q}} A+\tau f^{*}H.$$ 
Let $N \in \mathbb{Z}_{>0}$ be as in Lemma \ref{lem--lct}. 
By the property of $N$ in Lemma \ref{lem--lct}, there is a non-empty open subset $U\subset C$ such that $(X,\Delta+\frac{1}{N}A')$ is lc on $f^{-1}(U)$. 

Because $A'$ is ample over $C$ and $K_{X}+\Delta \sim_{\mathbb{Q},C}0$, we can find a positive integer $l$ such that $$\Phi:=l\left(K_{X}+\Delta-f^{*}K_{C}+\frac{1}{N}A'\right)$$
is Cartier and base point free over $C$ and $\Phi$ defines an embedding into $\mathbb{P}_{C}(f_{*}\mathcal{O}_{X}(\Phi))$. 
By applying \cite[Theorem 1.11]{fujino-semi-positivity} to $f\colon X \to C$ and $\Phi$, we see that $f_{*}\mathcal{O}_{X}(\Phi)$ is nef.
In other words, the Cartier divisor corresponding to $\mathcal{O}_{\mathbb{P}_{C}(f_{*}\mathcal{O}_{X}(\Phi))}(1)$ is nef. 
Because $\mathcal{O}_{X}(\Phi)$ coincides with the pullback of $\mathcal{O}_{\mathbb{P}_{C}(f_{*}\mathcal{O}_{X}(\Phi))}(1)$ to $X$, it follows that $\Phi$ is nef. 

Since $f\colon (X,\Delta) \to C$ is an element of $\mathfrak{G}_{d, n, v, u}$, it follows that the divisor
$$euf^{*}H-(K_{X}+\Delta)$$ 
is nef. 
Since ${\rm deg}H=1$, we see that $2H+K_{C}$ is nef. 
Therefore, the divisor
\begin{equation*}
\begin{split}
&(2+eu+\frac{\tau}{N})f^{*}H+\frac{1}{N}A\\
\sim_{\mathbb{Q}}&(2+eu)f^{*}H+\frac{1}{N}A'\\
=&(K_{X}+\Delta-f^{*}K_{C}+\frac{1}{N}A')-(K_{X}+\Delta)+f^{*}K_{C}+(2+eu)f^{*}H\\
=&\frac{1}{l}\Phi+f^{*}(2H+K_{C})+(euf^{*}H-(K_{X}+\Delta))
\end{split}
\end{equation*}
is nef. 
Thus, $A+(N(2+eu)+\tau)f^{*}H$ is nef. 
Since $A$ is $f$-ample, we see that 
$$A+(3N+euN+\tau)f^{*}H=A+(N(2+eu)+\tau)f^{*}H+Nf^{*}H$$
is ample. 

By definition of $\tau$, the divisor $A+(\tau-1)f^{*}H$ is not big.
Hence, we have 
$${\rm vol}(A+(\tau-N)f^{*}H)=0.$$ 
Since $K_{X}+\Delta \equiv uf^{*}H$, by the canonical bundle formula, we have $eu \geq {\rm deg}\,K_{C}$. 
Then $(4N+\lceil euN \rceil)H$ is very ample. 

We put $A'':=A+(\tau-N)f^{*}H$ and $N':=4N+\lceil euN \rceil$. 
Let $G\in |N'f^{*}H|$ be a member consisting of $N'$ general fibers. 
Then ${\rm vol}(A'')=0$ and $A''|_{G}=A|_{G}$ by definition. 
For each $m\in \mathbb{Z}_{> 0}$ and $0< k \leq m$, we consider the exact sequence
\begin{equation*}
\begin{split}
0\longrightarrow H^{0}\bigl(X,\mathcal{O}_{X}\bigl(mA''+(k-1)N'f^{*}H\bigr)\bigr)
\longrightarrow& H^{0}\bigl(X,\mathcal{O}_{X}\bigl(mA''+kN'f^{*}H\bigr)\bigr)\\
\longrightarrow &H^{0}\bigl(G,\mathcal{O}_{G}(mA|_{G})\bigr)
\end{split}
\end{equation*}
induced by
$$
0\longrightarrow \mathcal{O}_{X}\bigl(mA''+kN'f^{*}H-G\bigr)
\longrightarrow \mathcal{O}_{X}\bigl(mA''+kN'f^{*}H\bigr)
\longrightarrow \mathcal{O}_{G}(mA|_{G})\longrightarrow 0.$$
By similar arguments to \cite[Proof of Lemma 2.5]{jiang-fano-birat-bound}, we have 
\begin{equation*}\begin{split}
&{\rm dim}\,H^{0}(X,\mathcal{O}_{X}(mA''+mN'f^{*}H))\\
\leq&{\rm dim}\,H^{0}(X,\mathcal{O}_{X}(mA''))+m\cdot{\rm dim}H^{0}\bigl(G,\mathcal{O}_{G}(mA|_{G})\bigr).
\end{split}\end{equation*}
Since $G$ consists of $N'$ general fibers, taking the limit $m \to \infty$, we have
$${\rm vol}(A''+N'f^{*}H)\leq {\rm vol}(A'')+dN'\cdot {\rm vol}(A|_{F})
=0+dN'v. $$
Here, we used that ${\rm vol}(A'')=0$ and ${\rm vol}(A|_{F})=v$. 
We put
$$D=(3N+\lceil euN \rceil+\tau)H.$$
By recalling the definitions of $A''$, $N'$, and $\alpha$ (see Definition \ref{defn--boundedvol}), we obtain
$${\rm vol}(A+f^{*}D)={\rm vol}(A''+N'f^{*}H)\leq dN'v=\alpha.$$
Thus $D$ is the desired Cartier divisor on $C$. We finish the proof. 
\end{proof}

\begin{defn}\label{defn--maximum-multiple}
Let $\alpha$ be the positive real number in Definition \ref{defn--boundedvol}. 
For any element $f\colon (X,\Delta) \to C$ of $\mathfrak{G}_{d, n, v, u}$ and any $\mathbb{Q}$-Cartier Weil divisor $A$ on $X$ as in (iv) of $\mathfrak{F}_{d, n, v}$, we pick a Cartier divisor $H$ on $C$ with ${\rm deg}\,H=1$ and we
define
$$m_{(f,A)}:={\rm max}\left\{ m \in \mathbb{Z} \mid \text{$A+mf^{*}H$ is ample, ${\rm vol}(A+mf^{*}H)\leq \alpha$}\right\}.$$
Note that $m_{(f,A)}$ is well-defined by Proposition \ref{prop--nefthreshold-volume}.
We may have $m_{(f,A)}\leq 0$. 
We define $$L_{(f,A)}:=A+m_{(f,A)} f^{*}H.$$
\end{defn}

Now we are ready to prove the boundedness. 

\begin{thm}[Boundedness]\label{thm--eff-veryample}
The set of klt pairs $(X,\Delta)$ appearing in $\mathfrak{G}_{d, n, v, u}$ is log bounded. 
Furthermore, there exists a positive integer $I_{0}$, depending only on $d$, $n$, $v$, and $u$, such that $I_{0}L_{(f,A)}$ is an ample Cartier divisor on $X$.
In particular, $I_0A$ is Cartier.
\end{thm}

\begin{proof}
We pick $f\colon (X,\Delta) \to C \in \mathfrak{G}_{d, n, v, u}$ and a $\mathbb{Q}$-Cartier Weil divisor $A$ on $X$ as in (iv) of $\mathfrak{F}_{d, n, v}$. 
By Lemma \ref{lem--Cartierindex}, we can find a positive integer $r$, depending only on $d$, $n$, $v$, and $u$, such that $(X,\Delta)$ is $\frac{1}{r}$-lc. 
By \cite[Corollary 1.4.3]{BCHM}, there is a small $\mathbb{Q}$-factorization $\phi \colon X' \to X$ of $X$. 
Then $(X',0)$ is an $\frac{1}{r}$-lc pair and $3d\phi^{*}L_{(f,A)}-K_{X'}$ is big. 
By applying \cite[Theorem 1.1]{birkar-geometry-moduli} to $X'$ and $3d\phi^{*}L_{(f,A)}$, 
we can find a positive integer $m$, depending only on $d$, $n$, $v$, and $u$, such that 
$$H^{0}(X', \mathcal{O}_{X'}(m\phi^{*}L_{(f,A)})) \neq 0.$$ 
Thus, we can find an effective Weil divisor $E \sim mL_{(f,A)}$. 
We have 
$${\rm vol}(E)\leq m^{d}\alpha.$$

For each $f\colon (X,\Delta) \to C \in \mathfrak{G}_{d, n, v, u}$, we fix $E \sim mL_{(f,A)}$ as above, and we prove that the set of $(X,{\rm Supp}(\Delta+E))$ is bounded.  
If $u<0$, then we have $eu \leq 2$ by the canonical bundle formula (cf.~Theorem \ref{Bsemi}). 
By taking a reduced divisor $G$ on $X$ consisting of three general fibers of $f$, we get a klt Calabi--Yau pair $(X,\Delta+\frac{eu}{3}G)$. 
By applying \cite[Corollary 1.6]{birkar-geometry-moduli} to $(X,\Delta+\frac{eu}{3}G)$ and $E$, we see that the set of such couples $(X,{\rm Supp}(\Delta+G+E))$ is bounded. 
In particular, the set of $(X,{\rm Supp}(\Delta+E))$ for some $f\colon (X,\Delta) \to C \in \mathfrak{G}_{d, n, v, u}$ is bounded. 
If $u>0$, then we pick a Cartier divsior $H$ on $C$ such that ${\rm deg}H=1$ and 
$K_{X}+\Delta \equiv uf^{*}H$.  
Since the volume depends only on the numerical class, we have 
\begin{equation*}
\begin{split}
{\rm vol}(K_{X}+\Delta+E)=(K_{X}+\Delta+E)^{d}=&
d(uf^{*}H\cdot E^{d-1})+(E^{d})\\
=&duf^{*}H \cdot (mL_{(f,A)})^{d-1}+{\rm vol}(E)\\
\leq& dum^{d-1}v+m^{d}\alpha.
\end{split}
\end{equation*}
Since $(X,\Delta)$ is $\frac{1}{r}$-lc by Lemma \ref{lem--Cartierindex} and $n\Delta$ is a Weil divisor by definition, we may apply \cite[Theorem 1.5]{birkar-geometry-moduli} to $(X,\Delta)$ and $E$, and the set of $(X,{\rm Supp}(\Delta+E))$ is bounded. 
By these arguments, we obtain the boundedness of the set of $(X,{\rm Supp}(\Delta+E))$. 

The first statement of Theorem \ref{thm--eff-veryample} immediately follows from the above discussion. 
Moreover, \cite[Lemma 2.24]{birkar-compl} implies the existence of a positive integer $I'$, depending only on $d$, $n$, $v$, and $u$, such that $I'E$ is Cartier. 
Set $I_{0}:=I'm$. 
Then $I_{0}$ depends only on $d$, $n$, $v$, and $u$, and $I_0L_{(f,A)}$ is Cartier. 
We complete the proof.
\end{proof}

\begin{rem}\label{rem--ell-bound} 
We define
$$\mathfrak{G}_{d, n, v, 0}:=\left\{ f\colon (X,\Delta) \to C \in \mathfrak{F}_{d, n, v} \;\middle|
\begin{array}{l}
K_{X}+\Delta \equiv 0
\end{array}\right\}, \qquad {\rm and}$$
$$\mathfrak{V}_{d,n,v}:=\left\{ (X,\Delta) \;\middle|
\begin{array}{l}
\text{$(X,\Delta)$ is a klt pair and $f\colon (X,\Delta) \to C \in \mathfrak{G}_{d, n, v,0}$ for some $f$}
\end{array}\right\}.$$
Then the same argument as in this section implies that $\mathfrak{V}_{d,n,v}$ is log bounded.  
Indeed, applying the argument in Lemma \ref{lem--Cartierindex}, we can find a positive integer $r$, depending only on $d$, $n$, and $v$, such that $r(K_{X}+\Delta) \sim 0$. 
We define $\alpha:=4dNv$ (see Definition \ref{defn--boundedvol}). 
By the same argument as in Proposition \ref{prop--nefthreshold-volume}, for any element $f\colon (X,\Delta) \to C$ of $\mathfrak{G}_{d, n, v, 0}$ and any $\mathbb{Q}$-Cartier Weil divisor $A$ on $X$ as in (iv) of $\mathfrak{F}_{d, n, v}$, there exists a Cartier divisor $D$ on $C$ such that $A+f^{*}D$ is ample and  ${\rm vol}(A+f^{*}D)\leq \alpha$. 
Then we can define the line bundle $L_{(f,A)}$ as in Definition \ref{defn--maximum-multiple}. 
Then the argument in Theorem \ref{thm--eff-veryample} implies that $\mathfrak{V}_{d,n,v}$ is log bounded. 
Moreover, there exists a positive integer $I_0$, depending only on $d$, $n$, and $v$, such that $I_0L_{(f,A)}$ is an ample Cartier divisor (cf.~Theorem \ref{thm--eff-veryample}). 
\end{rem}

\begin{proof}[Proof of Theorem \ref{mainbound}]
By Lemma \ref{lem--Cartierindex}, we may assume that $\Theta= \frac{1}{n}\mathbb{Z}\cap [0,1]$ for some $n\in\mathbb{Z}_{>0}$.
Then the assertion immediately follows from Theorem \ref{thm--eff-veryample}, Remark \ref{rem--ell-bound}, and the existence of $L_{(f,A)}$ as in Definition \ref{defn--maximum-multiple}.
\end{proof}

We make use of the following result to construct moduli spaces in Section \ref{Consec}.

\begin{cor}\label{cor--hilbertpolynomial}
Fix $d\in\mathbb{Z}_{>0}$, a DCC subset $\Theta\subset \mathbb{Q}\cap[0,1]$ and rational numbers $u,v\in\mathbb{Q}$, where $v>0$. 
 For any $w\in\mathbb{Q}_{>0}$, consider the following set
 $$\mathfrak{G}_{d, \Theta, v,u,w}:=\left\{
\begin{array}{l}
f\colon (X,\Delta,A) \to C
\end{array}
\;\middle|
\begin{array}{rl}
(i)&\text{$f$ is a klt-trivial fibration over a curve}\\
&\text{$C$ such that $K_{X}+\Delta\equiv uf^{*}H$ with a}\\
&\text{line bundle $H$ on $C$ of ${\rm deg}\,H=1$,}\\
(ii)&\text{${\rm dim}X=d$,}\\
(iii)&\text{the coefficients of $\Delta$ belong to $\Theta$,}\\
(iv)&\text{$A$ is an ample $\mathbb{Q}$-Cartier Weil divisor}\\
&\text{on $X$ such that $(H\cdot A^{d-1})=v$ and}\\
& \text{$\mathrm{vol}(A)\le w$.}
\end{array}\right\}.$$
We also fix  $w' \in \mathbb{Q}_{>0}$.  
Then, there exist 
\begin{itemize}
\item
a positive integer $I$, depending only on $d$, $\Theta$, $u$, $v$, and $w$, and 
\item
finitely many polynomials $P_1,\cdots,P_l$, depending only on $d$, $\Theta$, $u$, $v$, $w$, and $w'$, 
\end{itemize}
satisfying the following.
  For any $f\colon(X,\Delta,A)\to C\in \mathfrak{G}_{d, \Theta, v,u,w}$ and nef Cartier divisor $M$ on $X$, 
\begin{itemize}
\item
$IA+M$ is very ample,
\item
$H^j(X,\mathcal{O}_X(m(IA+M)))=0$ for every $j>0$ and $m\in\mathbb{Z}_{>0}$, and 
\item
if $\mathrm{vol}(IA+M)\le w'$, then there is $1\le i\le l$ such that 
$$\chi(X,\mathcal{O}_X(m(IA+M)))=P_i(m)$$
for every $m\in\mathbb{Z}_{>0}$.
\end{itemize}
\end{cor}

Before the proof, we show the following criterion for very ampleness and finiteness of Hilbert polynomials.

\begin{lem}\label{lem--hilbertpolynomial}
Fix $d\in\mathbb{Z}_{>0}$ and $w\in \mathbb{R}_{>0}$. Then there are finitely many polynomials $P_{1},\cdots, P_{l}$, depending only on $d$ and $w$, such that for any $d$-dimensional projective klt pair  $(X,\Delta)$, very ample Cartier divisor $A$ on $X$, and nef Cartier divisor $M$ on $X$, if $A-(K_X+\Delta)$ is nef and big and $\mathrm{vol}((d+2)A+M)\le w$, then
\begin{itemize}
\item
$(d+2)A+M$ is very ample,
\item
$H^j(X, \mathcal{O}_{X}(m((d+2)A+M)))=0$ for every $j>0$ and $m\in\mathbb{Z}_{>0}$, and 
\item
there is $1\leq i \leq l$ such that 
$$\chi(X, \mathcal{O}_{X}(m((d+2)A+M)))=P_{i}(m)$$ 
for every $m\in\mathbb{Z}_{>0}$.
\end{itemize} 
\end{lem}

\begin{proof}
Put $A'=(d+2)A+M$. 
By the Kawamata--Viehweg vanishing theorem, we have
$$H^{j}(X, \mathcal{O}_{X}(mA'-kA))=0$$
for every $m \in \mathbb{Z}_{>0}$, $0\le k\le d+1$, and $j>0$. 
By \cite[Lemma 5.1]{FGA}, $A'-A$ is globally generated, and 
$$\chi(X, \mathcal{O}_{X}(mA'))={\rm dim}\,H^{0}(X,  \mathcal{O}_{X}(mA'))$$
for every $m \in \mathbb{Z}_{>0}$. 
Since $A$ is very ample, so is $A'=(A'-A)+A$.
Furthermore, we can check the following claim:
\begin{claim*}
For every $m \in \mathbb{Z}_{>0}$, we have 
$${\rm dim}\,H^{0}(X,  \mathcal{O}_{X}(mA')) \leq d+m^dw.$$
\end{claim*}
\begin{proof}[Proof of Claim]
Let $Y \sim mA'$ be a general hyperplane section. 
Then we have
$$0\longrightarrow \mathcal{O}_{X} \longrightarrow \mathcal{O}_{X}(mA') \longrightarrow \mathcal{O}_{Y}(mA'|_{Y}) \longrightarrow 0,$$
hence we see that 
$$H^{0}(X,  \mathcal{O}_{X}(mA'))\leq 1+H^{0}(Y,  \mathcal{O}_{Y}(mA'|_{Y}))).$$
This relation implies the case of ${\rm dim}X=1$ in the claim, and the general case follows from the relation and an induction on the dimension of $X$. 
\end{proof}
Therefore, if we put $P(m)=\chi(X, \mathcal{O}_{X}(mA'))$, then there are only finitely many possibilities of $P(1), \cdots, P(d+1)$ depending only on $d$ and $w$. 
In particular,  they do not depend on $M$. 
Lemma \ref{lem--hilbertpolynomial} follows from this fact. 
\end{proof}

\begin{proof}[Proof of Corollary \ref{cor--hilbertpolynomial}]
By Lemma \ref{lem--Cartierindex}, we can find $r$, depending only on $d$, $\Theta$, $v$, and $u$, such that $\Theta=[0,1]\cap\frac{1}{r}\mathbb{Z}$ and $r(K_X+\Delta)$ is Cartier for all $f \colon (X,\Delta,A)\to C\in\mathfrak{G}_{d,\Theta,v,u,w}$.
By Theorem \ref{thm--eff-veryample} and Remark \ref{rem--ell-bound}, there exists $I_0$, depending only on $d$, $\Theta$, $v$, and $u$, such that $I_0A$ is Cartier for all $f \colon (X,\Delta,A)\to C\in\mathfrak{G}_{d,\Theta,v,u,w}$. 
Note that $3dI_0A+K_X+\Delta$ is ample by \cite[Theorem 3.7]{KM}.
Set 
$$A':=3drI_0A+r(K_X+\Delta) \qquad {\rm and} \qquad A'':=A'+3drI_{0}A.$$ 
Then $A'$, $A'-(K_X+\Delta)$, $A''$, $A''-(K_X+\Delta)$ are ample.
By the effective base point freeness \cite[Theorem 1.1]{kollar-eff-basepoint-free} and the effective very ampleness \cite[Lemma 7.1]{fujino-eff-slc}, there exists $I_{1}\in\mathbb{Z}_{>0}$, depending only on $d$, such that $I_{1}A'$ and $I_{1}A''$ very ample.
Now
\begin{equation*}
\begin{split}
{\rm vol}((d+2)I_{1}(A'+A''))=& ((d+2)I_{1})^{d}{\rm vol}(A'+A'')\\
\leq& ((d+2)I_{1})^{d}\bigl((6drI_{0})^{d}w+(6drI_{0})^{d-1}drvu\bigr), 
\end{split}
\end{equation*}
hence Lemma \ref{lem--hilbertpolynomial} implies that there are only finitely many possibilities of the Hilbert polynomials 
$$m\mapsto \chi\bigl(X,\mathcal{O}_{X}\bigl(m((d+2)I_{1}(A'+A''))\bigr)\bigr)$$
for the elements $f \colon (X,\Delta,A)\to C\in\mathfrak{G}_{d,\Theta,v,u,w}$. 
Similarly, Lemma \ref{lem--hilbertpolynomial} implies that there are only finitely many possibilities of the Hilbert polynomials 
$$m\mapsto \chi\bigl(X,\mathcal{O}_{X}\bigl(m((d+2)I_{1}A')\bigr)\bigr)\quad {\rm and \quad} m\mapsto \chi\bigl(X,\mathcal{O}_{X}\bigl(m((d+2)I_{1}A'')\bigr)\bigr)$$
for the elements $f \colon (X,\Delta,A)\to C\in\mathfrak{G}_{d,\Theta,v,u,w}$. 
In particular, there exist positive integers $N_1$ and $N_2$, depending only on $d$, $\Theta$, $u$, $v$, and $w$, such that 
$${\rm dim}\,H^0(X,\mathcal{O}_{X}((d+2)I_1A'))\le N_1\quad {\rm and} \quad {\rm dim}\,H^0(X,\mathcal{O}_{X}((d+2)I_1A''))\le N_2
$$
for every $f \colon (X,\Delta,A)\to C\in\mathfrak{G}_{d,\Theta,v,u,w}$. 
From this fact, there exists a closed immersion $\varphi\colon X\hookrightarrow \mathbb{P}^{N_1}\times\mathbb{P}^{N_2}$ such that 
$$\varphi^*p_1^*\mathcal{O}_{\mathbb{P}^{N_1}}(1)\cong \mathcal{O}_{X}((d+2)I_1A') \quad {\rm and} \quad \varphi^*p_2^*\mathcal{O}_{\mathbb{P}^{N_2}}(1)\cong \mathcal{O}_{X}((d+2)I_1A''),$$
where $p_1\colon\mathbb{P}^{N_1}\times\mathbb{P}^{N_2}\to\mathbb{P}^{N_1}$ and $p_2\colon\mathbb{P}^{N_1}\times\mathbb{P}^{N_2}\to\mathbb{P}^{N_2}$ are the projections. 
By the theory of Hilbert schemes, there exist a scheme $S$ of finite type over $\mathbbm{k}$ and a closed subscheme $\mathcal{X}\subset \mathbb{P}^{N_1}\times\mathbb{P}^{N_2}\times S$, which is flat over $S$, such that for any $f\colon(X,\Delta,A)\to C\in\mathfrak{G}_{d,\Theta,v,u,w}$, there is a closed point $s\in S$ satisfying $X\cong \mathcal{X}_s$ and the condition that the immersion $\mathcal{X}_s \subset \mathbb{P}^{N_1}\times\mathbb{P}^{N_2}$ coincides with $\varphi$. 

By the definitions of $A'$ and $A''$ and shrinking $S$ if necessary, we may assume that
$\bigl(p_1^*\mathcal{O}_{\mathbb{P}^{N_1}}(-1)\otimes p_2^*\mathcal{O}_{\mathbb{P}^{N_2}}(1)\bigr)|_{\mathcal{X}}$ is ample over $S$ (cf.~\cite[Corollary 1.41]{KM}). 
Then there exists a positive integer $N'$, depending only on $d$, $\Theta$, $u$, $v$, and $w$, such that the line bundle
$\bigl(p_1^*\mathcal{O}_{\mathbb{P}^{N_1}}(-N'-2)\otimes p_2^*\mathcal{O}_{\mathbb{P}^{N_2}}(N'+1)\bigr)|_{\mathcal{X}}$ is ample over $S$. 
This fact and the definitions of $A'$ and $A''$ imply that
\begin{equation*}
\begin{split}
-(N'+2)(d+2)I_1A'+(N'+1)(d+2)I_1A'=&(d+2)I_1(-A'+(N'+1)\cdot3drI_{0}A)\\
=&(d+2)I_1(3drI_{0}N'A-r(K_{X}+\Delta))\\
=&(d+2)I_1r(3dI_{0}N'A-(K_{X}+\Delta))
\end{split}
\end{equation*}
is ample for all $f \colon (X,\Delta,A)\to C\in\mathfrak{G}_{d,\Theta,v,u,w}$, therefore $3dI_{0}N'A-(K_{X}+\Delta)$ is ample. 

Recall from the definition of $I_{0}$ (see Theorem \ref{thm--eff-veryample} and Remark \ref{rem--ell-bound}) that $I_{0}A$ is Cartier. 
By the effective base point freeness \cite[Theorem 1.1]{kollar-eff-basepoint-free} and the effective very ampleness \cite[Lemma 7.1]{fujino-eff-slc}, there exists $I_{2}\in\mathbb{Z}_{>0}$, depending only on $d$, such that $I_{2}A$ is very ample for every $f \colon (X,\Delta,A)\to C\in\mathfrak{G}_{d,\Theta,v,u,w}$. 
Now define 
$$I:=(d+2)3dI_{0}N'I_{2},$$ which depends only on $d$, $\Theta$, $u$, $v$, and $w$. 
Then $\frac{1}{d+2}IA$ is a very ample Cartier divisor and $\frac{1}{d+2}IA-(K_{X}+\Delta)$ is ample for all $f \colon (X,\Delta,A)\to C\in\mathfrak{G}_{d,\Theta,v,u,w}$. 
By Lemma \ref{lem--hilbertpolynomial}, we see that this $I$ is the desired positive integer.
\end{proof}

\section{Tools for construction of the moduli spaces}\label{toolsec}

In this section we prove some results to construct the moduli spaces in this paper.

\subsection{Openness}

In this subsection, we prove the openness of uniformly adiabatically K-stable klt-trivial fibrations.

\begin{lem}\label{lem--generic}
Let $\mathcal{X} \to S$ and $\mathcal{Z} \to S$ be flat projective surjective morphisms of normal varieties such that all the geometric fibers of the morphisms are normal and connected. 
Let $f \colon  \mathcal{X} \to \mathcal{Z}$ be a contration over $S$. 
Let $(\mathcal{X},\mathcal{D})$ be a pair such that $K_{\mathcal{X}}+\mathcal{D}\sim_{\mathbb{Q}, \mathcal{Z}}0$, ${\rm Supp}\,\mathcal{D}$ does not contain any fiber of $\mathcal{X} \to S$, and $(\mathcal{X}_{\overline{s}},\mathcal{D}_{\overline{s}})$ is a klt pair for every geometric point $\overline{s}\in S$. 
Let $\overline{\eta} \in S$ be the geometric generic point. 

Then there exists an open subset $U \subset S$ such that for every closed point $t \in U$, the discriminant $\mathbb{Q}$-divisors $\mathcal{B}$, $B_{\overline{\eta}}$, and $B_{t}$ with respect to $f \colon (\mathcal{X},\mathcal{D}) \to \mathcal{Z}$, 
$f_{\overline{\eta}} \colon (\mathcal{X}_{\overline{\eta}},\mathcal{D}_{\overline{\eta}}) \to \mathcal{Z}_{\overline{\eta}}$, and $f_{t} \colon (\mathcal{X}_{t},\mathcal{D}_{t}) \to \mathcal{Z}_{t}$ respectively satisfy
$$\underset{\overline{P}}{{\rm max}}\,{\rm coeff}_{\overline{P}}(B_{\overline{\eta}})=\underset{P}{{\rm max}}\,{\rm coeff}_{P}(\mathcal{B}|_{\mathcal{Z}\times_{S}U})=\underset{P'}{{\rm max}}\,{\rm coeff}_{P'}(B_{t}),$$ 
where $\overline{P}$ (resp.~$P$, $P'$) runs over prime divisors on $\mathcal{Z}_{\overline{\eta}}$ (resp.~$\mathcal{Z}\times_{S}U$, $\mathcal{Z}_{t}$). 
Furthermore, $\mathcal{B}|_{\mathcal{Z}_t}$ is well-defined and $\mathcal{B}|_{\mathcal{Z}_t}=B_t$ for any closed point $t\in U$.
\end{lem} 

\begin{proof}
First, note that we may shrink $S$ whenever we focus on an open subset of $S$. 
Moreover, as \cite[Lemma 5.1]{A}, we see that $\underset{P}{\rm max}\,{\rm coeff}_{P}(\mathcal{B}|_{\mathcal{Z}\times_{S}U})$ is not changed for any $U\subset S$ even if we replace $(\mathcal{X},\mathcal{D})\to \mathcal{Z} \to S$ with the base change by any \'etale surjective morphism $S' \to S$. 
Thus, in the rest of the proof, we will freely shrink $S$ and take the base change of $(\mathcal{X},\mathcal{D})\to \mathcal{Z} \to S$ by an \'etale surjective morphism if necessary. 

By shrinking $S$, we may assume that $S$ is smooth, ${\rm Supp}\,\mathcal{B}$ does not contain any fiber of $\mathcal{Z} \to S$ and the codimension of ${\rm Sing}(\mathcal{Z})\cap \mathcal{Z}_{s}$ in $\mathcal{Z}_{s}$ is at least two for every $s \in S$. 
In particular, we can define $\mathcal{B}_{\overline{\eta}}$, and we can also define $\mathcal{B}_{s}$ for every closed point $s \in S$. 

In this paragraph, we show the first equality of Lemma \ref{lem--generic}. 
We denote the morphism $\overline{\eta} \to S$ by $\tau$. 
By shrinking $S$, we can find a finite morphism $\varphi  \colon  S' \to S$ and a morphism $\psi  \colon  \overline{\eta} \to S'$ such that $ \tau=\varphi \circ \psi$ and
for any component $\overline{Q}$ of $B_{\overline{\eta}}$, there is a prime divisor $Q'$ on $\mathcal{Z}\times_{S}S'$ whose pullback to $\mathcal{Z}_{\overline{\eta}}$ is $\overline{Q}$. 
By shrinking $S$, we may assume that $\varphi$ is \'etale.
By replacing $(\mathcal{X},\mathcal{D})\to \mathcal{Z} \to S$ with the base change by $\varphi$, we may assume that for any component $\overline{Q}$ of $B_{\overline{\eta}}$, there is a prime divisor $Q$ on $S$ such that $\tau^{*}Q=\overline{Q}$. 
Let $B_{\overline{\eta}}$ and $\mathcal{B}$ be $\mathbb{Q}$-divisors as in Lemma \ref{lem--generic}.
By shrinking $S$ and replacing $(\mathcal{X},\mathcal{D})\to \mathcal{Z} \to S$ with an \'etale base change, we may assume $\tau^{*}(K_{\mathcal{Z}}+\mathcal{B})=K_{\mathcal{Z}_{\overline{\eta}}}+B_{\overline{\eta}}$. 
Then $\mathcal{B}_{\overline{\eta}}=B_{\overline{\eta}}$. 
Shrinking $S$, we may further assume that any component of $\mathcal{B}$ dominates $S$. 
Then
$$\underset{P}{{\rm max}}\,{\rm coeff}_{P}(\mathcal{B})=\underset{\overline{P}}{{\rm max}}\,{\rm coeff}_{\overline{P}}(B_{\overline{\eta}}),$$
where $P$ (resp.~$\overline{P}$) runs over prime divisors on $\mathcal{Z}$ (resp.~$\mathcal{Z}_{\overline{\eta}}$). 

From now on, we show the second equality of Lemma \ref{lem--generic}. 
We can construct a diagram 
$$
\xymatrix{
\mathcal{Y}\ar[r]^{g}\ar[d]_{f'}&\mathcal{X}\ar[d]^{f}\\
\mathcal{W}\ar[r]_{h}&\mathcal{Z}
}
$$
of projective morphisms, where $\mathcal{Y}$ and $\mathcal{W}$ are smooth varieties, and snc divisors $\Sigma$ on $\mathcal{W}$ and $\Xi$ on $\mathcal{Y}$
such that 
\begin{itemize}
\item
$h$ is birational and $g$ is a log resolution of $(\mathcal{X},\mathcal{D})$, 
\item
$f'$ is a contraction, 
\item
$\Xi\supset f'^{*}\Sigma \cup g_{*}^{-1}\mathcal{D}\cup {\rm Ex}(g)$ and the vertical part of $\Xi$ with respect to $f'$ maps into $\Sigma$, and
\item
$(\mathcal{Y},\Xi)$ is log smooth over $\mathcal{W}\setminus \Sigma$, in other words, the restriction of $f' \colon (\mathcal{Y},\Xi) \to \mathcal{W}$ over $\mathcal{W}\setminus \Sigma$ is log smooth. 
\end{itemize}
By shrinking $S$, we may assume that for every closed point $t \in S$, the restricted diagram 
$$
\xymatrix{
(\mathcal{Y}_{t},\Xi_{t})\ar[r]^{g_{t}}\ar[d]_{f'_{t}}&(\mathcal{X}_{t},\mathcal{D}_{t})\ar[d]^{f_{t}}\\
(\mathcal{W}_{t},\Sigma_{t})\ar[r]_{h_{t}}&\mathcal{Z}_{t},
}
$$
satisfies the same conditions as stated above. 
 We define $\mathcal{D}_{\mathcal{Y}}$ by $K_{\mathcal{Y}}+\mathcal{D}_{\mathcal{Y}}=g^{*}(K_{\mathcal{X}}+\mathcal{D})$ and $g_*\mathcal{D}_{\mathcal{Y}}=\mathcal{D}$. 
Let $\Gamma$ be the discriminant $\mathbb{Q}$-divisor with respect to $f' \colon (\mathcal{Y},\mathcal{D}_{\mathcal{Y}}) \to \mathcal{W}$. For each closed point $t \in S$, let $G_{t}$ be the discriminant $\mathbb{Q}$-divisor with respect to $f'_{t} \colon (\mathcal{Y}_{t},\mathcal{D}_{\mathcal{Y}_{t}}) \to \mathcal{W}_{t}$. 
Then $h_{*}\Gamma=\mathcal{B}$ and $h_{t*}G_{t}=B_{t}$, where $B_{t}$ is the discriminant $\mathbb{Q}$-divisor with respect to  the klt-trivial fibration $f_{t}  \colon (\mathcal{X}_{t},\mathcal{D}_{t}) \to \mathcal{Z}_{t}$. 
Shrinking $S$, we may assume $\Gamma_{t}=G_{t}$ for every closed point $t \in S$. 
Then $\mathcal{B}_{t}=h_{t*}\Gamma_{t}=h_{t*}G_{t}=B_{t}$, and the snc condition of $\Sigma_{t}$ implies that
$$\underset{P}{{\rm max}}\,{\rm coeff}_{P}(\mathcal{B})=\underset{P'}{{\rm max}}\,{\rm coeff}_{P'}(B_{t}),$$
where $P$ (resp.~$P'$) runs over prime divisors on $\mathcal{Z}$ (resp.~$\mathcal{Z}_{t}$). 

By the above discussion, Lemma \ref{lem--generic} holds. 
\end{proof}

\begin{thm}[Openness of uniform adiabatic K-stability]\label{op}
Let $S$ be a normal variety, $\pi \colon  (\mathcal{X},\mathcal{D}) \to S$ a log $\mathbb{Q}$-Gorenstein family, and let $f \colon \mathcal{X}\to \mathcal{P}$ be a contraction over $S$, where $\mathcal{P}$ is a scheme projective and smooth over $S$. 
Let $\mathcal{H}$ be an $f$-ample $\mathbb{Q}$-divisor on $\mathcal{X}$, and let $\mathcal{L}$ be a Cartier divisor on $\mathcal{P}$. Suppose that there exists an integer $m>0$ such that $(\mathcal{P}_{\bar{s}},\mathcal{L}_{\bar{s}})=(\mathbb{P}^1,\mathcal{O}(m))$ for any geometric point $\bar{s}\in S$. 
Assume that $-(K_{\mathcal{X}/S}+\mathcal{D})\sim_{\mathbb{Q},S}\frac{u}{m}f^*\mathcal{L}$ for some $u\in\mathbb{Q}_{>0}$ and all the geometric fibers of $\pi$ are klt.

Then the function
\begin{equation*}
h \colon S\ni s\mapsto \underset{P_{\overline{s}}}{{\rm max}}\,{\rm coeff}_{P_{\overline{s}}}(B_{\overline{s}}),
\end{equation*}
where $\overline{P}_{\overline{s}}$ runs over prime divisors on $\mathbb{P}_{\overline{s}}^{1}$, is constructible and upper semi-continuous. 
In particular, the subset 
$$W=\{s\in S\,|\,\text{$f_{\overline{s}} \colon (\mathcal{X}_{\overline{s}},\mathcal{D}_{\overline{s}},\mathcal{H}_{\overline{s}})\to \mathcal{P}_{\overline{s}}$ is uniformly adiabatically K-stable}\}$$ is open and there exists a positive real number $v$ such that 
$$\delta_{(\mathbb{P}^{1},B_{\overline{s}})}(-K_{\mathbb{P}^1}-B_{\overline{s}}-M_{\overline{s}})\geq 1+v$$
 for every geometric point $\overline{s} \in W$, where $B_{\overline{s}}$ and $M_{\overline{s}}$ are the discriminant $\mathbb{Q}$-divisor and the moduli $\mathbb{Q}$-divisor with respect to $f_{\overline{s}}$, respectively.
\end{thm}

\begin{proof}
We first reduce Theorem \ref{op} to the case where $\mathcal{P}\cong\mathbb{P}^1_S$, $m=1$, and $\mathcal{L}=\mathcal{O}_{\mathbb{P}^1_S}(1)$. 
For every closed point $s\in S$, there exists an \'{e}tale morphism $g^s\colon T^s\to S$ such that $s\in g^s(T^s)$ and $\mathcal{P}_{T^s}\to T^s$ has a section $\iota^s\colon T^s\to \mathcal{P}_{T^s}$ (cf.~\cite[Corollary 1.3.10]{Ols}).
By considering $T=\sqcup_{s_i}T^{s_i}$ for some finitely many closed points $s_i\in S$, we obtain an \'{e}tale surjective morphism $g\colon T\to S$ such that $h\colon\mathcal{P}_{T}\to T$ has a section $\iota\colon T\to \mathcal{P}_{T}$.
Then $\iota(T)$ is a Cartier divisor on $\mathcal{P}_T$ (\cite[Lemma 9.3.4]{FGA}, \cite[Definition-Lemma 4.20]{kollar-moduli}). 
By \cite[III, Corollary 12.9]{Ha}, the sheaf $h_*\mathcal{O}_{\mathcal{P}_T}(\iota(T))$ is
locally free of rank two and
\[
h^{*}h_{*}\mathcal{O}_{\mathcal{P}}(\iota(T))\longrightarrow H^0(\mathcal{P}_t,\mathcal{O}_{\mathcal{P}_t}(\iota(T)|_{\mathcal{P}_t}))
\]
is surjective. 
Therefore, we obtain a morphism $\mathcal{P}_T\to \mathbb{P}_T(h_*\mathcal{O}_{\mathcal{P}_T}(\iota(T)))$. 
Then the right hand side is a $\mathbb{P}^{1}$-bundle, and the morphism is an isomorphism. 
Since $g$ is open and surjective, if Theorem \ref{op} holds for $T$, then Theorem \ref{op} also holds for $S$.
Thus, we may assume that $\mathcal{P}$ is a $\mathbb{P}^1$-bundle over $S$. 
Since the problem is local, we may assume that $\mathcal{P}=\mathbb{P}^1_S$ by shrinking $S$. 
Then $\mathcal{L}\sim_S\mathcal{O}_{\mathbb{P}^1_S}(m)$. 
It is easy to see that we may replace $\mathcal{L}$ by $\mathcal{O}_{\mathbb{P}^1_S}(m)$. 
In this way, we may assume that $\mathcal{L}=\mathcal{O}_{\mathbb{P}^1_S}(m)$. By replacing $u$ with $\frac{u}{m}$, we may assume $m=1$. 

Next, we show that $h$ is constructible. By \cite[(6,C)]{Mat}, it suffices to show that $h^{-1}(w)$ contains a non-empty open subset of $S$ under the assumption that $S$ is a variety and that $h^{-1}(w)$ is dense for every $w\in\mathbb{Q}_{>0}$. We pick an open subset $V \subset S\setminus {\rm Sing}(S)$. 
Since the fibers of $\pi\colon \mathcal{X} \to S$ are normal, we have 
$K_{\pi^{-1}(V)/V}=K_{\pi^{-1}(V)}-(\pi|_{\pi^{-1}(V)})^{*}K_{V}$. 
Thus, $K_{\pi^{-1}(V)}+\mathcal{D}|_{\pi^{-1}(V)}$ is $\mathbb{Q}$-Cartier. 
Let $\eta$ be the generic point of $S$. 
By Lemma \ref{lem--generic} and shrinking $V$ if necessary, we may assume 
$$\underset{\overline{P}}{{\rm max}}\,{\rm coeff}_{\overline{P}}(B_{\overline{\eta}})=\underset{P}{{\rm max}}\,{\rm coeff}_{P}(B_{t})$$
for every closed point $t \in V$, where $\overline{P}$ (resp.~$P$) runs over prime divisors on $\mathbb{P}_{\overline{\eta}}^{1}$ (resp.~$\mathbb{P}_{t}^{1}$). 
For any point $s \in V$, by applying Lemma \ref{lem--generic} to $\overline{\{s\}}\cap V$, we see that 
$\underset{P_{\overline{s}}}{{\rm max}}\,{\rm coeff}_{P_{\overline{s}}}(B_{\overline{s}})$ are determined by $f_{t} \colon (\mathcal{X}_{t},\mathcal{D}_{t})\to \mathbb{P}_{t}^{1}$ for general closed points $t \in \overline{\{s\}}\cap V$. 
This means that $h$ is constant on $V$. Thus the constructibility holds. 

From now on we prove the upper semi-continuity. 
The constructibility of $h$ implies that $h$ takes only finitely many values. 
We fix $w \in \mathbb{Q}_{>0}$. 
By Lemma \ref{const--lem}, we may assume that $S$ is a curve and $h(s) \geq w$ for  every general point $s \in S$. 
Then $S$ is smooth, and hence we may write $K_{\mathcal{X}/S}=K_{\mathcal{X}}-\pi^{*}K_{S}$. 
Thus $K_{\mathcal{X}}+\mathcal{D}$ is $\mathbb{Q}$-Cartier. 
Let $\mathcal{B}$ be the discriminant $\mathbb{Q}$-divisor with respect to $f \colon (\mathcal{X},\mathcal{D}) \to \mathbb{P}^{1}_{S}$. 
By Lemma \ref{lem--generic}, we can find an open subset $V\subset S$ such that
\begin{equation*}
\underset{Q}{{\rm max}}\,{\rm coeff}_{Q}(\mathcal{B})=\underset{P}{{\rm max}}\,{\rm coeff}_{P}(B_{t})
\end{equation*}
for every closed point $t\in V$, where $Q$ (resp.~$P$) runs over prime divisors on $\mathbb{P}^{1}_{S}$ (resp.~$\mathbb{P}_{t}^{1}$). 
Since $h(s) \geq w$ for general points $s \in S$, we have
$$\underset{Q}{{\rm max}}\,{\rm coeff}_{Q}(\mathcal{B})\geq u.$$ 
In our situation, the klt property of the geometric fibers of $\pi$ and the inversion of adjunction \cite{kawakita} imply that $(\mathcal{X},\mathcal{D}+\mathcal{X}_{s})$ is lc for every closed point $s \in S$. Therefore, every component of $\mathcal{B}$ is horizontal over $S$. 
By our assumption, there exists a component $T$ of $\mathcal{B}$ such that 
$${\rm coeff}_{T}(\mathcal{B})=\underset{Q}{{\rm max}}\,{\rm coeff}_{Q}(\mathcal{B})\geq u.$$
Since every component of $\mathcal{B}$ dominates $S$, by \cite[Lemma 5.1]{A}, this fact is preserved even if we take any finite base change of $(\mathcal{X},\mathcal{D})\to \mathbb{P}^{1}_{S}\to S$. 
Let $\psi \colon T^{\nu} \to S$ be the natural morphism, where $T^{\nu}$ is the normalization of $T$. 
We consider the base change of $(\mathcal{X},\mathcal{D})\to \mathbb{P}^{1}_{S}\to S$ by $\psi$, which we denote by $(\mathcal{X}_{T^{\nu}},\mathcal{D}_{T^{\nu}})\to \mathbb{P}^{1}_{T^{\nu}}\to T^{\nu}$, with the morphism $\psi_{\mathbb{P}^{1}} \colon  \mathbb{P}^{1}_{T^{\nu}} \to \mathbb{P}^{1}_{S}$. 
By construction, $\psi_{\mathbb{P}^{1}}^{*}T$ has a component isomorphic to $T^{\nu}$. 
Since we only need to deal with closed points of $S$, we may replace $(\mathcal{X},\mathcal{D})\to \mathbb{P}^{1}_{S}\to S$ with $(\mathcal{X}_{T^{\nu}},\mathcal{D}_{T^{\nu}})\to \mathbb{P}^{1}_{T^{\nu}}\to T^{\nu}$. 
By this replacement, we may assume that $T\to S$ is an isomorphism. 
We put $\gamma= 1-{\rm coeff}_{T}(\mathcal{B})$. 
Then there is a prime divisor $\mathcal{E}$ over $\mathcal{X}$ such that $\mathcal{E}$ maps onto $T$ and $A_{(\mathcal{X},\mathcal{D}+u'f^{*}T)}(\mathcal{E})<0$ for any real number $u'>\gamma$.  
Since $T$ dominates $S$, for every closed point $c \in S$, the pair $(\mathcal{X}, \mathcal{D}+u'f^{*}T+\mathcal{X}_{c})$ is not lc around $\mathcal{X}_{c}$. 
By the inversion of adjunction \cite{kawakita}, the pair $(\mathcal{X}_{c}, \mathcal{D}_{c}+u'f_{c}^{*}T|_{\mathbb{P}_{c}^{1}})$ is not lc for any $u'>\gamma$. 
Since $\mathbb{P}^{1}_{S}$ is smooth and $T \to S$ is an isomorphism, $T|_{\mathbb{P}_{c}^{1}}$ is a prime divisor on $\mathbb{P}_{c}^{1}$. 
Thus, the discriminant $\mathbb{Q}$-divisor $B_{c}$ with respect to $f_{c} \colon (\mathcal{X}_{c},\mathcal{D}_{c}) \to \mathbb{P}_{c}^{1}$ has a component whose coefficient is at least $1-\gamma$.
This shows that 
for every closed point $c \in S$, 
$$u \leq\underset{Q}{{\rm max}}\,{\rm coeff}_{Q}(\mathcal{B})=1-\gamma \leq \underset{P'}{{\rm max}}\,{\rm coeff}_{P'}(B_{c}),$$
where $P'$ runs over prime divisors on $\mathbb{P}^{1}_{c}$.
Thus the upper semi-continuity of $h$ holds. 
The final statement of Theorem \ref{op} follows from this fact and Example \ref{ex--calculation}. 
\end{proof}

\subsection{Separatedness}

In this subsection we show the separatedness of the moduli spaces that we will construct in \S\ref{Consec}. 

\begin{note}
Let $C$ be an affine curve. We say that $f \colon (X,\Delta,L)\to C$ is a {\em polarized $\mathbb{Q}$-Gorenstein family} if $f \colon (X,\Delta)\to C$ is a log $\mathbb{Q}$-Gorenstein family over $C$ and $L$ is an $f$-ample line bundle. 
Let $0\in C$ be a closed point and $C^\circ=C\setminus \{0\}$ the punctured curve. We put 
$$(X,\Delta,L)\times_CC^\circ=(X\times_CC^\circ,\Delta\times_CC^\circ,L|_{X\times_CC^\circ}).$$ 
For another polarized $\mathbb{Q}$-Gorenstein family $f' \colon (X',\Delta',L')\to C$, we define 
$$g \colon (X,\Delta,L)\to(X',\Delta',L')$$
 to be a $C$-isomorphism $g \colon X\to X'$ such that $f'\circ g=f$, $g_* \Delta =\Delta'$ and $g^*L'\sim_{C}L$. We define $C^\circ$-isomorphisms between $(X,\Delta,L)\times_CC^\circ$ and $(X',\Delta',L')\times_CC^\circ$ similarly.

Let $f \colon (X,\Delta,L)\to C$ and $f' \colon (X',\Delta',L')\to C$ be polarized $\mathbb{Q}$-Gorenstein families. For contractions $\pi \colon (X,\Delta,L)\to (\mathbb{P}^1_{C},\mathcal{O}(1))$ and $\pi' \colon (X',\Delta',L')\to (\mathbb{P}^1_{C},\mathcal{O}(1))$ over $C$, we define $(\alpha,\beta) \colon \pi\to\pi'$ as a pair of $C$-isomorphisms $\alpha \colon (X,\Delta,L)\to(X',\Delta',L')$ and $\beta \colon (\mathbb{P}^1_{C},\mathcal{O}(1))\to(\mathbb{P}^1_{C},\mathcal{O}(1))$ such that $\pi'\circ\alpha=\beta\circ\pi$. 
\end{note}

The following has been already known by Boucksom when $\Delta=0$, but we write here the proof for the sake of completeness. 
We also note that Odaka \cite{O4} proved that the following holds for  polarized klt varieties such
that with numerically trivial canonical divisor is K-stable.

\begin{prop}[cf.~{\cite[Theorem 1.1]{B}}, {\cite[Theorem 3.1 and Remark 3.6]{BX}}]\label{plussep}
Let $C$ be an affine curve. 
Let $f \colon (X,\Delta,L)\to C$ and $f'\colon (X',\Delta',L')\to C$ be two polarized $\mathbb{Q}$-Gorenstein families. 
Suppose that there exists a 
$C^\circ$-isomorphism $$g^\circ \colon (X,\Delta,L)\times_CC^\circ\to(X',\Delta',L')\times_CC^\circ$$
 and both $K_X+\Delta$ and $K_{X'}+\Delta'$ are nef over $C$. 
If $(X_0,\Delta_0)$ is klt and $(X'_0,\Delta'_0)$ is lc, then $g^\circ$ can be extended to a $C$-isomorphism $g \colon (X,\Delta,L)\to(X',\Delta',L')$.
\end{prop}

\begin{proof}
Let $g \colon X \dashrightarrow X'$ be the birational map induced by $g^{\circ}$. 
It is sufficient to prove that $g$ is a $C$-isomorphism.

We first show that $g$ and $g^{-1}$ do not contract any divisor. We apply the argument in \cite{B}.
By the inversion of adjunction \cite{kawakita} and shrinking $C$ around $0 \in C$, we may assume that $(X,\Delta+X_0)$ is plt and $(X',\Delta'+X'_0)$ is lc. Take a common log resolution $\pi \colon Y\to X$ and $\pi' \colon Y\to X'$ of $g$. 
By construction, $g^{\circ}\circ \pi|_{Y\times_CC^\circ}$ coincides with $\pi'|_{Y\times_CC^\circ}$. 
Let $\Gamma$ be the sum of $\pi|_{Y\times_CC^\circ}$-exceptional prime divisors, and let $\overline{\Gamma}$ be the closure in $Y$. 
Then $\overline{\Gamma}$ is $\pi$-exceptional and also $\pi'$-exceptional.
By the log canonicity of $(X,\Delta+X_0)$, the $\mathbb{Q}$-divisor
$$E :=K_{Y}+\pi_*^{-1}\Delta+\overline{\Gamma}+Y_{0,\mathrm{red}}-\pi^*(K_{X}+\Delta+X_0)$$
is effective and $\pi$-exceptional. 
Similarly, we see that 
$$E' :=K_{Y}+\pi_*^{-1}\Delta+\overline{\Gamma}+Y_{0,\mathrm{red}}-\pi'^*(K_{X'}+\Delta'+X'_0)$$
is effective and $\pi'$-exceptional. 
Since $K_X+\Delta$ is nef over $C$, by applying the negativity lemma to 
$\pi'$ and $E-E'$, we see that $E-E'$ is effective. 
Similarly, we see that $E'-E$ is effective. 
These facts imply that $E=E'$ and hence  
$$\pi^*(K_{X}+\Delta+X_0)=\pi'^*(K_{X'}+\Delta'+X'_0).$$
Therefore, $A_{(X,\Delta+X_0)}(F)=A_{(X',\Delta'+X'_0)}(F)$ for every prime divisor $F$ on $Y$. 
Recalling that $(X,\Delta+X_0)$ is plt, we see that $A_{(X,\Delta+X_0)}(F)=0$ if and only if $F=\pi^{-1}_{*}X_{0}$. 
Now
$$A_{(X,\Delta+X_0)}(\pi'^{-1}_{*}X'_{0})=A_{(X',\Delta'+X'_0)}(\pi'^{-1}_{*}X'_{0})=1-{\rm coeff}_{X'_{0}}(\Delta'+X'_{0})=0.$$
From these facts, we have $\pi_{*}\pi'^{-1}_{*}X'_{0}=X_{0}$. 
Since $X_{0}$ (resp.~$X'_{0}$) is the fiber of $f$ (resp.~$f'$) over $0 \in C$, we see that $g$ and $g^{-1}$ do not contract any divisor.

We now prove that $g$ is a $C$-isomorphism. 
Consider $L'$ as a Cartier divisor on $X'$, and put $D=g^{-1}_{*}L'$. 
By our hypothesis, we have $L|_{X\times_{C} C^{\circ}}\sim_{C^{\circ}}(g^{\circ})^{*}L|_{X'\times_{C} C^{\circ}}$. 
Since $g$ does not contract any divisor, we have $D\sim_{C}L$ (cf.~\cite[II, Proposition 6.5]{Ha}). Thus, by Serre's $S_2$-condition, $g$ induces
\[
X=\boldsymbol{\mathrm{Proj}}_C(\bigoplus_{m\ge0}f_*L^{\otimes m})=\boldsymbol{\mathrm{Proj}}_C(\bigoplus_{m\ge0}f_*\mathcal{O}_X(mD))\cong \boldsymbol{\mathrm{Proj}}_C(\bigoplus_{m\ge0}f'_*L'^{\otimes m})=X'.
\]
This shows that $g$ is indeed a $C$-isomorphism.
\end{proof}

\begin{cor}\label{plusfin}
Let $(X,\Delta,L)$ be a polarized klt pair such that $K_X+\Delta$ is nef. Then $\mathrm{Aut}(X,\Delta,L)$ is finite.
\end{cor}

\begin{proof}
It follows from Proposition \ref{plussep} as \cite[Corollary 3.5]{BX}.
\end{proof}

We are ready to prove the main theorem of this subsection.
 
\begin{thm}[Separatedness]\label{sep2}
 Let $\pi\colon (X,\Delta,H)\to C$ and $\pi'\colon (X',\Delta',H')\to C$ be two polarized $\mathbb{Q}$-Gorenstein families over a curve such that $(X_{0},\Delta_{0})$ is klt and $(X'_{0},\Delta'_{0})$ is lc.
Let  $g\colon (X,\Delta,H)\to(\mathbb{P}^1_{C},\mathcal{O}(1))$ and $g'\colon (X',\Delta',H')\to(\mathbb{P}^1_{C},\mathcal{O}(1))$ be contractions over $C$ such that $K_{X}+\Delta \sim_{\mathbb{Q},\mathbb{P}^{1}_{C}}0$ and $K_{X'}+\Delta' \sim_{\mathbb{Q},\mathbb{P}^{1}_{C}}0$. 
Let $0\in C$ be a closed point, and let $g_{0}\colon  X_{0} \to \mathbb{P}^{1}$ (resp.~$g'_{0}\colon X'_{0} \to \mathbb{P}^{1}$) be the restriction of $g$ (resp.~$g'$) to $0 \in C$.
Suppose that there exists an isomorphism $(\alpha^{\circ},\beta^\circ)\colon g|_{X\times_CC^{\circ}}\cong g'|_{X'\times_CC^{\circ}}$ over $C^{\circ}$ satisfying the following. 
 \begin{itemize}
 \item $g_0\colon(X_0,\Delta_0,H_0)\to \mathbb{P}^1$ is uniformly adiabatically K-stable,
 \item for the discriminant $\mathbb{Q}$-divisor $B'_0$ and the moduli $\mathbb{Q}$-divisor $M'_0$ with respect to $g'_0\colon(X'_0,\Delta'_0)\to\mathbb{P}^1$, we have $\delta_{(\mathbb{P}^1,B'_0)}(-K_{\mathbb{P}^1}-B'_{0}-M'_{0})\ge1$, and
\item  we have $-K_{X}-\Delta \sim_{C,\,\mathbb{Q}} wg^{*}\mathcal{O}_{\mathbb{P}^1_{C}}(1)$ and $-K_{X'}-\Delta' \sim_{C,\,\mathbb{Q}} w'g'^{*}\mathcal{O}_{\mathbb{P}^1_{C}}(1)$ for some positive rational numbers $w$ and $w'$.
 \end{itemize} 
Then $(\alpha^{\circ},\beta^\circ)$ can be extended to an isomorphism $(\alpha,\beta)\colon g\to g'$ over $C$.
\end{thm}

\begin{proof}
We will reduce the theorem to Proposition \ref{plussep} as \cite[Theorem 3.1]{BX}. By our hypothesis of $(\alpha^{\circ},\beta^\circ)$, it is easy to see that $w=w'$. Denote the bases of $g$ and $g'$ by $\mathcal{C}$ and $\mathcal{C}'$, respectively. 
Note that $\mathcal{C}$ and $\mathcal{C}'$ are isomorphic to $\mathbb{P}^1_{C}$. 
Let $\mathcal{L}$ (resp.~$\mathcal{L}'$) be the line bundle on $\mathcal{C}$ (resp.~$\mathcal{C}'$) isomorphic to $\mathcal{O}(1)$. 
We denote the structure morphisms $\mathcal{C}\to C$ and $\mathcal{C}'\to C$ by $\eta$ and $\eta'$, respectively. 

Replacing $\mathcal{L}$ by $\mathcal{L}+d\mathcal{C}_0$ for some sufficiently large $d\in\mathbb{Z}_{>0}$, we may assume that the birational map $(\mathcal{C},\mathcal{L})\dashrightarrow (\mathcal{C}',\mathcal{L}')$ over $C$ induces an inclusion $\eta'_*\mathcal{L}'\subset \eta_*\mathcal{L}$ as sheaves of $\mathcal{O}_{C}$-modules. 
Now  $\mathcal{O}_{C,0}$ is a divisorial valuation ring and both $\eta_*\mathcal{L}\otimes\mathcal{O}_{C,0}$ and $\eta'_*\mathcal{L}'\otimes\mathcal{O}_{C,0}$ are free $\mathcal{O}_{C,0}$-modules of rank two. 
Hence, by fixing a generator $t$ of the maximal ideal of $\mathcal{O}_{C,0}$, we may find free bases $\{s,u\}\subset\eta_*\mathcal{L}\otimes\mathcal{O}_{C,0}$ and $\{s',u'\}\subset\eta'_*\mathcal{L}'\otimes\mathcal{O}_{C,0}$ such that $s'=t^{\lambda}s$ and $u'=t^\mu u$ for some $\lambda,\mu\in\mathbb{Z}_{\ge0}$. 

By shrinking $C$ around $0$, we may assume that $t$, $s$, $u$, $s'$, and $u'$ are global sections. 
By definition, $s$ and $u$ correspond to prime divisors $E_{1}$ and $E_{2}$ on $\mathcal{C}$ respectively such that $E_{1}|_{\mathcal{C}_{0}} \neq E_{2}|_{\mathcal{C}_{0}}$. 
Similarly, $s'$ and $u'$ correspond to prime divisors $E'_{1}$ and $E'_{2}$ on $\mathcal{C}'$ respectively such that $E'_{1}|_{\mathcal{C}'_{0}} \neq E'_{2}|_{\mathcal{C}'_{0}}$. 
We put $D=\frac{w}{2}(E_{1}+E_{2})$ and $D'=\frac{w}{2}(E'_{1}+E'_{2})$. Then $D$ is the strict transform of $D'$ since they coincide over $C^\circ$. Note that $(X_0,\Delta_0+g_0^*D_0)$ is klt and $(X'_0,\Delta'_0+g_0'^*D'_0)$ is lc. 
Indeed, $(\mathbb{P}^1,B'_0+D'_0)$ is lc since 
\[
\alpha_{(\mathbb{P}^1,B'_0)}(-K_{\mathbb{P}^1}-B'_0-M'_0)=\frac{1}{2}\delta_{(\mathbb{P}^1,B'_0)}(-K_{\mathbb{P}^1}-B'_0-M'_0)\ge\frac{1}{2}
\]
by Example \ref{ex--calculation} and the second assumption of Theorem \ref{sep2}.
Then the log canonicity of $(X'_0,\Delta'_0+g_0'^*D'_0)$ follows from the log canonicity of $(\mathbb{P}^1,B'_0+D'_0)$ in the same way as \cite[Theorem 3.1]{A}. Similarly,  we have
\[
\alpha_{(\mathbb{P}^1,B_0)}(-K_{\mathbb{P}^1}-B_0-M_0)=\frac{1}{2}\delta_{(\mathbb{P}^1,B_0)}(-K_{\mathbb{P}^1}-B_0-M_0)>\frac{1}{2}
\]
by Example \ref{ex--calculation} and the first assumption of Theorem \ref{sep2},  where $B_0$ (resp.~$M_0$) is the discriminant $\mathbb{Q}$-divisor (resp.~moduli $\mathbb{Q}$-divisor) with respect to the klt-trivial fibration $g_0$. 
Thus, we see that $(X_0,\Delta_0+g_0^*D_0)$ is klt in the same way.
We also have the relations $K_X+\Delta+g^*D\sim_{C,\mathbb{Q}}0$ and $K_{X'}+\Delta'+g'^*D'\sim_{C,\mathbb{Q}}0$. Thus, we have a unique extension 
$$\alpha \colon (X,\Delta+g^*D,H)\cong(X',\Delta'+g'^*D',H')$$
of $\alpha^\circ$ over $C$ by Proposition \ref{plussep}. Here, $\alpha^*g'^*D'=g^*D$. Thus, $\alpha^*g'^* E'_{1}=g^* E_{1}$ and $\alpha^*g'^* E'_{2}=g^* E_{2}$ hold and they generate the pencils defining the two contractions $g'\circ \alpha$ and $g$. Therefore, we have a natural extension $\beta \colon (\mathcal{C},\mathcal{L})\cong(\mathcal{C}',\mathcal{L}')$ of $\beta^\circ$. It is easy to see that $(\alpha,\beta)$ is an isomorphism from $g$ to $g'$ over $C$.
\end{proof}

To construct our moduli spaces, we only need 
Theorem \ref{sep2} for uniformly adiabatically K-stable $g_0'$.
In this case, Theorem \ref{sep2} follows from \cite[Corollary 3.22]{CM} when $\mathbbm{k}=\mathbb{C}$. 
Since we do not know whether $(X'_0,\Delta'_0,\epsilon H'_0+g'^*_0\mathcal{O}(1))$ is specially K-semistable or not, we cannot apply \cite[Corollary 3.22]{CM} directly.
Theorem \ref{sep2} is applicable to the case when $g_0'$ is an adiabatically K-semistable klt-trivial fibration (cf.~\cite[Theorem A]{Hat}).

The following is also important for construction of our moduli spaces.

 \begin{cor}[Finiteness of stabilizers]\label{fin2}
 Let $f \colon (X,\Delta,H)\to \mathbb{P}^1$ be a polarized uniformly adiabatically K-stable klt-trivial fibration such that $-(K_{X}+\Delta)$ is nef and not numerically trivial. Then $\mathrm{Aut}\,(f \colon (X,\Delta,H)\to(\mathbb{P}^1,\mathcal{O}_{\mathbb{P}^1}(1)))$ is a finite group.
  \end{cor}
 
 \begin{proof}
Since $\mathrm{Aut}\,(f \colon (X,\Delta,H)\to(\mathbb{P}^1,\mathcal{O}_{\mathbb{P}^1}(1)))$ is represented by a closed subgroup of a linear algebraic group $\mathrm{Aut}\,(X,\Delta,H)\times\mathrm{Aut}\, (\mathbb{P}^1,\mathcal{O}_{\mathbb{P}^1}(1))$ (\cite[\S5.6]{FGA}), it is sufficient to show that $\mathrm{Aut}\,(f \colon (X,\Delta,H)\to(\mathbb{P}^1,\mathcal{O}_{\mathbb{P}^1}(1)))$ is proper. 
Let $C$ be an arbitrary affine curve and fix a closed point $0\in C$. Set $C^\circ:=C\setminus\{0\}$ and take an arbitrary isomorphism $\varphi^\circ\in\mathrm{Aut}\,(f|_{(X,\Delta,H)\times C^\circ} \colon (X,\Delta,H)\times C^\circ\to(\mathbb{P}^1,\mathcal{O}_{\mathbb{P}^1}(1))\times C^\circ)$. By Theorem \ref{sep2}, we extend $\varphi^\circ$ to $\varphi$ over $C$ entirely. Thus we complete the proof.
 \end{proof}

\subsection{Invariance of plurigenera}
In this subsection we prove a result on the invariance of plurigenera, which is a generalization of \cite{Nak} and a key statement to construct our moduli spaces.

\begin{thm}\label{thm--inv-pluri}
Let $f \colon  (X,\Delta) \to S$ be a log $\mathbb{Q}$-Gorenstein family such that $S$ is a normal variety. 
Suppose that there is $e \in \{1,-1\}$ such that for every geometric point $\overline{s}\in S$, $(X_{\overline{s}},\Delta_{\overline{s}})$ is a klt pair and $e(K_{X_{\overline{s}}}+\Delta_{\overline{s}})$ is semiample. 
Let $r$ be a positive integer such that $r(K_{X/S}+\Delta)$ is Cartier. 
Then, for every positive integer $n$, the function
$$S\ni t \mapsto {\rm dim}\,H^{0}(X_{t},\mathcal{O}_{X_{t}}(enr(K_{X_{t}}+\Delta_{t})))$$
 is constant. 
\end{thm}

First, we treat the case when $S$ is a curve.

\begin{prop}\label{prop-inv-plurigenera-almighty}
 Let $f \colon  (X,\Delta) \to C$ be a log $\mathbb{Q}$-Gorenstein family such that $C$ is a curve and all closed fibers of $f$ are klt pairs.  
Let $D$ be a Cartier divisor on $X$ such that $D-(K_{X}+\Delta)$ is semiample over $C$. 
Then, for every closed fiber $F$ of $f$, the natural morphism $f_{*}\mathcal{O}_{X}(D) \to H^{0}(F, \mathcal{O}_{F}(D|_{F}))$ is surjective.
\end{prop}

\begin{proof}
Note that $X$ is a normal variety. 
The klt property of the closed fibers of $f$ and the inversion of adjunction \cite{kawakita} imply that $(X,\Delta)$ is klt. 
Let $g  \colon  Y \to X$ be a log resolution of $(X,\Delta)$. 
We can write
$$K_{Y}+\Delta_{Y}=g^{*}(K_{X}+\Delta)+E$$
for some effective $\mathbb{Q}$-divisors $\Delta_{Y}$ and $E$ which have no common component. 
Then $(Y,\Delta_{Y}+\lceil E \rceil-E)$ is a log smooth klt pair. 
Let $c\in C$ be an arbitrary closed point with the fiber $F :=f^{*}c$. 
Then
$$(g^{*}D+\lceil E \rceil-g^{*}F)-\bigl(K_{Y}+(\Delta_{Y}+\lceil E \rceil -E)\bigr)=g^{*}\bigl(D-(K_{X}+\Delta-F)\bigr).$$
Thus, $(g^{*}D+\lceil E \rceil-g^{*}F)-\bigl(K_{Y}+(\Delta_{Y}+\lceil E \rceil -E)\bigr)$ is nef and big over $X$, and the divisor is semiample over $C$ because $D-(K_{X}+\Delta)$ is semiample over $C$ by the hypothesis. 
By Kawamata--Viehweg vanishing theorem, we have $R^{q}g_{*}\mathcal{O}_{Y}(g^{*}D+\lceil E \rceil-g^{*}F)=0$ for every $q>0$. 
Thus, the Leray spectral sequence implies
\begin{equation*}
\begin{split}
R^{1}(f\circ g)_{*}\mathcal{O}_{Y}(g^{*}D+\lceil E \rceil-g^{*}F)\cong & R^{1}f_{*}\bigl(g_{*}\mathcal{O}_{Y}(g^{*}(D-F)+\lceil E \rceil)\bigr)\\
=& R^{1}f_{*}\mathcal{O}_{X}(D-F),
\end{split}
\end{equation*}
where the last equality follows from that $E$ is effective and $g$-exceptional. 
By applying the torsion free theorem \cite[Theorem 6.3 (i)]{fujino-fund} to $f \circ g \colon  Y\to C$, $g^{*}D+\lceil E \rceil-g^{*}F$, and $(Y, \Delta_{Y}+\lceil E \rceil-E)$, we see that 
$R^{1}f_{*}\mathcal{O}_{X}(D-F)$ is torsion free. 
Now consider the exact sequence
$$f_{*}\mathcal{O}_{X}(D)\longrightarrow f_{*}\mathcal{O}_{F}(D|_{F}) \overset{\delta}{\longrightarrow} R^{1}f_{*}\mathcal{O}_{X}(D-F)$$
which is induced by
$$0\longrightarrow \mathcal{O}_{X}(D-F) \longrightarrow \mathcal{O}_{X}(D) \longrightarrow \mathcal{O}_{F}(D|_{F}) \longrightarrow 0.$$
Since $R^{1}f_{*}\mathcal{O}_{X}(D-F)$ is torsion free and $f_{*}\mathcal{O}_{F}(D|_{F})$ is zero outside $c$, we see that $\delta$ is the zero map. 
This implies that 
$$f_{*}\mathcal{O}_{X}(D)\longrightarrow f_{*}\mathcal{O}_{F}(D|_{F})=H^{0}(F, \mathcal{O}_{F}(D|_{F}))$$ is surjective.
\end{proof}

\begin{proof}[Proof of Theorem \ref{thm--inv-pluri}]
It is sufficient to prove that the equality 
$${\rm dim}\,H^{0}\bigl(X_{s},\mathcal{O}_{X_{s}}(enr(K_{X_{s}}+\Delta_{s}))\bigr)={\rm dim}\,H^{0}\bigl(X_{s'},\mathcal{O}_{X_{s'}}(enr(K_{X_{s'}}+\Delta_{X_{s'}}))\bigr)$$
holds for any two closed points $s,\,s' \in S$. 
Let $C \subset S$ be a connected (but not necessarily irreducible or smooth) curve passing through $s$ and $s'$ (cf.~\cite[\S6, Lemma]{Ab}). 
Replacing $S$ with the normalization of any component of $C$, we may assume that $S$ is a curve. 

By the hypothesis, $e(K_{X}+\Delta)$ is $f$-nef. 
Since the restriction of $e(K_{X}+\Delta)$ to the geometric generic fiber is semiample,   $e(K_{X}+\Delta)$ is $f$-abundant (\cite[Definition 4.1]{fujino-bpf}). 
By \cite[Theorem 1.1]{fujino-bpf}, $e(K_{X}+\Delta)$ is semiample over $S$. 
Then $enr(K_{X}+\Delta)-(K_{X}+\Delta)$ is also semiample over $S$ for every positive integer $n$. 
By Proposition \ref{prop-inv-plurigenera-almighty}, the morphism
$$f_{*}\mathcal{O}_{X}(enr(K_{X}+\Delta))\longrightarrow H^{0}(F,\mathcal{O}_{F}(enr(K_{F}+\Delta|_{F})))$$
is surjective for every closed fiber $F$. 
By the cohomology and base change theorem, ${\rm dim}\,H^{0}(X_{t},\mathcal{O}_{X_{t}}(enr(K_{X_{t}}+\Delta_{t})))$ is independent of $t \in S$. 
\end{proof}

\section{Construction of Moduli}\label{Consec}

In this section, we construct the moduli of uniformly adiabatically K-stable polarized klt-trivial fibrations over curves such that the canonical divisor is not numerically trivial. 
Throughout this section, we fix $d \in \mathbb{Z}_{>0}$, $u \in \mathbb{Q}_{\neq 0}$ with $e:=\frac{u}{|u|}$,  $v  \in \mathbb{Q}_{> 0}$, and $w\in\mathbb{Q}_{>0}$. 
We define
$$\mathfrak{Z}_{d, v,u,w}:=\left\{
\begin{array}{l}
f\colon (X,\Delta=0,A) \to C
\end{array}
\;\middle|
\begin{array}{rl}
(i)&\text{$f$ is a uniformly adiabatically}\\
&\text{K-stable polarized klt-trivial} \\
&\text{fibration over a curve $C$,}\\
(ii)&\text{${\rm dim}X=d$,}\\
(iii)&\text{$K_X\equiv uf^*H$ for some line bundle} \\
&\text{$H$ on $C$ such that $\mathrm{deg}\,H=1$,} \\
(iv)&\text{$A$ is an ample line bundle on $X$}\\
&\text{such that $(K_X\cdot A^{d-1})=uv$ and}\\
&\text{${\rm vol}(A)\le w$.}
\end{array}\right\}.$$  

Then it is not difficult to check that if $f\colon (X,\Delta=0,A) \to C$ is an element of $\mathfrak{Z}_{d, v,u,w}$ then $(X,0) \to C \in \mathfrak{G}_{d, \{0\}, v, u}$, where $\mathfrak{G}_{d, \{0\}, v, u}$ is the set $\mathfrak{G}_{d, \Theta, v, u}$ in \S\ref{Bousec} with $\Theta=\{0\}$. 
By Lemma \ref{lem--Cartierindex}, there exists an $r\in\mathbb{Z}_{>0}$, depending only on $d$, $u$, and $v$, such that for any element $f\colon (X,0) \to C$ of $\mathfrak{G}_{d, \{0\}, v, u}$, we have $erK_X\sim f^{*}D$ for some very ample Cartier divisor $D$ on $C$. 

The following theorem is the main result of this paper.

\begin{thm}\label{mod2}
We fix $d\in\mathbb{Z}_{>0}$, $u\in\mathbb{Q}_{\ne0}$ with $e:=\frac{u}{|u|}$, $v\in\mathbb{Q}_{>0}$, $w\in\mathbb{Q}_{>0}$, and $r \in\mathbb{Z}_{>0}$ in Lemma \ref{lem--Cartierindex} for $\mathfrak{G}_{d, \{0\}, v, u}$. 
Let $\mathscr{M}_{d,v,u,w,r}$ be a full subcategory of $\mathfrak{Pol}$ such that for any locally Noetherian scheme $S$ over $\mathbbm{k}$, we define $\mathscr{M}_{d,v,u,w,r}(S)$ to be a groupoid whose objects are
 $$\left\{
 \vcenter{
 \xymatrix@C=12pt{
(\mathcal{X},\mathscr{A})\ar[rr]^-{f}\ar[dr]_{\pi_{\mathcal{X}}}&& \mathcal{C} \ar[dl]\\
&S
}
}
\;\middle|
\begin{array}{rl}
(i)&\text{$\pi_{\mathcal{X}}$ is a flat projective morphism and $\mathcal{X}$ is a scheme,}\\
(ii)&\text{$\mathscr{A}\in\mathrm{Pic}_{\mathcal{X}/S}(S)$ such that $\mathscr{A}_{\bar{s}}$ is ample for any}\\
&\text{geometric point $\bar{s}\in S$,}\\
(iii)&\text{$\omega_{\mathcal{X}/S}^{[r]}$ exists as a line bundle,}\\
(iv)&\text{$\pi_{\mathcal{X}*}\omega_{\mathcal{X}/S}^{[ler]}$ is locally free and generates}\\
&\text{$H^0(\mathcal{X}_{s}, \mathcal{O}_{\mathcal{X}_{s}}(lerK_{\mathcal{X}_{s}}))$ for any point $s\in S$ and any}\\
&\text{$l\in\mathbb{Z}_{>0}$,}\\
(v)&\text{$f$ is the ample model of $\omega_{\mathcal{X}/S}^{[er]}$ over $S$ and}\\
&\text{$f_{\bar{s}}\colon(\mathcal{X}_{\overline{s}},0,\mathscr{A}_{\overline{s}})\to \mathcal{C}_{\overline{s}} \in \mathfrak{Z}_{d, v,u,w}$ for any geometric}\\
&\text{point $\overline{s}\in S$,}
\end{array}\right\}.$$

Then $\mathscr{M}_{d,v,u,w,r}$ is a separated Deligne-Mumford stack of finite type over $\mathbbm{k}$. 
Furthermore, there exists a coarse moduli space of $\mathscr{M}_{d,v,u,w,r}$.
 \end{thm}

 \begin{rem}
     Note that for any $S$-isomorphism $g\colon\mathcal{X}\to\mathcal{X'}$ as above, we have a unique $S$-isomorphism $h\colon\mathcal{C}\to\mathcal{C}'$ such that $f'\circ g=h\circ f$.
     This is the reason why we do not consider morphisms between $\mathcal{C}$ and $\mathcal{C}'$.
 \end{rem}
 
In this section, for every object $(\mathcal{X},\mathscr{A})\to\mathcal{C}\in\mathscr{M}_{d,v,u,w,r}(S)$, the structure morphism $(\mathcal{X},\mathscr{A}) \to S$ is denoted by $\pi_{\mathcal{X}}$ unless otherwise stated. 
When an object $(\mathcal{X}_{T},\mathscr{A}_{T})\to\mathcal{C}_{T}$ of $\mathscr{M}_{d,v,u,w,r}(T)$ is the base change of $(\mathcal{X},\mathscr{A})\to\mathcal{C}$ by $T \to S$, the morphism $\pi_{\mathcal{X}_{T}}$ is nothing but $(\pi_{\mathcal{X}})_{T}$ as in (\ref{notation-(11)}) in Notation and Convention.

\begin{lem}\label{lem-stack}
$\mathscr{M}_{d,v,u,w,r}$ is a stack.
\end{lem}

\begin{proof}
We first check that $\mathscr{M}_{d,v,u,w,r}$ is a category fibered in groupoids.
It suffices to show that for any $\pi_{\mathcal{X}} \colon (\mathcal{X},\mathscr{A})\to\mathcal{C}\to S \in\mathscr{M}_{d,v,u,w,r}(S)$ and any morphism $h\colon T\to S$ of schemes, the base change $\pi_{\mathcal{X}_{T}} \colon (\mathcal{X}_{T},\mathscr{A}_T)\to\mathcal{C}_{T} \to T$ is the pullback of $\pi$ along $h$ in the sense of \cite[Definition 3.1.1]{Ols}. 
By the conditions (iv) and (v) in the definition of $\mathscr{M}_{d,v,u,w,r}$ and the theorem of cohomology and base change, we see that  
    \[
    \mathcal{C}_{T}:=\mathcal{C}\times_ST=\mathbf{Proj}_S(\bigoplus_{l\ge0}\pi_{\mathcal{X}*}\omega_{\mathcal{X}/S}^{[ler]})\times_ST\cong\mathbf{Proj}_T(\bigoplus_{l\ge0}\pi_{\mathcal{X}_{T}*}\omega_{\mathcal{X}_T/T}^{[ler]}).
    \]
 This shows $(\mathcal{X}_{T},\mathscr{A}_T)\to\mathcal{C}_{T}\in\mathscr{M}_{d,v,u,w,r}(T)$.
 Hence, $\mathscr{M}_{d,v,u,w,r}$ is indeed a category fibered in groupoids.

From now on, we check that $\mathscr{M}_{d,v,u,w,r}$ is a stack. 
Since Definition \ref{defn--stacks} (\ref{defn--stack-(1)}) has been already checked in Lemma \ref{lem--descent}, it suffices to show the condition of Definition \ref{defn--stacks}  (\ref{defn--stack-(2)}) for $\mathscr{M}_{d,v,u,w,r}$.
We note that $\mathscr{M}_{d,v,u,w,r}$ satisfies the condition of Remark \ref{rem-descent}.
Let $g\colon S'\to S$ be an \'{e}tale covering and $(f'\colon(\mathcal{X}',\mathscr{A}')\to\mathcal{C}',\sigma)$ a descent datum with the structure morphism $\pi_{\mathcal{X}'} \colon (\mathcal{X}',\mathscr{A}') \to S'$. 
We will show that $(f',\sigma)$ is effective.
By Lemma \ref{lem--descent}, $(\pi_{\mathcal{X}'}\colon (\mathcal{X}',\mathscr{A}') \to S',\sigma)$ is a descent datum in $\mathfrak{Pol}$. Therefore, the datum comes from some element $\pi\colon(\mathcal{X},\mathscr{A})\to S\in\mathfrak{Pol}(S)$. 
By the functoriality of $\omega_{\mathcal{X}'/S'}^{[r]}$ and \cite[Theorem 4.23]{FGA}, there exists a line bundle $\mathscr{L}$ on $\mathcal{X}$ such that $g_{\mathcal{X}}^*\mathscr{L}=\omega_{\mathcal{X}'/S'}^{[r]}$, and there exists a morphism $\omega_{\mathcal{X}/S}^{\otimes r}\to\mathscr{L}$ whose pullback $g_{\mathcal{X}}^*\omega_{\mathcal{X}/S}^{\otimes r} \to g_{\mathcal{X}}^*\mathscr{L}$ coincides with the natural morphism $\omega_{\mathcal{X}'/S'}^{\otimes r}\to\omega_{\mathcal{X}'/S'}^{[r]}$.
From these facts, we have $\mathscr{L}=\omega^{[r]}_{\mathcal{X}/S}$.
By the faithfully flatness of $g$ and the flat base change theorem \cite[III, Proposition 9.3]{Ha}, the condition (iv) of $\mathscr{M}_{d,v,u,w,r}$ holds for $\pi$.
Thus, $\omega_{\mathcal{X}'/S'}^{[er]}$ is relatively semiample, and if we set $\mathcal{C}:=\mathbf{Proj}_S(\bigoplus_{l\ge0}\pi_{*}\omega^{[ler]}_{\mathcal{X}/S})$, then 
$$\mathcal{C}\times_SS'\cong\mathbf{Proj}_{S'}(\bigoplus_{l\ge0}\pi_{\mathcal{X}'*}\omega^{[ler]}_{\mathcal{X}'/S'})=\mathcal{C}'$$
by \cite[III, Theorem 12.11]{Ha}.
Let $f\colon\mathcal{X}\to\mathcal{C}$ be the canonical morphism. 
Then the base change of $f$ by $S'\to S$ is isomorphic to $f'$.
From this, (v) of $\mathscr{M}_{d,v,u,w,r}$ holds for $f\colon(\mathcal{X},\mathscr{A})\to \mathcal{C}$. 
This shows $f\colon(\mathcal{X},\mathscr{A})\to \mathcal{C}\in\mathscr{M}_{d,v,u,w,r}(S)$ and hence $(f',\sigma)$ is an effective descent datum. 
\end{proof}

Note that the set of all klt-trivial fibrations over $\mathbbm{k}$ belonging to $\mathfrak{Z}_{d, v,u,w}$ coincides with the set of isomorphic classes of $\mathscr{M}_{d,v,u,w,r}({\rm Spec}\,\mathbbm{k})$.
From now on, we fix $I\in\mathbb{Z}_{>0}$ as Corollary \ref{cor--hilbertpolynomial} for $\mathfrak{G}_{d,\{0\},v,u,w}$.
Note that $\mathfrak{Z}_{d, v,u,w}\subset \mathfrak{G}_{d,\{0\},v,u,w}$.

\begin{lem}\label{lem--stack-decomp}
For any $d_{1}, \, d_{2}, \, d_{3} \in \mathbb{Z}_{>0}$, and $h \in \mathbb{Q}[t]$, let $\mathscr{M}_{d_1,d_2,d_3,h}$ be a full subcategory of $\mathscr{M}_{d,v,u,w,r}$ such that for any locally Noetherian scheme $S$ over $\mathbbm{k}$, we define a groupoid $\mathscr{M}_{d_1,d_2,d_3,h}(S)$ whose objects are
$$\left\{
f\colon(\mathcal{X},\mathscr{A})\to\mathcal{C} \in \mathscr{M}_{d,v,u,w,r}(S)
\;\middle|
\begin{array}{rl}
&\text{\!\!\!\!\!\!\!\!\!\! for every geometric point $\overline{s} \in S$,}\\
\bullet&h^0(\mathcal{X}_{\bar{s}},\mathcal{O}_{\mathcal{X}_{\bar{s}}}(I\mathscr{A}_{\bar{s}}))=d_1,\\
\bullet&h^0(\mathcal{X}_{\bar{s}},\mathcal{O}_{\mathcal{X}_{\bar{s}}}((I+1)\mathscr{A}_{\bar{s}}))=d_2,\\
\bullet&h^0(\mathcal{X}_{\bar{s}},\mathcal{O}_{\mathcal{X}_{\bar{s}}}(erK_{\mathcal{X}_{\bar{s}}}))=d_3,\\
\bullet&\text{the Hilbert polynomial of $\mathcal{X}_{\bar{s}}$ with}\\
&\text{respect to $(2I+1)\mathscr{A}_{\bar{s}}+erK_{\mathcal{X}_{\bar{s}}}$ is $h$.}
\end{array}\right\}.$$ 
Then $\mathscr{M}_{d_1,d_2,d_3,h}$ is an open and closed substack of $\mathscr{M}_{d,v,u,w,r}$. 
Furthermore, there are only finitely many $d_{1}, \, d_{2}, \, d_{3} \in \mathbb{Z}_{>0}$, and $h \in \mathbb{Q}[t]$ such that $\mathscr{M}_{d_1,d_2,d_3,h}$ is not an empty stack. 
\end{lem}

\begin{proof}
By Theorem \ref{thm--inv-pluri}, any scheme $S$ and $f\colon(\mathcal{X},\mathscr{A})\to\mathcal{C} \in \mathscr{M}_{d,v,u,w,r}(S)$ satisfy the property that $h^0(\mathcal{X}_{\bar{s}},\mathcal{O}_{\mathcal{X}_{\bar{s}}}(I\mathscr{A}_{\bar{s}}))$, $h^0(\mathcal{X}_{\bar{s}},\mathcal{O}_{\mathcal{X}_{\bar{s}}}((I+1)\mathscr{A}_{\bar{s}}))$, $h^0(\mathcal{X}_{\bar{s}},\mathcal{O}_{\mathcal{X}_{\bar{s}}}(erK_{\mathcal{X}_{\bar{s}}}))$, and
the Hilbert polynomial of $\mathcal{X}_{\bar{s}}$ with
respect to $(2I+1)\mathscr{A}_{\bar{s}}+erK_{\mathcal{X}_{\bar{s}}}$ are locally constant on $s\in S$. 
The first assertion follows from this fact. 
The second assertion follows from Lemma \ref{lem--Cartierindex} and Corollary \ref{cor--hilbertpolynomial}. 
\end{proof}

The invariants $h^0(\mathcal{X}_{\bar{s}},\mathcal{O}_{\mathcal{X}_{\bar{s}}}(I\mathscr{A}_{\bar{s}}))$ and $h^0(\mathcal{X}_{\bar{s}},\mathcal{O}_{\mathcal{X}_{\bar{s}}}((I+1)\mathscr{A}_{\bar{s}}))$ in Lemma \ref{lem--stack-decomp} are used to determine $\mathscr{A}\in\mathrm{Pic}_{\mathcal{X}/S}(S)$. 

\begin{note}
For each $d_{1}, \, d_{2}, \, d_{3} \in \mathbb{Z}_{>0}$, and $h \in \mathbb{Q}[t]$, we set 
$$H:=\mathrm{Hilb}_{\mathbb{P}^{d_1-1}\times \mathbb{P}^{d_2-1}\times\mathbb{P}^{d_{3}-1}}^{h,\,p_1^*\mathcal{O}(1)\otimes p_2^*\mathcal{O}(1) \otimes  p_3^*\mathcal{O}(1)}.$$ 
Let $\tilde{\pi} \colon \mathcal{U} \to H$ be the morphism from the universal family $\mathcal{U}$. 
We set $p_{i} \colon  \mathcal{U} \to \mathbb{P}_H^{d_i-1}$ as the morphism induced by the projections $\mathbb{P}_H^{d_1-1}\times_H \mathbb{P}_H^{d_2-1}\times_H\mathbb{P}_H^{d_{3}-1} \to \mathbb{P}_H^{d_i-1}$. 
We remark that $H$ is of finite type over $\mathbbm{k}$.
\end{note}

For any morphism $T\to H$, the morphism $\tilde{\pi}_{T}\colon\mathcal{U}_T\to T$ denotes the base change of $\tilde{\pi}$ by $T\to H$. 

\begin{prop}\label{lem-N-univ}
Fix $I \in \mathbb{Z}_{>0}$ of Corollary \ref{cor--hilbertpolynomial}. 
For all $d_1,\, d_2, \, d_3\in\mathbb{Z}_{>0}$ and $h\in\mathbb{Q}[t]$, the following $\mathfrak{H}\colon(\mathbf{Sch}_{/\mathbbm{k}})^{\mathrm{op}}\to \mathbf{Sets}$ is a well-defined functor and $\mathfrak{H}$ is represented by a locally closed subscheme $N_{d_1,d_2,d_3,h}\subset H${\rm :} 
For a scheme $S$,
\[
\mathfrak{H}(S):=\left\{ \begin{array}{l}
(f\colon(\mathcal{X},\mathscr{A})\to\mathcal{C},\rho_1,\rho_2,\rho_3)
\end{array}
\;\middle|
\begin{array}{rl}
&\text{$f \in \mathscr{M}_{d_1,d_2,d_3,h}(S)$ such that $\mathscr{A}$ is}\\
& \text{represented by a line bundle, and}\\
& \text{$\rho_{1}\colon\mathbb{P}_S(\pi_{\mathcal{X}*}\mathscr{A}^{\otimes I})\to\mathbb{P}_S^{d_1-1}$,}\\
& \text{$\rho_{2}\colon\mathbb{P}_S(\pi_{\mathcal{X}*}\mathscr{A}^{\otimes I+1})\to\mathbb{P}_S^{d_2-1}$, and}\\
& \text{$\rho_{3}\colon\mathbb{P}_S(\pi_{\mathcal{X}*}\omega_{\mathcal{X}/S}^{[er]})\to\mathbb{P}_S^{d_3-1}$}\\
& \text{are isomorphisms.}
\end{array}\right\}/\sim,
\]
where $(f\colon(\mathcal{X},\mathscr{A})\to\mathcal{C},\rho_1,\rho_2,\rho_3)\sim (f'\colon(\mathcal{X}',\mathscr{A}')\to\mathcal{C}',\rho'_1,\rho'_2,\rho'_3)$ if and only if there exists an isomorphism $\alpha\colon(\mathcal{X},\mathscr{A}) \to(\mathcal{X}',\mathscr{A}')$ of $\mathscr{M}_{d,v,u,w,r}(S)$ (see the definition of $\mathfrak{Pol}$) such that the induced isomorphisms
\begin{equation*}
\begin{split} 
\alpha_1&\colon\mathbb{P}_S(\pi_{\mathcal{X}*}\mathscr{A}^{\otimes I})\to \mathbb{P}_S(\pi_{\mathcal{X}'*}\mathscr{A}'^{\otimes I}),\quad  \alpha_2\colon\mathbb{P}_S(\pi_{\mathcal{X}*}\mathscr{A}^{\otimes I+1})\to \mathbb{P}_S(\pi_{\mathcal{X}'*}\mathscr{A}'^{\otimes I+1}), \quad {\rm and} \\
 \alpha_3&\colon\mathbb{P}_S(\pi_{\mathcal{X}*}\omega_{\mathcal{X}/S}^{[er]})\to \mathbb{P}_S(\pi_{\mathcal{X}'*}\omega_{\mathcal{X}'/S}^{[er]})
\end{split}
\end{equation*}
satisfy $\rho'_i\circ\alpha_i=\rho_i$ $(i =1,2,3)$. 
Here, the structure morphisms $\mathcal{X}\to S$ and $\mathcal{X}'\to S$ are denoted by $\pi_{\mathcal{X}}$ and $\pi_{\mathcal{X}'}$ respectively, and the line bundle representing $\mathscr{A}$ is denoted by $\mathscr{A}$ by abuse of notations.

In particular, $N_{d_1,d_2,d_3,h}$ inherits the $ PGL(d_1)\times PGL(d_2)\times PGL(d_3)$ action on $H$.
\end{prop}

\begin{proof}
We first note that $\mathbb{P}_S(\pi_{\mathcal{X}*}\mathscr{A}^{\otimes I})$ and $\mathbb{P}_S(\pi_{\mathcal{X}*}\mathscr{A}^{\otimes I+1})$ are independent of a representative of $\mathscr{A}$ (cf.~Claim in the proof of Lemma \ref{lem--descent}).
The well-definedness of $\mathfrak{H}$ follows from the fact that we can define the pullback of $(f,\rho_1,\rho_2,\rho_3)\in\mathfrak{H}(S)$ by any morphism $S'\to S$ by using \cite[III, Theorem 12.11]{Ha} and the condition (iv) of $\mathscr{M}_{d,v,u,w,r}$.
Indeed, by the properties of $I$ (see Corollary \ref{cor--hilbertpolynomial}), we have 
$$h^i(\mathcal{X}_{\bar{s}},\mathcal{O}_{\mathcal{X}_{\bar{s}}}(I\mathscr{A}_{\bar{s}}))=h^i(\mathcal{X}_{\bar{s}},\mathcal{O}_{\mathcal{X}_{\bar{s}}}((I+1)\mathscr{A}_{\bar{s}}))=0$$
for every $i>0$ and geometric point $\bar{s}\in S$.
Thus, 
\begin{equation*}
\begin{split} 
&\mathbb{P}_{S'}(\pi_{\mathcal{X_{S'}}*}\mathscr{A}_{S'}^{\otimes I})\cong \mathbb{P}_S(\pi_{\mathcal{X}*}\mathscr{A}^{\otimes I})\times_SS',\;\;  \mathbb{P}_{S'}(\pi_{\mathcal{X_{S'}}*}\mathscr{A}_{S'}^{\otimes I+1})\cong \mathbb{P}_S(\pi_{\mathcal{X}*}\mathscr{A}^{\otimes I+1})\times_SS', \;\;{\rm and}\\ &\mathbb{P}_{S'}(\pi_{\mathcal{X}_{S'}*}\omega_{\mathcal{X}_{S'}/S'}^{[er]})\cong \mathbb{P}_S(\pi_{\mathcal{X}*}\omega_{\mathcal{X}/S}^{[er]})\times_SS'.
\end{split}
\end{equation*}

We will prove the proposition in several steps. 

\begin{step}\label{keyprop-step1}
In this step, we introduce a claim and give an explanation of the claim. 

We will consider the following claim, which will be proved in Step \ref{keyprop-step3}. 
\begin{claim}\label{cl2}
There exists a locally closed subscheme $N$ of $H$ such that a morphism $T\to H$ factors through $N\hookrightarrow H$ if and only if there exists a $\tilde{\pi}_{T}$-ample line bundle $\mathscr{A}'$ on $\mathcal{U}_T$ such that $\mathscr{A}'$ and
$\tilde{\pi}_{T}\colon\mathcal{U}_T\to T$ satisfy the following. 
\begin{enumerate}[\rm{(}a\rm{)}]
\item\label{claim1-(a)}
any geometric fiber of $\tilde{\pi}_{T}$ is connected and normal,
\item\label{claim1-(b)}
 $p_{1,\bar{t}}$ and $p_{2,\bar{t}}$ are closed immersions for any geometric point $\bar{t}\in T$
\item\label{claim1-(c)}
$\mathscr{A}'^{\otimes I}\sim_Tp_{1,T}^*\mathcal{O}_{\mathbb{P}_T^{d_1-1}}(1)$ and $\mathscr{A}'^{\otimes I+1}\sim_Tp_{2,T}^*\mathcal{O}_{\mathbb{P}^{d_2-1}_T}(1)$,
\item
\label{claim1-(d)}
 the morphisms $\mathcal{O}_{T}^{\oplus d_1}\to H^0(\mathcal{U}_t,\mathscr{A}'^{\otimes I}_{t})$ and $\mathcal{O}_{T}^{\oplus d_2}\to H^0(\mathcal{U}_t,\mathscr{A}'^{\otimes I+1}_{t})$  are surjective, $h^0(\mathcal{U}_t,\mathscr{A}'^{\otimes I}_{t})=d_1$ and $h^0(\mathcal{U}_t,\mathscr{A}'^{\otimes I+1}_{t})=d_2$ for any point $t\in T$,
\item\label{claim1-(e)}
 $\omega_{\mathcal{U}_T/T}^{[er]}\sim_Tp_{3,T}^*\mathcal{O}(1)$,
\item\label{claim1-(f)} 
$\mathcal{U}_{\bar{t}}$ is a klt variety for any geometric point $\bar{t}\in T$,
\item\label{claim1-(g)}
 $\tilde{\pi}_{T*}\omega_{\mathcal{U}_T/T}^{[ler]}\to H^0(\mathcal{U}_{t}, \mathcal{O}_{\mathcal{U}_{t}}(lerK_{\mathcal{U}_{t}}))$ is surjective for any point $t\in T$ and any $l\in\mathbb{Z}_{>0}$,
\item\label{claim1-(h)}
 $\mathcal{O}_{T}^{\oplus d_3}\to H^0(\mathcal{U}_{t}, \mathcal{O}_{\mathcal{U}_{t}}(erK_{\mathcal{U}_{t}}))$ is surjective and $h^0(\mathcal{U}_{t}, \mathcal{O}_{\mathcal{U}_{t}}(erK_{\mathcal{U}_{t}}))=d_3$ for any point $t\in T$, 
\item\label{claim1-(i)}
$(\mathcal{U}_{\bar{t}},0,\mathscr{A}'_{\bar{t}}) \to \widetilde{\mathcal{C}}_{\bar{t}}\in\mathfrak{Z}_{d,v,u,w} $ for any geometric point ${\bar{t}}\in T$, where $\mathcal{U}_T\to\widetilde{\mathcal{C}_T}$ is the ample model of $\omega_{\mathcal{U}_T/T}^{[er]}$.
\end{enumerate}
Here, the morphism $\mathcal{O}_T^{\oplus d_{1}}\to H^0(\mathcal{U}_t,\mathscr{A}'^{\otimes I}_{t})$ in (\ref{claim1-(d)}) is defined to be the composition
$$\mathcal{O}_T^{\oplus d_{1}}\longrightarrow \tilde{\pi}_{T*}p_{3,T}^*\mathcal{O}_{\mathbb{P}_T^{d_3-1}}(1)\longrightarrow H^0(\mathcal{U}_t,p_{3,t}^*\mathcal{O}_{\mathbb{P}^{d_3-1}}(1))\overset{\cong}{\longrightarrow} H^0(\mathcal{U}_t,\mathscr{A}'^{\otimes I}_{t}),$$ 
where the last isomorphism is induced by $\mathscr{A}'^{\otimes I}\sim_Tp_{1,T}^*\mathcal{O}_{\mathbb{P}_T^{d_1-1}}(1)$ in (\ref{claim1-(c)}), and the other morphisms in (\ref{claim1-(d)}) and (\ref{claim1-(h)}) are defined similarly.
\end{claim}

We give a few words about the conditions (\ref{claim1-(a)})--(\ref{claim1-(i)}). 
The roles of these conditions are as follows.
\begin{itemize}
\item
(\ref{claim1-(e)}), (\ref{claim1-(g)}), and (\ref{claim1-(i)}) are related to the conditions of $\mathscr{M}_{d,v,u,w,r}(T)$,
\item (\ref{claim1-(c)}), (\ref{claim1-(d)}), (\ref{claim1-(e)}), and (\ref{claim1-(h)}) are utilized to prove the representability of $\mathfrak{H}$, and
\item
(\ref{claim1-(a)}), (\ref{claim1-(b)}), 
and (\ref{claim1-(f)}) are extra and written just for the convenience of the proof. 
\end{itemize}
More precisely, (\ref{claim1-(a)}), (\ref{claim1-(b)}), and (\ref{claim1-(f)}) immediately follow from (\ref{claim1-(d)}), (\ref{claim1-(i)}), and the properties of $I$ in Corollary \ref{cor--hilbertpolynomial}, and there are following correspondences.
\begin{itemize}
\item
(\ref{claim1-(e)}) implies (iii) of $\mathscr{M}_{d,v,u,w,r}(T)$,
\item
(\ref{claim1-(g)}) corresponds to (iv) of $\mathscr{M}_{d,v,u,w,r}(T)$,
\item
(\ref{claim1-(i)}) corresponds to (v) of $\mathscr{M}_{d,v,u,w,r}(T)$.
\end{itemize}
Thus, every morphism $(\mathcal{U}_T,\mathscr{A}')\to\widetilde{\mathcal{C}_T}$  satisfying (\ref{claim1-(a)})--(\ref{claim1-(i)}) is an object of $\mathscr{M}_{d,v,u,w,r}(T)$.
\end{step}

\begin{step}\label{keyprop-step2}
In this step, we prove Proposition \ref{lem-N-univ} assuming the existence of $N$ in Claim \ref{cl2}. 

 Let $N$ be the scheme in Claim \ref{cl2} and let $\mathcal{U}_N\subset \mathbb{P}_N^{d_1-1}\times_N\mathbb{P}_N^{d_2-1}\times_N\mathbb{P}_N^{d_3-1}$ be the universal subscheme. 
We fix $\tilde{\mathscr{A}}_N$ as in Claim \ref{cl2}. 
By (\ref{claim1-(c)}) and (\ref{claim1-(e)}), we can find line bundles $\mathcal{M}_1$, $\mathcal{M}_2$, and $\mathcal{M}_3$ on $N$ such that 
\begin{equation*}
\begin{split} 
p_{1,N}^*\mathcal{O}(1)&\sim\tilde{\pi}_{N}^*\mathcal{M}_1\otimes\tilde{\mathscr{A}}_N^{\otimes I},\qquad  p_{2,N}^*\mathcal{O}(1)\sim\tilde{\pi}_{N}^*\mathcal{M}_2\otimes\tilde{\mathscr{A}}_N^{\otimes I+1}, \qquad {\rm and} \\
 p_{3,N}^*\mathcal{O}(1)&\sim\tilde{\pi}_{N}^*\mathcal{M}_3\otimes\omega^{[er]}_{\mathcal{U}_N/N}.
\end{split}
\end{equation*}
By (\ref{claim1-(d)}) and (\ref{claim1-(h)}) and applying \cite[Lecture 7, Corollary 2]{mumford-lecture} to the natural morphisms $\mathcal{O}_{N}^{\oplus d_1}\to\tilde{\pi}_{{N}*}p_{1,N}^*\mathcal{O}(1)\cong \mathcal{M}_1\otimes\tilde{\pi}_{N*}\tilde{\mathscr{A}}_N^{\otimes I}$, $\mathcal{O}_{N}^{\oplus d_2}\to\tilde{\pi}_{{N}*}p_{2,N}^*\mathcal{O}(1)\cong \mathcal{M}_2\otimes\tilde{\pi}_{N*}\tilde{\mathscr{A}}_N^{\otimes I+1}$, and $\mathcal{O}_{N}^{\oplus d_3}\to \tilde{\pi}_{{N}*}p_{3,N}^*\mathcal{O}(1) \cong \mathcal{M}_3\otimes\tilde{\pi}_{N*}\omega^{[er]}_{\mathcal{U}_N/N}$, we see that 
\begin{equation*}
\begin{split} 
\mathcal{O}_N^{\oplus d_1}&\cong \mathcal{M}_1\otimes\tilde{\pi}_{N*}\tilde{\mathscr{A}}_N^{\otimes I},\qquad  \mathcal{O}_N^{\oplus d_2}\cong \mathcal{M}_2\otimes\tilde{\pi}_{N*}\tilde{\mathscr{A}}_N^{\otimes I+1}, \qquad {\rm and} \\
\mathcal{O}_N^{\oplus d_3}&\cong \mathcal{M}_3\otimes\tilde{\pi}_{N*}\omega^{[er]}_{\mathcal{U}_N/N}.
\end{split}
\end{equation*}
From these relations, we obtain isomorphisms
\begin{equation*}
\begin{split} 
\tilde{\rho}_1\colon&\mathbb{P}_N(\tilde{\pi}_{N*}\tilde{\mathscr{A}}_N^{\otimes I})\overset{\cong}{\longrightarrow}\mathbb{P}_N^{ d_1-1} ,\qquad  \tilde{\rho}_2\colon\mathbb{P}_N(\tilde{\pi}_{N*}\tilde{\mathscr{A}}_N^{\otimes I+1})\overset{\cong}{\longrightarrow}\mathbb{P}_N^{ d_2-1},  \qquad {\rm and} \\
\tilde{\rho}_3\colon&\mathbb{P}_N(\tilde{\pi}_{N*}\omega^{[er]}_{\mathcal{U}_N/N})\overset{\cong}{\longrightarrow}\mathbb{P}_N^{ d_3-1}. 
\end{split}
\end{equation*}  
This fact and the universal property of $N$ show that there exists an injective map 
$$\eta(S)\colon\mathrm{Hom}(S,N)\hookrightarrow\mathfrak{H}(S),$$
which maps $\gamma\colon S\to N$ to $(f_S\colon(\mathcal{U}_S,\tilde{\mathscr{A}}_S)\to\widetilde{\mathcal{C}_S},\tilde{\rho}_{1,S},\tilde{\rho}_{2,S},\tilde{\rho}_{3,S})$, where $\tilde{\rho}_{i,S}$ is the base change of $\tilde{\rho}_{i}$ by $S$. 
Therefore we obtain a morphism $\eta\colon\mathrm{Hom}(\bullet,N)\to\mathfrak{H}$. 

It suffices to prove the surjectivity of $\eta$. 
In general, for two locally free sheaves $\mathcal{E}$ and $\mathcal{E}'$ on $S$ with an $S$-isomorphism $g\colon\mathbb{P}_S(\mathcal{E})\to \mathbb{P}_S(\mathcal{E}')$, we have $g^*\mathcal{O}_{\mathbb{P}_S(\mathcal{E}')}(1)\sim_S\mathcal{O}_{\mathbb{P}_S(\mathcal{E})}(1)$. 
Indeed, we put $\mathcal{F}:=g^*\mathcal{O}_{\mathbb{P}_S(\mathcal{E}')}(1)\otimes\mathcal{O}_{\mathbb{P}_S(\mathcal{E})}(-1)$. Then $\mathcal{F}$ is locally trivial over $S$ by \cite[\S0.5 (b)]{GIT}. 
Thus, the pushforward of $\mathcal{F}$ to $S$ is an invertible sheaf. 
From this fact and the global generation of $\mathcal{F}$ over $S$, we have $g^*\mathcal{O}_{\mathbb{P}_S(\mathcal{E}')}(1)\sim_S\mathcal{O}_{\mathbb{P}_S(\mathcal{E})}(1)$. 
By using this fact, for any object $(f\colon(\mathcal{X},\mathscr{A})\to\mathcal{C},\rho_1,\rho_2,\rho_3)$ of $\mathfrak{H}(S)$ with the canonical morphisms $f_1\colon\mathcal{X}\to\mathbb{P}_S(\pi_{\mathcal{X}*}\mathscr{A}^{\otimes I})$, $f_2\colon\mathcal{X}\to\mathbb{P}_S(\pi_{\mathcal{X}*}\mathscr{A}^{\otimes I+1})$, and $f_{3}\colon\mathcal{X}\to\mathbb{P}_S(\pi_{\mathcal{X}*}\omega_{\mathcal{X}/S}^{[er]})$, we have 
\begin{equation*}\tag{I}\label{eq-c-g}
\begin{split}
(\rho_1\circ f_{1})^*\mathcal{O}(1)&\sim_S\mathscr{A}^{\otimes I},\qquad  
(\rho_2\circ f_{2})^*\mathcal{O}(1)\sim_S\mathscr{A}^{\otimes I+1},\qquad  {\rm and} \\(\rho_3\circ f_{3})^*\mathcal{O}(1)&\sim_S\omega_{\mathcal{X}/S}^{[er]},\end{split}
\end{equation*}
By the properties of $I$ in Corollary \ref{cor--hilbertpolynomial} and the condition (iv) of $\mathscr{M}_{d,v,u,w,r}$ with the aid of \cite[III, Theorem 12.11]{Ha}, we see that the fibers of $\pi_{\mathcal{X}*}\mathscr{A}^{\otimes I}$, $\pi_{\mathcal{X}*}\mathscr{A}^{\otimes I+1}$, and $\pi_{\mathcal{X}*}\omega_{\mathcal{X}/S}^{[er]}$ coincide with $H^0(\mathcal{X}_{\bar{s}},\mathscr{A}^{\otimes I}_{\bar{s}})$, $H^0(\mathcal{X}_{\bar{s}},\mathscr{A}^{\otimes I+1}_{\bar{s}})$, and $H^0(\mathcal{X}_{\bar{s}},\mathcal{O}_{\mathcal{X}_{\bar{s}}}(erK_{\mathcal{X}_{\bar{s}}}))$, respectively, over every geometric point ${\bar{s}}\in S$.
Then the three linear equivalences in (\ref{eq-c-g}) induce the following surjective morphisms  
\begin{equation*}\tag{II}\label{eq-e-k}
\begin{split}
 \mathcal{O}_{S}^{\oplus d_1}&\to H^0(\mathcal{X}_s,\mathscr{A}^{\otimes I}_{s}),\qquad  \mathcal{O}_{S}^{\oplus d_2}\to H^0(\mathcal{X}_s,\mathscr{A}^{\otimes I+1}_{s}),\qquad {\rm and}\\
 \mathcal{O}_S^{\oplus d_3}&\to H^0(\mathcal{X}_s,\mathcal{O}_{\mathcal{X}_s}(erK_{\mathcal{X}_s})).\nonumber
\end{split}
\end{equation*}
  for any point $s \in S$. 
  
We set $p_i:=\rho_i\circ f_i$. 
By the properties of $I$ in Corollary \ref{cor--hilbertpolynomial}, $p_{1,\bar{s}}$ is a closed immersion for every geometric point $\bar{s}\in S$.
 Thus, 
 $$p_1\times p_2\times p_3\colon\mathcal{X}\hookrightarrow \mathbb{P}_S^{d_1-1}\times_S\mathbb{P}_S^{d_2-1}\times_{S}\mathbb{P}_S^{d_3-1}$$ is a closed immersion.
 The morphism $\gamma\colon S\to H$ corresponding to $p_1\times p_2\times p_3$ factors through $N$ since (\ref{eq-c-g}) (resp.~(\ref{eq-e-k})) corresponds to (\ref{claim1-(c)}) and (\ref{claim1-(e)}) (resp.~(\ref{claim1-(d)}) and (\ref{claim1-(h)})).
Then it immediately follows that $\eta(S)$ is surjective and hence $\eta$ is an isomorphism. 

Therefore, $\mathfrak{H}$ is represented by $N$ and hence Proposition \ref{lem-N-univ} holds if Claim \ref{cl2} holds. 
We finish this step. 
\end{step}

\begin{step}\label{keyprop-step3}
In this final step, we prove Claim \ref{cl2}. 
To prove Claim \ref{cl2}, it suffices to check that (\ref{claim1-(a)})--(\ref{claim1-(i)}) are locally closed conditions.

We first deal with (\ref{claim1-(a)}) and (\ref{claim1-(b)}). 
By \cite[Th\'eor\`eme (12.2.1) and (12.2.4)]{EGA}, the subset 
$$U_1:=\{ s\in H|\text{ $\mathcal{U}_s$ is geometrically connected and geometrically normal}\}$$
is open. 
By \cite[Proposition 12.93]{gortz-wedhorn}, the subset 
$$U_2:=\{s\in U_1\,|\, \text{$p_{1,s} \colon \mathcal{U}_s\to \mathbb{P}^{d_1-1}_s$ and $p_{2,s} \colon \mathcal{U}_s\to \mathbb{P}^{d_2-1}_s$ are closed immersions}\}\subset U_{1}$$ is also open.  

Next, we treat (\ref{claim1-(c)}). 
We put 
$$\tilde{\mathscr{A}}=p_{2,U_2}^*\mathcal{O}_{\mathbb{P}^{d_2-1}_{U_2}}(1)\otimes p_{1,U_2}^*\mathcal{O}_{\mathbb{P}^{d_1-1}_{U_2}}(-1).$$ 
Then the condition (\ref{claim1-(c)}) implies that 
$$\mathscr{A}'\sim_{T}p_{1,T}^*\mathcal{O}_{\mathbb{P}^{d_1-1}_T}(-1)\otimes p_{2,T}^*\mathcal{O}_{\mathbb{P}^{d_2-1}_T}(1)=\tilde{\mathscr{A}}_{T}.$$
Hence, the existence of $\mathscr{A}'$ satisfying (\ref{claim1-(c)}) is equivalent to the $\tilde{\pi}_{T}$-ampleness of $\tilde{\mathscr{A}}_T$ and the relations $\tilde{\mathscr{A}}_T^{\otimes I}\sim_Tp_{1,T}^*\mathcal{O}_{\mathbb{P}^{d_1-1}_{T}}(1)$ and $\tilde{\mathscr{A}}_T^{\otimes I+1}\sim_Tp_{2,T}^*\mathcal{O}_{\mathbb{P}^{d_2-1}_{T}}(1)$. By Corollary \ref{cor--hako}, there exists a locally closed subscheme 
$$U_3\subset U_2$$
 such that a morphism $T\to U_2$  factors through $U_3 \hookrightarrow U_{2}$ if and only if the relations $\tilde{\mathscr{A}}_T^{\otimes I}\sim_Tp_{1,T}^*\mathcal{O}_{\mathbb{P}^{d_1-1}_{T}}(1)$ and $\tilde{\mathscr{A}}_T^{\otimes I+1}\sim_Tp_{2,T}^*\mathcal{O}_{\mathbb{P}^{d_2-1}_{T}}(1)$ hold true. 
 Since $p_{1,U_3}$ is a closed immersion, $\tilde{\mathscr{A}}_{U_3}$ is $\tilde{\pi}_{{U_3}}$-ample.

For (\ref{claim1-(d)}), set 
$$U_4:=\left\{ \begin{array}{l} s \in U_3\end{array} \;\middle| \begin{array}{rl} 
\bullet&\text{$h^0(\mathcal{U}_{s},\tilde{\mathscr{A}}_{s}^{\otimes I})=d_1$,}\\
\bullet&\text{$h^0(\mathcal{U}_{s}, \tilde{\mathscr{A}}_{s}^{\otimes I+1})=d_2$, and}\\
\bullet&\text{both $\mathcal{O}_{U_3}^{\oplus d_1}\to H^0(\mathcal{U}_s,\tilde{\mathscr{A}}^{\otimes I}_{s})$ and} \\
&\text{$\mathcal{O}_{U_3}^{\oplus d_2}\to H^0(\mathcal{U}_s,\tilde{\mathscr{A}}^{\otimes I+1}_{s})$ are surjective}
\end{array}\right\}.$$
Then $U_4$ is open. 
Indeed, pick a point $s \in U_{4}$.
We take a line bundle $\mathcal{M}$ on $U_3$ such that $\tilde{\pi}_{{U_3}}^{*}\mathcal{M}\otimes\tilde{\mathscr{A}}_{{U_3}}^{\otimes I}\sim p_{1,\mathcal{U}_{U_3}}^*\mathcal{O}_{\mathbb{P}}^{d_1-1}(1)$. 
By the third condition in $U_{4}$ and the construction of $\mathcal{O}_{U_3}^{\oplus d_1}\to H^0(\mathcal{U}_s,\tilde{\mathscr{A}}^{\otimes I}_{s})$, we have
$$\tilde{\pi}_{{U_3}*}p_{1,U_3}^*\mathcal{O}(1)\cong \mathcal{M}\otimes_{\mathcal{O}_{U_3}} \tilde{\pi}_{{U_3}*} \tilde{\mathscr{A}}_{{U_3}}^{\otimes I} \longrightarrow H^0(\mathcal{U}_s,\tilde{\mathscr{A}}^{\otimes I}_{s})$$ is surjective. 
By \cite[III, Theorem 12.11]{Ha}, $\mathcal{M}\otimes_{\mathcal{O}_{U_3}} \tilde{\pi}_{{U_3}*} \tilde{\mathscr{A}}_{{U_3}}^{\otimes I}$ is locally free  near $s$ and  
\begin{equation*}\begin{split}\mathcal{M}\otimes_{\mathcal{O}_{U_3}} \tilde{\pi}_{{U_3}*} \tilde{\mathscr{A}}_{{U_3}}^{\otimes I}\otimes_{\mathcal{O}_{U_3}} \mathcal{O}_{U_3,s'}/\mathfrak{m}_{s'} \longrightarrow H^0(\mathcal{U}_{s'},\tilde{\mathscr{A}}^{\otimes I}_{s'})\end{split}\end{equation*} is an isomorphism for every point $s' \in U_{3}$ on some neighborhood of $s$, where $\mathfrak{m}_{s'}$ is the maximal ideal of $\mathcal{O}_{U_3,s'}$. 
Therefore we have $h^0(\mathcal{U}_{s'},\tilde{\mathscr{A}}_{s'}^{\otimes I})=d_1$ for every point $s'$ on some neighborhood of $s$, and the third condition in $U_{4}$ implies that $\mathcal{O}_{U_3}^{\oplus d_1} \to \mathcal{M}\otimes_{\mathcal{O}_{U_3}} \tilde{\pi}_{{U_3}*} \tilde{\mathscr{A}}_{{U_3}}^{\otimes I}$ is surjective at $s$. 
Then the morphism
\begin{equation*}\begin{split}\mathcal{O}_{U_3}^{\oplus d_1} \longrightarrow \tilde{\pi}_{{U_3}*}p_{1,U_3}^*\mathcal{O}(1)\cong\mathcal{M}\otimes_{\mathcal{O}_{U_3}} \tilde{\pi}_{{U_3}*} \tilde{\mathscr{A}}_{{U_3}}^{\otimes I}\longrightarrow H^0(\mathcal{U}_{s'},\tilde{\mathscr{A}}^{\otimes I}_{s'})\end{split}\end{equation*}
 is surjective for every point $s'$ on some neighborhood of $s$. 
 By the same argument, we see that $h^0(\mathcal{U}_{s'},\tilde{\mathscr{A}}_{s'}^{\otimes I})=d_2$ and $\mathcal{O}_{U_3}^{\oplus d_2}\to H^0(\mathcal{U}_{s'},\tilde{\mathscr{A}}^{\otimes I+1}_{s'})$  is surjective for every point $s'$ on some neighborhood of $s$. 
 In this way, if $s \in U_{4}$ then a neighborhood of $s$ is contained in $U_{4}$, which implies the openness of $U_{4}$.  

In this paragraph, we discuss the conditions (\ref{claim1-(e)}) and (\ref{claim1-(f)}). 
By Corollary \ref{cor--hako}, we may find a locally closed subscheme 
$$U_5\subset U_4$$ such that a morphism $T\to U_4$ factors through 
$U_5$ if and only if $\omega_{\mathcal{U}_T/T}^{[er]}\sim_Tp_{3,T}^*\mathcal{O}(1)$. 
Furthermore, the following subset 
\[
U_6=\{t\in U_5\,|\,\mathcal{U}_{\bar{t}}\text{ is klt}\}
\]
is open since $\omega_{\mathcal{U}_T/T}^{[r]}$ is a line bundle (see \cite[Corollary 4.10]{kollar-mmp}).

We will discuss (\ref{claim1-(g)}) for three paragraphs.  
We fix an $l \in \mathbb{Z}_{>0}$, and we  will discuss the surjectivity of  $\tilde{\pi}_{T*}\omega_{\mathcal{U}_T/T}^{[ler]}\to H^0(\mathcal{U}_{t}, \mathcal{O}_{\mathcal{U}_{t}}(lerK_{\mathcal{U}_{t}}))$ for every $t\in T$, which we call ``(\ref{claim1-(g)}$_l$)". 
For any  Noetherian affine scheme $\bar{U}$ with a morphism $\bar{U} \to U_{8}$, we define functors $F^{0}_{\bar{U}}$ and $F^{1}_{\bar{U}}$ that send an affine scheme $U'$ over $\bar{U}$ to $\tilde{\pi}_{{U'}*}\omega_{\mathcal{U}_{U'}/U'}^{[ler]}$ and $R^{1}\tilde{\pi}_{{U'}*}\omega_{\mathcal{U}_{U'}/U'}^{[ler]}$, respectively. 
These are the same functors as discussed in \cite[III, Section 12]{Ha}, and $F^{0}_{\bar{U}}$ is always left exact by the flatness of $\tilde{\pi}_{{U_{6}}*}\omega_{\mathcal{U}_{U_{6}}/U_{6}}^{[ler]}$. 
Pick an affine open subset $U\subset U_{6}$. 
We pick a Grothendieck complex $(K^{\bullet},d^{\bullet})$ for $\omega_{\mathcal{U}_U/U}^{[ler]}$. 
This is the same complex as in \cite[\S5, Lemma 1]{Ab} (see also \cite[III, Proposition 12.2]{Ha}). 
We define a coherent sheaf 
$$W^{1}:=\mathrm{Coker}\,(d^{0}:K^{0}\to K^{1})$$
 on $U$. For any affine morphism $g \colon T\to U$, the pullback of the complex $(g^*K^\bullet,g^*d^{\bullet})$ is a Grothendieck complex with respect to $\omega_{\mathcal{U}_T/T}^{[ler]}$ and 
$$g^{*}W^{1}=\mathrm{Coker}\,(g^{*}d^{0} \colon g^*K^{0}\to g^*K^{1}).$$
By applying Theorem \ref{thm--inv-pluri} and \cite[\S5, Corollary 2]{Ab} to the normalization of $U_{6}$, we see that $\tilde{\pi}_{{U_{6,\mathrm{red}}}}$ satisfies (\ref{claim1-(g)}$_{l}$). 
By \cite[III, Corollary 12.6 and Proposition 12,10]{Ha}, $F^{0}_{U_{\rm red}}$ is exact and hence $F^{1}_{U_{\rm red}}$ is left exact. 
By \cite[III, Proposition 12,4]{Ha}, $W^{1}\otimes_{\mathcal{O}_{U}} \mathcal{O}_{U_{\rm red}}$ is flat for any choice of $(K^{\bullet},d^{\bullet})$. 
By this fact and the flattening stratification for $W^{1}$, we get a closed subscheme $Z_{U}\subset U$ such that a morphism $g \colon  T \to U$ factors through $Z_{U}\hookrightarrow U$ if and only if $g^{*}W^{1}$ is flat, which is 
equivalent to the exactness of $F^{0}_{T}$ by \cite[III, Proposition 12.4]{Ha}. 
From this, we can check that $Z_U$ is independent of the choice of $(K^{\bullet},d^{\bullet})$ and hence $Z_{U}|_{U'}=Z_{U'}$ for any affine open embedding $U'\hookrightarrow U$. 
By these facts, we can construct a closed subscheme 
$$U^{(l)}_{7}\subset U_{6}$$ by gluing all $Z_{U}$ for affine open subsets $U\subset U_{6}$. 

By construction, a morphism $g'  \colon  T'\to U_{6}$ factors through $U^{(l)}_{7}\hookrightarrow U_{6}$ if and only if $g'^{*}_{V}W^{1}$ is flat for any affine open subsets $V\subset T'$ and $U\subset U_{6}$ with the induced morphism $g'_{V} \colon  V\to U$. 
By \cite[III, Proposition 12.4]{Ha}, $g'^{*}_{V}W^{1}$ is flat if and only if $F^{1}_{V}$ is left exact. 
By recalling that $F^{0}_{V}$ is always left exact, we see that $F^{1}_{V}$ is left exact if and only if $F^{0}_{V}$ is exact. 
Thus, $g'  \colon  T'\to U_{6}$ factors through $U^{(l)}_{7}\hookrightarrow U_{6}$ if and only if $F^{0}_{V}$ is exact for any affine open subset $V\subset T'$. By the argument of cohomology and base change \cite[III, 12.5, 12.6, and 12.10]{Ha}, the exactness of $F^{0}_{V}$ for every $V\subset T'$ is equivalent to the condition (\ref{claim1-(g)}$_l$). 
In this way, $g'  \colon  T'\to U_{6}$ factors through $U^{(l)}_{7}\hookrightarrow U_{6}$ if and only if the morphism $\tilde{\pi}_{{T'}} \colon  \mathcal{U}_{T'}\to T'$ satisfies (\ref{claim1-(g)}$_l$). 

By the above argument, we have a sequence of closed subschemes of $U_{6}$
$$U_{6} \supset U^{(1)}_{7} \supset U^{(1)}_{7} \cap U^{(2)}_{7}:=U^{(1)}_{7} \times_{U_{6}} U^{(2)}_{7} \supset \cdots,$$
and the Noetherian property of $U_{6}$ implies that the above sequence is stationary and
$$U_{7}:=\bigcap_{l\in \mathbb{Z}_{>0}}U^{(l)}_{7}\subset U_{6}$$
is well-defined as a closed subscheme. 
By the construction of $U^{(l)}_{7}$, a morphism $T \to U_{7}$ factors through $U_{7}\hookrightarrow U_{6}$ if and only if $\mathcal{U}_T\to T$ satisfies (\ref{claim1-(g)}). 
We finish to discuss (\ref{claim1-(g)}).

By Theorem \ref{thm--inv-pluri} and applying the same argument as in the construction of $U_4$, we see that
\[
U_{8}:=\{s\in U_{7}\,|\,\text{$\mathcal{O}_{U_{7}}^{\oplus d_3}\to H^0(\mathcal{U}_s,\mathcal{O}_{\mathcal{U}_s}(erK_{\mathcal{U}_s}))$ is surjective, $h^0(\mathcal{U}_{\overline{s}},\mathcal{O}_{\mathcal{U}_{\overline{s}}}(erK_{\mathcal{U}_{\overline{s}}}))=d_3$}\}
\]
is open. 
This corresponds to (\ref{claim1-(h)}).

Finally, we discuss the condition (\ref{claim1-(i)}). 
We put 
 $$\widetilde{\mathcal{C}} :=\mathbf{Proj}_{U_{8}}(\bigoplus_{l\ge0}\tilde{\pi}_{{U_{8}*}}\omega_{\mathcal{U}_{U_{8}/U_{8}}}^{[ler]}).$$ 
 Let $f\colon\mathcal{U}_{U_8}\to\widetilde{\mathcal{C}} $ be the induced morphism.
Now the sheaf $\omega_{\mathcal{U}_{T}/T}^{[er]}$ is $\tilde{\pi}_{T}$-semiample by the construction of $U_5$.  
Furthermore, 
$$\widetilde{\mathcal{C}}_{T}=\mathbf{Proj}_{T}(\bigoplus_{l\ge0}\tilde{\pi}_{{T*}}\omega_{\mathcal{U}_{T}/T}^{[ler]})$$
for any morphism $T \to U_{8}$ by the condition (\ref{claim1-(g)}) and \cite[Theorem 12.11]{Ha}. 
Thus, $f_{\bar{s}} \colon \mathcal{U}_{\bar{s}}\to \widetilde{\mathcal{C}}_{\bar{s}}$ is the contraction induced by $eK_{\mathcal{U}_{\bar{s}}}$ for any geometric point $\bar{s}\in U_{8}$.
We consider the following set
$$U_9:=\left\{ \begin{array}{l} s \in U_8\end{array} \;\middle| \begin{array}{rl} &\bigl(p_{3,\overline{s}}^*\mathcal{O}(1)^2\cdot \tilde{\mathscr{A}}_{\overline{s}}^{d-2}\bigr)=0,\,\bigl(p_{3,\overline{s}}^*\mathcal{O}(1)\cdot \tilde{\mathscr{A}}_{\overline{s}}^{d-1}\bigr)=eruv,\\
& \mathrm{vol}(\tilde{\mathscr{A}}_{\overline{s}})\le w\text{ and }\mathrm{Ivol}(erK_{\mathcal{U}_{\overline{s}}})=eru
\end{array}\right\}.$$ 
We note that if the Iitaka dimension of $eK_{\mathcal{U}_{\overline{s}}}$ is one, we have $\mathrm{Ivol}(erK_{\mathcal{U}_{\overline{s}}})=r \cdot \mathrm{Ivol}(eK_{\mathcal{U}_{\overline{s}}})$.
This fact and (\ref{claim1-(e)}) show that a point $s \in U_{8}$ is contained $U_{9}$ if and only if $(\mathcal{U}_{\overline{s}},\tilde{\mathscr{A}}_{\overline{s}})\to \widetilde{\mathcal{C}}_{\overline{s}}$ satisfies (ii)--(iv) of $\mathfrak{Z}_{d, v,u,w}$. 
We will check that $U_{9}$ is open. 
By applying Theorem \ref{thm--inv-pluri} to the normalization of $U_{8}$, we see that the function 
$$ U_8\ni s \mapsto h^0(\mathcal{U}_{\overline{s}},\mathcal{O}_{\mathcal{U}_{\overline{s}}}(emrK_{\mathcal{U}_{\overline{s}}})) $$ is locally constant for every $m \in \mathbb{Z}_{>0}$.  
We also see that 
$$U_8\ni s\mapsto \Bigl((p_{3,s}^*\mathcal{O}(1)^2\cdot \tilde{\mathscr{A}}_{s}^{d-2}),\,(p_{3,s}^*\mathcal{O}(1)\cdot \tilde{\mathscr{A}}_{s}^{d-1}),\mathrm{vol}(\tilde{\mathscr{A}}_s)\Bigr)\in\mathbb{Q}^3$$
is locally constant by the flatness. 
Therefore, we see that $U_{9}$ is open. 
Now it suffices to show the uniform adiabatic K-stability of $f_{\bar{s}}$ for any geometric point $\bar{s}\in U_9$.
If $u>0$, every $f_{\bar{s}}$ is uniformly adiabatically K-stable and hence we may set $N:=U_{9}$.
If $u<0$, then we apply Theorem \ref{op} and Example \ref{ex--calculation} to the normalization $U_{9}^{\nu}$ of $U_{9}$ and $\mathcal{U}_{U_{9}^{\nu}} \to \widetilde{\mathcal{C}}_{U_{9}^{\nu}}$, and we obtain an open subset
$$
W\subset U_{9}^{\nu}
$$ 
such that $s \in U_{9}^{\nu}$ is contained in $W$ if and only if $f_{\overline{s}}$ is uniformly adiabatically K-stable.
Let $\nu \colon U_{9}^{\nu}\to U_{9}$ be the morphism of the normalization. 
Since $\nu$ is surjective and closed and $W=\nu^{-1}(\nu(W))$, the set
\[
N=\nu(W)
\]
is open. Moreover, a geometric point $\overline{s}\in U_{9}$ is a point of $N$ if and only if $f_{\overline{s}}$ is uniformly adiabatically K-stable.
Thus, we finish discussing the condition (\ref{claim1-(i)}).

By the above argument, a morphism $T\to H$ factors through $N\hookrightarrow H$, if and only if there exists $\mathscr{A}'$ as in Claim \ref{cl2} such that $\tilde{\pi}_{T}\colon\mathcal{U}_T\to T$ and $\mathscr{A}'$ satisfy (\ref{claim1-(a)})--(\ref{claim1-(i)}). 
We finish the proof of Claim \ref{cl2}. 
 \end{step}
We complete the proof of Proposition \ref{lem-N-univ}.
\end{proof}

Now we are ready to show Theorem \ref{mod2}.  

\begin{proof}[Proof of Theorem \ref{mod2}] 
By Lemma \ref{lem-stack}, $\mathscr{M}_{d,v,u,w,r}$ is a stack. 
We have 
$$\mathscr{M}_{d,v,u,w,r}=\underset{d_1, \, d_2, \, d_3, \, h}{\bigsqcup} \mathscr{M}_{d_1,d_2,d_3,h}$$ as stacks. 
Thanks to Lemma \ref{lem--stack-decomp}, it suffices to check that $\mathscr{M}_{d_1,d_2,d_3,h}$ is a separated Deligne-Mumford stack of finite type over $\mathbbm{k}$ for the fixed $d_1$, $d_2$, $d_3$, and $h$. 

Fix $d_1$, $d_2$, $d_3$, and $h$. 
We put $N:=N_{d_1,d_2,d_3,h}$, where $N_{d_1,d_2,d_3,h}$ is in Lemma \ref{lem-N-univ}, and let $\pi_{N}\colon (\mathcal{U}_{N}, \tilde{\mathscr{A}}_{N})\to N$ be the universal family in Lemma \ref{lem-N-univ}. 
We check that
 $$\mathscr{M}_{d_1,d_2,d_3,h}\cong [N/PGL(d_1)\times PGL(d_2)\times PGL(d_3)].$$
We first construct a morphism from $\mathscr{M}_{d_1,d_2,d_3,h}$ to $[N/PGL(d_1)\times PGL(d_2)\times PGL(d_3)]$. 
By regarding $N$ and $\mathscr{M}_{d_1,d_2,d_3,h}$ as stacks and using $(\mathcal{U}_{N}, \tilde{\mathscr{A}}_{N})\to N$, we get a morphism $N\to \mathscr{M}_{d_1,d_2,d_3,h}$  between stacks.
Take a scheme $S$ and $g \colon (\mathcal{X},\mathscr{A})\to \mathcal{C}\in\mathscr{M}_{d_1,d_2,d_3,h}(S)$. 
By the $2$-Yoneda lemma (cf.~\cite[Proposition 3.2.2]{Ols}) and regarding $S$ and $\mathscr{M}_{d_1,d_2,d_3,h}$ as stacks, we can find a morphism $S\to\mathscr{M}_{d_1,d_2,d_3,h}$ that corresponds to $g \colon (\mathcal{X},\mathscr{A})\to \mathcal{C}$. 
For any \'{e}tale covering $S'\to S$ such that the pullback of $\mathscr{A}$ is represented by a $\pi_{\mathcal{X}_{S'}}$-ample line bundle $\mathscr{A}'$, the definition of $\mathscr{M}_{d,v,u,w,r}$ and \cite[III, Theorem 12.11]{Ha} imply that $\pi_{\mathcal{X}_{S'}*}\mathscr{A}'^{\otimes I}$, $\pi_{\mathcal{X}_{S'}*}\mathscr{A}'^{\otimes I+1}$, and $\pi_{\mathcal{X_{S'}}*}\omega_{\mathcal{X_{S'}}/S'}^{[er]}$ are locally free sheaves of rank $d_1$, $d_2$, and $d_3$, respectively.
Since $N$ is the scheme representing $\mathfrak{H}$ (Proposition \ref{lem-N-univ}), $S'\times_{\mathscr{M}_{d_1,d_2,d_3,h}}N$ is represented by
\begin{equation*}\label{eq-isom-princ.fiber-bundle}
 \begin{split}
{\rm V}_{(\mathcal{X}_{S'},\mathscr{A}_{S'})}:=\mathrm{Isom}_{S'}\bigl(\mathbb{P}_{S'}(\pi_{\mathcal{X}_{S'}*}\mathscr{A}'^{\otimes I}),\mathbb{P}^{d_1-1}_{S'}\bigr)&\times_{S'}\mathrm{Isom}_{S'}\bigl(\mathbb{P}_{S'}(\pi_{\mathcal{X}_{S'}*}\mathscr{A}'^{\otimes I+1}),\mathbb{P}_{S'}^{d_2-1}\bigr)\\ 
&\times_{S'}\mathrm{Isom}_{S'}\bigl(\mathbb{P}_{S'}(\pi_{\mathcal{X}_{S'}*}\omega_{\mathcal{X}_{S'}/{S'}}^{[er]}),\mathbb{P}_{S'}^{d_3-1}\bigr). \end{split} \end{equation*}
Thus, we can think $S'\times_{\mathscr{M}_{d_1,d_2,d_3,h}}N$ of a principal $PGL(d_1)\times PGL(d_2)\times PGL(d_3)$-bundle over $S'$.
In particular, $S'\times_{\mathscr{M}_{d_1,d_2,d_3,h}}N$ is affine over $S'$. Hence, $S\times_{\mathscr{M}_{d_1,d_2,d_3,h}}N$ is represented by an affine scheme over $S$ (\cite[Proposition 4.4.9]{Ols}).
Then $S\times_{\mathscr{M}_{d_1,d_2,d_3,h}}N$ is a principal $PGL(d_1)\times PGL(d_2)\times PGL(d_3)$-bundle over $S$ (cf.~\cite[Proposition 2.36]{FGA}), and the natural morphism $S\times_{\mathscr{M}_{d_1,d_2,d_3,h}}N\to N$ is $PGL(d_1)\times PGL(d_2)\times PGL(d_3)$-equivariant by Proposition \ref{lem-N-univ}. 
For each scheme $S$, by considering the map 
$$(S\to \mathscr{M}_{d_1,d_2,d_3,h})\mapsto (S\times_{\mathscr{M}_{d_1,d_2,d_3,h}}N\to N)$$ 
and using the $2$-Yoneda lemma to the left hand side, 
we obtain a morphism 
$$\xi\colon\mathscr{M}_{d_1,d_2,d_3,h}\longrightarrow [N/PGL(d_1)\times PGL(d_2)\times PGL(d_3)]$$
between stacks.

In this paragraph we prove that $\xi$ is an isomorphism. 
By \cite[Proposition 3.1.10]{Ols}, it suffices to show the full faithfulness and the essential surjectivity of 
$$\xi(S):=\xi|_{\mathscr{M}_{d_1,d_2,d_3,h}(S)}\colon\mathscr{M}_{d_1,d_2,d_3,h}(S)\longrightarrow [N/PGL(d_1)\times PGL(d_2)\times PGL(d_3)](S)$$
 for a fixed scheme $S$. 
To prove the full faithfulness, we pick objects $g \colon (\mathcal{X},\mathscr{A})\to \mathcal{C}$ and $g' \colon (\mathcal{X}',\mathscr{A}')\to \mathcal{C}'$ of $\mathscr{M}_{d_1,d_2,d_3,h}(S)$.
Taking an \'etale covering of $S$, we may assume that $\mathscr{A}$ and $\mathscr{A'}$ are line bundles. 
By construction, $\xi(S)$ defines
$$
\xymatrix
{
 \mathrm{Isom}_{S}((\mathcal{X},\mathscr{A}),(\mathcal{X}',\mathscr{A}'))\ar@{}[d]|*{\rotatebox[origin=c]{90}{$\in$}}&& {\rm Hom}({\rm V}_{(\mathcal{X},\mathscr{A})}, {\rm V}_{(\mathcal{X'},\mathscr{A'})})\ar@{}[d]|*{\rotatebox[origin=c]{90}{$\in$}}\\
\phi \ar@{|->}[rr]&&\bigl((\rho_{1},\rho_{2},\rho_{3})\mapsto (\rho_{1}\circ \phi^{-1},\rho_{2}\circ \phi^{-1},\rho_{3}\circ \phi^{-1})\bigr).
}
$$
From this, the full faithfulness of $\xi(S)$ follows.
For the essential surjectivity, we pick any object $\alpha\colon\mathcal{P}\to N$ of $[N/PGL(d_1)\times PGL(d_2)\times PGL(d_3)](S)$. 
Here, $\mathcal{P}$ is a principal $PGL(d_1)\times PGL(d_2)\times PGL(d_3)$-bundle over $S$ and $\alpha$ is $PGL(d_1)\times PGL(d_2)\times PGL(d_3)$-equivariant. 
By \cite[Corollary 1.3.10]{Ols}, there exists an \'{e}tale covering $\beta\colon S'\to S$ such that $S'\times_S\mathcal{P}$ has a section $S'\to  S'\times_S\mathcal{P}$. 
Let $\sigma\colon S'\to  S'\times_S\mathcal{P}\to N$ be the composition of the section and the natural morphism, and let $(\mathcal{U}_{S'},\tilde{\mathscr{A}}_{S'})\to\widetilde{\mathcal{C}}_{S'}$ be an object of $\mathscr{M}_{d_1,d_2,d_3,h}(S')$ defined by the pullback of the universal family $(\mathcal{U}_N,\tilde{\mathscr{A}}_N)\to\widetilde{\mathcal{C}_N}$ via $\sigma$.
Then there exists an $S'\times N$-isomorphism $S'\times_S\mathcal{P}\to S'\times_{\mathscr{M}_{d_1,d_2,d_3,h}}N$, where $S'\to\mathscr{M}_{d_1,d_2,d_3,h}$ is the morphism corresponding to $(\mathcal{U}_{S'},\tilde{\mathscr{A}}_{S'})\to\widetilde{\mathcal{C}}_{S'}$.
By this discussion and the full faithfulness, $\xi(S)$ is essentially surjective.
Thus $\xi$ is an isomorphism. 
In this way, $\mathscr{M}_{d_1,d_2,d_3,h}$ is categorically equivalent to $[N/PGL(d_1)\times PGL(d_2)\times PGL(d_3)]$.
Thus, $\mathscr{M}_{d_1,d_2,d_3,h}$ is an Artin stack of finite type over $\mathbbm{k}$. 

In the rest of the proof, we will prove that $\mathscr{M}_{d_1,d_2,d_3,h}$ is a separated Deligne-Mumford stack with a coarse moduli space. 
By Theorem \ref{dm}, it suffices to prove that the diagonal morphism is finite. 
By Corollaries \ref{plusfin} and \ref{fin2}, we only need to prove that the diagonal morphism is proper. For any scheme $S$ and $g_i \colon (\mathcal{X}_i,\mathscr{A}_i)\to \mathcal{C}_i\in\mathscr{M}_{d_1,d_2,d_3,h}(S)$ for $i=1,2$, it suffices to show that the scheme 
$\mathscr{I} :=\mathrm{Isom}_S((\mathcal{X}_{1},\mathscr{A}_{1}),(\mathcal{X}_{2},\mathscr{A}_{2}))$ 
is proper over $S$.   
By \cite[Proposition 2.36]{FGA} and taking an \'{e}tale covering of $S$, we may assume that $\mathscr{A}_1$ and $\mathscr{A}_2$ are represented by line bundles. 
By abuse of notation, we denote them by $\mathscr{A}_1$ and $\mathscr{A}_2$, respectively. 
Then $\mathscr{I}$ is locally quasi-projective over $S$ (Subsection \ref{subsec-Hilb}).

Since the problem is local, by shrinking $S$, we may assume that {{$\mathscr{I}$ is quasi-projective over $S$}} and there exist morphisms $\gamma_i \colon S\to N$ ($i=1,2$) such that $\gamma_i^*\bigl(f_N\colon(\mathcal{U}_N,\tilde{\mathscr{A}}_{N})\to\widetilde{\mathcal{C}}_{N}\bigr)=g_i$, where $f_N$ is the canonical morphism of the ample model of $\omega_{\mathcal{U}_{T}/T}^{[er]}$. 
By the morphism $S\to N\times N$ naturally induced from $\gamma_{1}$ and $\gamma_{2}$, we obtain
 \begin{align*}
 \mathscr{I}=\mathrm{Isom}_{N\times N}\bigl((\mathcal{U}_1,\mathscr{A}_1),(\mathcal{U}_2,\mathscr{A}_2)\bigr)\times_{N\times N}S,
 \end{align*}
where $\mathcal{U}_1$ (resp.~$\mathcal{U}_2$) is the base change $\mathcal{U}_N\times_N(N\times N)$ by the first (resp.~second) projection $N\times N\to N$ and $\mathscr{A}_1$ (resp.~$\mathscr{A}_2$) is the pullback of $\tilde{\mathscr{A}}$.
From this, we may replace $S$ by $N \times N$.
Hence, we may assume that $S$ is of finite type over $\mathbbm{k}$.
By \cite[Corollary 13.101]{gortz-wedhorn}, it suffices to prove that the natural morphism $\mathscr{I}\times_{S}(S\times \mathbb{A}^n)\to S\times\mathbb{A}^n$, which we denote by $\varphi$, is a closed map for every $n\in\mathbb{Z}_{>0}$. We pick a closed subset $Z\subset \mathscr{I}\times\mathbb{A}^n$. 
Note that $\varphi(Z)$ is constructible.
By Lemma \ref{const--lem}, to prove the closedness of $\varphi(Z)$ it suffices to show that for any morphism $q \colon C\to \overline{\varphi(Z)}$ from a curve $C$ such that $q^{-1}(\varphi(Z))$ is dense in $C$, we have $q(C) \subset \varphi(Z)$. 
Consider the scheme 
$$\mathscr{J} :=(\mathscr{I}\times \mathbb{A}^n)\times_{S\times \mathbb{A}^n}C\cong \mathscr{I}\times_{S}C.$$
 Since $\mathscr{J}$ is quasi-projective, we can find a curve $D$ with a morphism $D \to \mathscr{J}$ such that the composition $D \to \mathscr{J} \to C$ is a dominant morphism. 
By considering a compactification of $D$ over $C$, we obtain a curve $\overline{D}$ such that $\overline{D} \to C$ is surjective and $D$ is an open subset of $\overline{D}$. 
Then $\mathscr{I} \times_{S}D$ is an open subset of $\mathscr{I} \times_{S}\overline{D}$, and $\mathscr{I} \times_{S}D \cong \mathscr{J} \times_{C}D \to D$ has a section $D' \subset \mathscr{I} \times_{S}D$. 
This section can be extended to a section $\overline{D}'$ of 
$(\mathscr{I}\times \mathbb{A}^n)\times_{S\times \mathbb{A}^n}\overline{D}\cong \mathscr{I} \times_{S}\overline{D} \to \overline{D}$. 
Indeed, this fact follows from Proposition \ref{plussep} if $u>0$.
If $u<0$, then $\mathcal{C}_{\overline{D}}$ of any object $f \colon (X,A)\to \mathcal{C}_{\overline{D}}$ of $ \mathscr{M}_{d_1,d_2,d_3,h}(\overline{D})$ is a $\mathbb{P}^1$-bundle over $\overline{D}$ by Tsen's theorem (cf.~\cite[V \S2]{Ha}). 
Thus, we can apply Theorem \ref{sep2} and obtain $\overline{D}'$.
Then 
\begin{equation*}
\begin{split}
p(C)=&{\rm Im}(\overline{D} \to C \to S\times \mathbb{A}^n)\\
=&{\rm Im}(\overline{D}'\hookrightarrow (\mathscr{I}\times \mathbb{A}^n)\times_{S\times \mathbb{A}^n}\overline{D}\to \mathscr{I}\times \mathbb{A}^n \overset{\varphi}{\to}S\times \mathbb{A}^n)\subset \varphi(Z).
\end{split}
\end{equation*}
Hence, $\varphi$ is a closed map, which implies that the diagonal morphism is proper. 
It follows from this that $\mathscr{M}_{d_1,d_2,d_3,h}$ is separated.

By the above argument, $\mathscr{M}_{d_1,d_2,d_3,h}$ is a separated Deligne-Mumford stack of finite type over $\mathbbm{k}$ with the coarse moduli space. 
We complete the proof.
 \end{proof}

 Now we prove Theorem \ref{quesmain}. 
More specifically, we prove that $\mathscr{M}_{d,v,u,r}$ is an open and closed substack of $\mathscr{M}_{d,v,u,r,w}$ for some $w\in\mathbb{Z}_{>0}$.

\begin{proof}[Proof of Theorem \ref{quesmain}]
We will freely use the notations $\mathfrak{Z}_{d,v,u}$ and $\mathfrak{Z}_{d,v,u,w}$ in \S\ref{intro}. 

We first check that for any scheme $S$ and $f\colon(\mathcal{X},\mathscr{A})\to \mathcal{C} \in \mathscr{M}_{d,v,u,r}(S)$, there exists $\mathscr{L}\in\mathrm{Pic}_{\mathcal{C}/S}(S)$ such that $\mathscr{L}_{\bar{s}}=\mathcal{O}_{\mathbb{P}^1}(1)$ for any geometric point $\bar{s}\in S$.  
Since $u<0$, the morphism $\mathcal{C}\to S$ is a smooth morphism whose geometric fibers are $\mathbb{P}^1$ (cf.~Theorem \ref{Bsemi}). 
Thus, there is an \'{e}tale covering $S'\to S$ such that $\mathcal{C}\times_SS'\cong\mathbb{P}^1_{S'}$ (cf.~the first paragraph of the proof of Theorem \ref{op}).
 Then $\mathcal{O}_{\mathbb{P}^1_{S'}}(1)$ satisfies $p_{1,\mathbb{P}^1_{S'}}^*\mathcal{O}_{\mathbb{P}^1_{S'}}(1)\cong p_{2,\mathbb{P}^1_{S'}}^*\mathcal{O}_{\mathbb{P}^1_{S'}}(1)$, where $p_{1}\colon S'\times_SS'\to S'$ (resp.~$p_{2}\colon S'\times_SS'\to S'$) is the first (resp.~second) projection.
Since $\mathrm{Pic}_{\mathcal{C}/S}$ is an \'{e}tale sheaf, there exists $\mathscr{L}\in\mathrm{Pic}_{\mathcal{C}/S}(S)$ that corresponds to $\mathcal{O}_{\mathbb{P}^1_{S'}}(1)$ under the canonical injection $\mathrm{Pic}_{\mathcal{C}/S}(S)\to\mathrm{Pic}_{\mathcal{C}/S}(S')$.
By definition, it is easy to check that $\mathscr{L}_{\bar{s}}=\mathcal{O}_{\mathbb{P}^1}(1)$ for any geometric point $\bar{s}\in S$.  

By the same argument as in the proof of Lemma \ref{lem-stack}, it follows that $\mathscr{M}_{d,v,u,r}$ is a category fibered in groupoids. 
By Proposition \ref{prop--nefthreshold-volume}, we can find a positive integer $w'$, depending only on $d$, $v$, and $u$, such that for any $f\colon(X,0,A)\to C\in\mathfrak{Z}_{d, v,u}$ together with a general fiber $F$ of $f$, the divisor $A+t_{A}F$ is ample and ${\rm vol}(A+t_{A}F) \leq w'$ for some integer $t_A$. 
Then 
$$f\colon(X,0,A+t_{A}F)\to C\in\mathfrak{Z}_{d, v,u,w'}.$$ 
Pick integers $w_1,\,\cdots ,\,w_k\in(w',w'+dv]$ such that for each $1 \leq i \leq k$, there is an object $f_{i}\colon(X_{i},0,A_{i})\to C_{i}\in\mathfrak{Z}_{d, v,u,w'+dv}$ such that ${\rm vol}(A_{i})=w_{i}$. 
Note that ${\rm vol}(A') \in \mathbb{Z}$ for every $f'\colon(X',0,A')\to C'\in\mathfrak{Z}_{d, v,u,w'+dv}$ because $A'$ is a line bundle.
Hence, we have 
$${\rm vol}(A') =w_{i}$$
for some $i$ if ${\rm vol}(A') >w'$. 

For each $1 \leq i \leq k$, we set $\mathscr{M}_{w_i}$ as the open and closed substack of $\mathscr{M}_{d,v,u,w'+dv,r}$ that parametrizes $f\colon(\mathcal{X},\mathscr{A})\to\mathcal{C}$ such that $\mathrm{vol}(\mathscr{A}_{\overline{s}})=w_i$ for all geometric points $\overline{s}\in S$.
Then there is a natural morphism
$$\gamma\colon \mathscr{M}_{w_1}\sqcup\ldots\sqcup\mathscr{M}_{w_k} \longrightarrow \mathscr{M}_{d,v,u,r}$$
between categories fibered in groupoids. 
If $\gamma$ is an isomorphism, then $\mathscr{M}_{d,v,u,r}$ is an open and closed substack of $\mathscr{M}_{d,v,u,w'+dv,r}$, and Theorem \ref{quesmain} immediately follows from Theorem \ref{mod2}. 
Therefore, it suffices to prove that $\gamma$ is an isomorphism. 

For a scheme $S$, we regard $f\colon(\mathcal{X},\mathscr{A})\to\mathcal{C}\in\mathscr{M}_{w_j}(S)$ as an object of $\mathscr{M}_{d,v,u,r}(S)$.
It suffices to show the full faithfulness and the essential surjectivity of 
$$\gamma(S):=\gamma|_{\mathscr{M}_{w_1}\sqcup\ldots\sqcup\mathscr{M}_{w_k}(S)}$$ for any scheme $S$.
Firstly, we prove the full faithfulness.
By the definitions of $\mathscr{M}_{d,v,u,r}$ and $\mathscr{M}_{d,v,u,w'+dv,r}$, we see that $\gamma(S)$ is faithful.
To show the fullness, take two objects $f\colon(\mathcal{X},\mathscr{A})\to \mathcal{C}$ and $f'\colon(\mathcal{X}',\mathscr{A}')\to \mathcal{C}'$ of $\mathscr{M}_{w_1}\sqcup\ldots\sqcup\mathscr{M}_{w_k}(S)$ and an isomorphism $\alpha\colon f\to f'$ in $\mathscr{M}_{d,v,u,r}(S)$.
This means that $\alpha\colon\mathcal{X}\to \mathcal{X}'$ is an $S$-isomorphism and there exists an element $\mathscr{B}\in\mathrm{Pic}_{\mathcal{C}/S}(S)$ such that $\alpha^*\mathscr{A}'=\mathscr{A}\otimes f^*\mathscr{B}$.
To show that $\alpha$ comes from an isomorphism in $\mathscr{M}_{w_1}\sqcup\ldots\sqcup\mathscr{M}_{w_k}$, it suffices to prove $\mathscr{B}=0$ as an element of $\mathrm{Pic}_{\mathcal{C}/S}(S)$.
For any geometric point $\bar{s}\in S$, we have $\mathscr{B}_{\bar{s}}\sim\mathcal{O}_{\mathbb{P}^1}(m)$ for some $m\in\mathbb{Z}$ and 
\[
\mathrm{vol}(\mathscr{A}'_{\bar{s}})=\mathrm{vol}(\mathscr{A}_{\bar{s}})+dmv.
\]
By the property of $w_1,\,\cdots ,\,w_k$, there exist two indices $i$ and $j$ such that $w_i=\mathrm{vol}(\mathscr{A}_{\bar{s}})$ and $w_j=\mathrm{vol}(\mathscr{A}'_{\bar{s}})$.
Since $|w_i-w_j|<dv$, we have that $m=0$.
This implies $\mathscr{B}_{\bar{s}}\sim\mathcal{O}_{\mathbb{P}^1}$. 
By the proof of \cite[\S0.5, (b)]{GIT}, we see that $\mathscr{B}=0$ as an element of $\mathrm{Pic}_{\mathcal{C}/S}(S)$.
From this, it follows that $\gamma(S)$ is a fully faithful functor.

Secondly, we prove the essential surjectivity of $\gamma(S)$. 
We fix an object $f\colon(\mathcal{X},\mathscr{A})\to \mathcal{C}$ of $\mathscr{M}_{d,v,u,r}(S)$ and a geometric point $\bar{s}\in S$.
As in the third paragraph of this proof, there is $p\in\mathbb{Z}$ such that $\mathscr{A}_{\bar{s}}\otimes f_{\bar{s}}^*\mathscr{L}_{\bar{s}}^{\otimes p}$ is ample and
$$\mathrm{vol}(\mathscr{A}_{\bar{s}}\otimes f_{\bar{s}}^*\mathscr{L}_{\bar{s}}^{\otimes p})\le w'$$
for every geometric point $\overline{s} \in S$. 
Then there exists an open neighborhood $U$ of $\bar{s}$ such that $\mathscr{A}_{\bar{t}}\otimes f_{\bar{t}}^*\mathscr{L}_{\bar{t}}^{\otimes p}$ is ample for any geometric point $\bar{t}\in U$.
By shrinking $U$, we may assume that 
$\mathrm{vol}(\mathscr{A}_{\bar{t}}\otimes f_{\bar{t}}^*\mathscr{L}_{\bar{t}}^{\otimes p})$
is independent of $\bar{t}\in U$. 
This is because the function 
 \[
 U\ni t\mapsto\mathrm{vol}(\mathscr{A}_{\bar{t}}\otimes f_{\bar{t}}^*\mathscr{L}_{\bar{t}}^{\otimes p})\in\mathbb{Z}_{>0}
 \]
 is locally constant.
Take a positive integer $q$ such that $\mathscr{A}_{\bar{t}}\otimes f_{\bar{t}}^*\mathscr{L}_{\bar{t}}^{\otimes p+q}$ is ample and 
$$w'<\mathrm{vol}(\mathscr{A}_{\bar{t}}\otimes f_{\bar{t}}^*\mathscr{L}_{\bar{t}}^{\otimes p+q})=\mathrm{vol}(\mathscr{A}_{\bar{t}}\otimes f_{\bar{t}}^*\mathscr{L}_{\bar{t}}^{\otimes p})+dqv\le w'+dv$$
for any geometric point $\bar{t}\in U$. 
Then we see that $\mathrm{vol}(\mathscr{A}_{\bar{t}}\otimes f_{\bar{t}}^*\mathscr{L}_{\bar{t}}^{\otimes p+q})=w_i$ for some $i$.
The above argument shows that for any geometric point $\bar{s}\in S$, there exists an open neighborhood $U$ of $\bar{s}$ such that $f|_U$ comes from $\mathscr{M}_{w_1}\sqcup\ldots\sqcup\mathscr{M}_{w_k}(U)$.
More precisely, there exists a set of open subsets $\{U_\lambda\subset S\}_{\lambda\in \Lambda}$ with $q_{\lambda}\in\mathbb{Z}$ such that there exists an integer $i\in [1,k]$ such that $\mathscr{A}_{\bar{s}}\otimes f_{\bar{s}}^*\mathscr{L}_{\bar{s}}^{\otimes q_{\lambda}}$ is ample and $\mathrm{vol}(\mathscr{A}_{\bar{s}}\otimes f_{\bar{s}}^*\mathscr{L}_{\bar{s}}^{\otimes q_{\lambda}})=w_i$ for any geometric point $\bar{s}\in U_{\lambda}$.
For any $\lambda,\lambda'\in\Lambda$, if $U_{\lambda}\cap U_{\lambda'}\ne\emptyset$, then
\[
\mathrm{vol}(\mathscr{A}_{\bar{s}}\otimes f_{\bar{s}}^*\mathscr{L}_{\bar{s}}^{\otimes q_{\lambda}})=\mathrm{vol}(\mathscr{A}_{\bar{s}}\otimes f_{\bar{s}}^*\mathscr{L}_{\bar{s}}^{\otimes q_{\lambda'}})+d(q_{\lambda}-q_{\lambda'})v
\]
for any geometric point $\bar{s}\in U_{\lambda}\cap U_{\lambda'}$.
Since we have $|d(q_{\lambda}-q_{\lambda'})v|<dv$ by construction, we have $q_{\lambda}=q_{\lambda'}$.
Then we can glue $\mathscr{A}\otimes f^*\mathscr{L}^{\otimes q_{\lambda}}|_{\pi_{\mathcal{X}}^{-1}(U_{\lambda})}$, and we obtain $\mathscr{A}'\in\mathrm{Pic}_{\mathcal{X}/S}(S)$.
By construction, there exists an element $\mathscr{B}\in\mathrm{Pic}_{\mathcal{C}/S}(S)$ such that $\mathscr{A}\otimes f^*\mathscr{B}=\mathscr{A}'$ and $f\colon(\mathcal{X},\mathscr{A}')\to \mathcal{C}\in\mathscr{M}_{w_1}\sqcup\ldots\sqcup\mathscr{M}_{w_k}(S)$.
This means that $\gamma(S)$ is essentially surjective.
Thus, we conclude that $\gamma$ is an isomorphism. 
We complete the proof.
\end{proof}

\begin{rem}\label{rem-samemoduli}    
In Theorem \ref{mod2}, we have constructed the moduli spaces depending on the choice of $r$ of Lemma \ref{lem--Cartierindex}. However, the reduced structures of these moduli stacks are independent of $r$. 
To see this, it suffices to show the following. 
\begin{itemize}
    \item[($*$)] Fix $r$ as in Lemma \ref{lem--Cartierindex} and $l\in\mathbb{Z}_{>0}$. 
    For every reduced scheme $S$, every object $f\colon (\mathcal{X},\mathscr{A})\to\mathcal{C} \in \mathscr{M}_{d,v,u,w,lr}(S)$ is an object of $\mathscr{M}_{d,v,u,w,r}(S)$.
\end{itemize}
To show ($*$), we only need to check (iii) and (iv) of $\mathscr{M}_{d,v,u,w,r}(S)$ for $f\colon (\mathcal{X},\mathscr{A})\to\mathcal{C} \in \mathscr{M}_{d,v,u,w,lr}(S)$ as above. 

We first check (iii).
Note that the function $S\ni s\mapsto \chi(\mathcal{X}_s,\mathcal{O}_{\mathcal{X}_s}(lerK_{\mathcal{X}_s}+m\mathscr{A}_s))$ is locally constant for every $m \in \mathbb{Z}$. 
By the base point freeness of $erK_{\mathcal{X}_{\overline{s}}}$ and considering the exact sequence
$$0 \longrightarrow \mathcal{O}_{\mathcal{X}_{\overline{s}}}(kerK_{\mathcal{X}_{\overline{s}}}+m\mathscr{A}_{\overline{s}}) \longrightarrow \mathcal{O}_{\mathcal{X}_{\overline{s}}}((k+1)erK_{\mathcal{X}_{\overline{s}}}+m\mathscr{A}_{\overline{s}}) \longrightarrow \mathcal{O}_{D}(m\mathscr{A}_{\overline{s}}|_{D}) \longrightarrow 0$$
for every $k\ge0$, where $D \in |erK_{\mathcal{X}_{\overline{s}}}|$ is a general member, we have
$$\chi(\mathcal{X}_s,\mathcal{O}_{\mathcal{X}_s}(kerK_{\mathcal{X}_s}+m\mathscr{A}_s))=\chi(\mathcal{X}_s,\mathcal{O}_{\mathcal{X}_s}(m\mathscr{A}_s))+k\chi(D,\mathcal{O}_{D}(m\mathscr{A}_s|_D))$$
for every $m \in \mathbb{Z}$. 
Considering the case $k=l$, we see that $\chi(D,\mathcal{O}_{D}(m\mathscr{A}_s|_D))$ is locally constant on $s\in S$.
Therefore, the function
$$S\ni s\mapsto \chi(\mathcal{X}_s,\mathcal{O}_{\mathcal{X}_s}(kerK_{\mathcal{X}_s}+m\mathscr{A}_s))$$
is locally constant for every $m \in \mathbb{Z}$ and $k\in\mathbb{Z}_{\ge0}$. 
By Theorem \ref{hullsdecomp}, there is the universal hull $\omega_{\mathcal{X}/S}^{[r]}$. 
By the definition of $r$ in Lemma \ref{lem--Cartierindex}, the sheaf $\omega_{\mathcal{X}_{\bar{s}}}^{[r]}$ is invertible for any geometric point $\bar{s}\in S$.  
From this, $\omega_{\mathcal{X}/S}^{[r]}$ exists as a line bundle. 
Therefore (iii) of $\mathscr{M}_{d,v,u,w,r}(S)$ is satisfied.

Next, we check (iv) of $\mathscr{M}_{d,v,u,w,r}(S)$. 
By applying Theorem \ref{thm--inv-pluri} to the normalization of $S$, the function 
$$S\ni s\mapsto {\rm dim}\,H^{0}(\mathcal{X}_s,\mathcal{O}_{\mathcal{X}_s}(kerK_{\mathcal{X}_s}))$$
is locally constant for every $k\ge0$. 
By \cite[\S5, Corollary 2]{Ab} and the reducedness of $S$, we see that (iv) of $\mathscr{M}_{d,v,u,w,r}(S)$ is satisfied. 

From the above discussion, we have $f\colon (\mathcal{X},\mathscr{A})\to\mathcal{C} \in \mathscr{M}_{d,v,u,w,r}(S)$.
Thus, we see that ($*$) holds and the reduced structures of $\mathscr{M}_{d,v,u,w,lr}$ and $\mathscr{M}_{d,v,u,w,r}$ are the same.
As we saw in the proof of Theorem \ref{quesmain}, $\mathscr{M}_{d,v,u,r}$ is an open and closed substack of $\mathscr{M}_{d,v,u,w,r}$ and hence its reduced structure does not depend on the choice of $r$. 
\end{rem}

  \section{Uniformity of adiabatic K-stability}\label{seckst}
  
  This section is devoted to show Theorem \ref{mainunif} and Corollary \ref{cor-main}.
 Throughout this section, we work over the field of complex numbers $\mathbb{C}$, and we will use $\mathfrak{F}_{d,n,v}$ and $\mathfrak{G}_{d, n, v, u}$ in \S \ref{Bousec}.
 We fix $d, \, n \in \mathbb{Z}_{>0}$, $u\in\mathbb{Q}$, and $v\in \mathbb{Q}_{>0}$. 
  We consider the following set for any $w \in \mathbb{Q}_{>0}$. 
  \begin{equation*}
\begin{split}
\mathfrak{G}_{d, n, v, u,w}^{\mathrm{Car}}:=&\left\{ f\colon (X,\Delta,A) \to C \;\middle|
\begin{array}{l}
\text{$f\colon (X,\Delta) \to C \in \mathfrak{G}_{d, n, v, u}$ and $A$ is an}\\
\text{ample Cartier divisor satisfying (iv) of}\\
\text{$\mathfrak{F}_{d,n,v}$ such that $\mathrm{vol}(A)\le w$.}
\end{array}\right\}.
\end{split}
\end{equation*}
Recall from by Theorem \ref{thm--eff-veryample} and Remark \ref{rem--ell-bound} that there exists $m\in\mathbb{Z}_{>0}$, depending only on $d, \, n, \, u$, and $v$, such that for any element $f\colon(X,\Delta,A)\to C$ of $\mathfrak{G}_{d, \frac{1}{n}\mathbb{Z}\cap[0,1], v, u, w}$, $mA$ is Cartier. 
 Thus, we can regard $f\colon(X,\Delta,mA)\to C$ as an element of $\mathfrak{G}_{d, n, m^{d-1}v, u, m^dw}^{\mathrm{Car}}$.

 First, we parametrize all elements of $\mathfrak{G}_{d, n, v, u,w}^{\mathrm{Car}}$ when $u\neq 0$.

\begin{prop}\label{bou6}
   Fix $d$, $n$, $u$, $v$, and $w$ as above such that $u\ne0$.
   Then there exist a log $\mathbb{Q}$-Gorenstein family $f\colon(\mathcal{X},\mathcal{D})\to S$, an $f$-ample Carter divisor $\mathscr{A}$ on $\mathcal{X}$, and an $S$-morphism $g\colon\mathcal{X}\to \mathcal{C}$ such that $S$ is a normal scheme of finite type over $\mathbb{C}$, $\mathcal{C}$ is a normal scheme which is smooth and projective over $S$, and moreover the following hold.
   \begin{itemize}
      \item
        We have $g_s \colon (\mathcal{X}_s,\mathcal{D}_s,\mathscr{A}_s)\to\mathcal{C}_s\in\mathfrak{G}_{d,n, v, u, w}^{\mathrm{Car}}$ for any closed point $s\in S$,
       and
\item
 for any element $h \colon (X,\Delta,A)\to C\in\mathfrak{G}_{d,n, v, u, w}^{\mathrm{Car}}$, there exist a closed point $s\in S$ and isomorphisms $\alpha\colon(\mathcal{X}_s,\mathcal{D}_s) \to (X,\Delta)$ and $\beta\colon \mathcal{C}_s\to C $ satisfying $h\circ\alpha=\beta\circ g_s$ and $\alpha^*A\sim \mathscr{A}_s$.
   \end{itemize}
  \end{prop}
  
\begin{proof}
The case when the boundary $\Delta$ is zero in the proposition easily follows from Proposition \ref{lem-N-univ}. 
In the general case, the proposition holds true by the standard argument of the boundedness and the idea in the proof of Proposition \ref{lem-N-univ}. 
\end{proof}

  The following is the key step to show Theorem \ref{mainunif} for the case when $u>0$.

  \begin{thm}\label{unif+}
 Let $f \colon (\mathcal{X},\mathcal{D})\to S$ be a log $\mathbb{Q}$-Gorenstein family and $\mathcal{A}$ an $f$-ample line bundle on $\mathcal{X}$. 
 Suppose that $K_{\mathcal{X}/S}+\mathcal{D}$ is nef over $S$ and the fiber $(\mathcal{X}_s,\mathcal{D}_s)$ over any closed point $s\in S$ is klt.
 
 Then there exists a positive rational number $\epsilon_0$ such that $(\mathcal{X}_s,\mathcal{D}_s,\epsilon\mathcal{A}_s+K_{\mathcal{X}_s}+\mathcal{D}_s)$ is specially K-stable for the fiber $(\mathcal{X}_s,\mathcal{D}_s,\mathcal{A}_s)$ over any closed point $s\in S$ and any rational number $\epsilon\in (0,\epsilon_0)$. Furthermore, there exists a positive rational number $\alpha$ such that
 \[
M^\mathrm{NA}_{\mathcal{D}_s}(\mathcal{Y},\mathcal{M})\ge\alpha (\mathcal{J}^{\epsilon\mathcal{A}_s+K_{\mathcal{X}_s}+\mathcal{D}_s})^\mathrm{NA}(\mathcal{Y},\mathcal{M})
\]
 for any rational number $\epsilon\in (0,\epsilon_0)$, closed point $s\in S$, and normal semiample test configuration $(\mathcal{Y},\mathcal{M})$ for $(\mathcal{X}_s,\epsilon\mathcal{A}_s+K_{\mathcal{X}_s}+\mathcal{D}_s)$. 
 \end{thm}

 \begin{proof}
 First, we note that $\mathcal{A}+K_{\mathcal{X}/S}+\mathcal{D}$ is $f$-ample. 
By \cite[Proposition 5.3]{BL}, there exists $\delta_0>0$ such that $\alpha_{(\mathcal{X}_s,\mathcal{D}_s)}(\mathcal{A}_s+K_{\mathcal{X}_s}+\mathcal{D}_s)\ge\delta_0$ for any closed point $s\in S$.  
 We also have $$\alpha_{(\mathcal{X}_s,\mathcal{D}_s)}(\mathcal{A}_s+K_{\mathcal{X}_s}+\mathcal{D}_s)\le \alpha_{(\mathcal{X}_s,\mathcal{D}_s)}(\epsilon\mathcal{A}_s+K_{\mathcal{X}_s}+\mathcal{D}_s)$$
 for $\epsilon \in (0, 1)$.
We put $d$ as the relative dimension of $f$. 
By Lemma \ref{lem-delta-alpha}, we have 
\begin{equation*}\tag{$2$}
\begin{split} 
 \delta_{(\mathcal{X}_s,\mathcal{D}_s)}(\epsilon \mathcal{A}_s+K_{\mathcal{X}_s}+\mathcal{D}_s)&\ge\frac{d+1}{d}\alpha_{(\mathcal{X}_s,\mathcal{D}_s)}(\epsilon\mathcal{A}_s+K_{\mathcal{X}_s}+\mathcal{D}_s)\label{eq-alpha-calculate}\\
 &\ge\frac{d+1}{d}\alpha_{(\mathcal{X}_s,\mathcal{D}_s)}(\mathcal{A}_s+K_{\mathcal{X}_s}+\mathcal{D}_s) \ge\frac{d+1}{d}\delta_0.\nonumber
 \end{split}
 \end{equation*}
 By Theorem \ref{k-plus}, there exists a positive rational number $C$, depending only on the numbers $(K_X+\Delta)^{d-i}\cdot A^{i}$ for $0\le i\le d$, such that $(X,\epsilon A+K_X+\Delta)$ is uniformly J$^{K_X+\Delta+C\epsilon(\epsilon A+K_X+\Delta)}$-stable for any $\epsilon>0$ and $d$-dimensional polarized klt pair $(X,\Delta,L)$ such that $K_X+\Delta$ is nef.
 By the flatness of $f$, $(K_{\mathcal{X}_s}+\mathcal{D}_s)^{d-i}\cdot \mathcal{A}_s^{i}$ are independent of $s$, hence we can choose $C$ so that $(\mathcal{X}_s,\epsilon\mathcal{A}_s+K_{\mathcal{X}_s}+\mathcal{D}_s)$ is J$^{K_{\mathcal{X}_s}+\mathcal{D}_s+C\epsilon(\epsilon \mathcal{A}_s+K_{\mathcal{X}_s}+\mathcal{D}_s)}$-stable for any $s\in S$. 
 By taking $\epsilon_{0}$ such that $0<\epsilon_0<\min\{\frac{(d+1)\delta_0}{dC},1\}$, we have 
  \[
M^\mathrm{NA}_{\mathcal{D}_s}(\mathcal{Y},\mathcal{M})\ge\left(\frac{(d+1)\delta_0}{d}-C\epsilon_0\right)(\mathcal{J}^{\epsilon\mathcal{A}_s+K_{\mathcal{X}_s}+\mathcal{D}_s})^\mathrm{NA}(\mathcal{Y},\mathcal{M})
\]
 for any rational number $\epsilon\in (0,\epsilon_0)$, closed point $s\in S$ and normal semiample test configuration $(\mathcal{Y},\mathcal{M})$ for $(\mathcal{X}_s,\epsilon\mathcal{A}_s+K_{\mathcal{X}_s}+\mathcal{D}_s)$.
 \end{proof}

Now we assume $u <0$. In this case, we need to show that the uniform ``convergence of the $\delta$-invariant" (cf.~\cite[Theorem D]{Hat}) holds for all polarized klt-trivial fibrations belonging to one family. 

\begin{prop}\label{lem-delta-conv}
Let $S$ be a normal variety and $f\colon\mathcal{X}\to \mathbb{P}^1_S$ be a contraction of normal varieties over $S$. 
Suppose that $\pi\colon(\mathcal{X},\mathcal{D})\to S$ is a log $\mathbb{Q}$-Gorenstein family such that any geometric fiber is a klt pair. 
Let $\mathcal{H}$ be a $\pi$-ample Cartier divisor on $\mathcal{X}$. 
Suppose further that there exists a positive real number $u$ such that
\[
\lim_{\epsilon\to0}\delta_{(\mathcal{X}_{\bar{s}},\mathcal{D}_{\bar{s}})}(\epsilon\mathcal{H}_{\bar{s}}+f_{\bar{s}}^*\mathcal{O}(1))\ge u
\]
for any geometric fiber $f_{\bar{s}}\colon(\mathcal{X}_{\bar{s}},\mathcal{D}_{\bar{s}},\mathcal{H}_{\bar{s}})\to \mathbb{P}^1$ over ${\bar{s}}\in S$. 

Then for any $\delta_0>0$, there exists a positive real number $\epsilon_0$ such that 
\[
\delta_{(\mathcal{X}_{\bar{s}},\mathcal{D}_{\bar{s}})}(\epsilon\mathcal{H}_{\bar{s}}+f_{\bar{s}}^*\mathcal{O}(1))\ge u-\delta_0
\]
for any rational number $\epsilon\in(0,\epsilon_0)$ and geometric point $\bar{s}\in S$.
\end{prop}

\begin{proof}
 By \cite[Theorem 6.6]{BL}, if 
 \[
\delta_{(\mathcal{X}_{s},\mathcal{D}_{s})}(\epsilon\mathcal{H}_{s}+f_{s}^*\mathcal{O}(1))\ge u-\delta_0
\]
holds for any closed point $s\in S$, then
 \[
\delta_{(\mathcal{X}_{\bar{s}},\mathcal{D}_{\bar{s}})}(\epsilon\mathcal{H}_{\bar{s}}+f_{\bar{s}}^*\mathcal{O}(1))\ge u-\delta_0
\]
also holds for any geometric point ${\bar{s}}\in S$ since the set of all closed points is Zariski dense. 
Thus, it suffices to show the assertion for all closed points of $S$.

First, we note that to show the assertion, we may freely shrink $S$ or replace $S$ by $S'$ with an \'{e}tale morphism $S'\to S$.
 Indeed, if we can prove Proposition \ref{lem-delta-conv} for a non-empty open subset $U\subset S$, then the assertion for $S$ follows from Noetherian induction. 
 Thus, we may assume as in the proof of Lemma \ref{lem--generic} that $S$ is smooth and there exists a diagram \begin{equation}\label{diag2}
\xymatrix{
\mathcal{Y}\ar[r]^{g}\ar[d]_{f'}&\mathcal{X}\ar[d]^{f}\\
\mathcal{W}\ar[r]_{h}&\mathbb{P}^1_S
}
\end{equation}
of projective morphisms, where $\mathcal{Y}$ and $\mathcal{W}$ are smooth varieties, with snc divisors $\Sigma$ on $\mathcal{W}$ and $\Xi$ on $\mathcal{Y}$ respectively
such that 
\begin{itemize}
\item
$h$ is birational and $g$ is a log resolution of $(\mathcal{X},\mathcal{D})$, 
\item
$f'$ is a contraction, 
\item
$\Xi\supset f'^{*}\Sigma \cup \mathrm{Supp}(g_{*}^{-1}\mathcal{D})\cup {\rm Ex}(g)$ and the vertical part of $\Sigma_{\mathcal{Y}}$ with respect to $f'$ maps into $\Sigma$,
\item
the restriction of $f':(\mathcal{Y},\Xi) \to \mathcal{W}$ over $\mathcal{W}\setminus \Sigma$ is log smooth.
\end{itemize}
Since $\mathcal{W}$ is isomorphic to $\mathbb{P}^1_S$ over any codimension one point of $\mathbb{P}^1_S$, we may shrink $S$ and assume that $\mathcal{W}\cong\mathbb{P}^1_S$. 
Taking a suitable \'etale morphism $T\to S$ and replacing $S$ by $T$, we may further assume that 
\begin{itemize}
    \item $(\mathcal{Y},\Xi)$ and $(\mathbb{P}^1_S,\Sigma)$ are log smooth over $S$ and any stratum of $\Xi$ or $\Sigma$ has geometrically integral fibers over $S$. 
\end{itemize}
Then we apply Lemma \ref{lem--generic} to $S$ and we conclude by shrinking $S$ that $\mathcal{B}_s=B_s$ for any closed point $s\in S$, where $\mathcal{B}$ (resp.~$B_s$) is the discriminant divisor with respect to $f$ (resp.~$f_s$).
We may also assume by shrinking $S$ that there exists a section $S'$ of $\mathbb{P}^1_S\to S$ disjoint from $\Sigma$.
Then, we show the following claim.

 \begin{claim}\label{lem-cl3}
   There exist positive real numbers $c$ and $\epsilon'_0$ that satisfy the following for any closed point $s\in S$ and rational number $\epsilon\in(0,\epsilon'_0)$.
\begin{enumerate}[{\rm (}i{\rm )}]
\item \label{claim-2-(1)}$\sup_{p}\mathrm{mult}_p(g_{s}^*D)_{\mathrm{hor}}\le c\epsilon$ for any effective $\mathbb{Q}$-divisor $D$ that is $\mathbb{Q}$-lineraly equivalent to $\epsilon\mathcal{H}_{s}+f_{s}^*\mathcal{O}(1)$ (cf.~(\ref{defn-hor-vert}) in Notation and Convention), where $p\in \mathcal{Y}_{s}$ runs over all closed points. Here, we set $\mathrm{mult}_p(D')=\mathrm{ord}_E(\mu^*D'-\mu_*^{-1}D')$ for any $\mathbb{Q}$-divisor $D'$, where $\mu$ is the blow up of $\mathcal{Y}_{s}$ at $p$ and $E$ is the exceptional divisor,
\item \label{claim-2-(2)} there exists a positive integer $m_{\epsilon}$ such that $m_\epsilon\epsilon\in\mathbb{Z}$ and it holds that for any irreducible component $G_{s}$ of a fiber $F_{s}$ of $f_{s}\circ g_{s}$ and $m\in\mathbb{Z}_{>0}$ such that $m_\epsilon|m$,
$$S_{m,\epsilon\mathcal{H}_{s}+f_{s}^*\mathcal{O}(1)}(G_{s})\le \frac{1}{2}T_{f_{s}^*\mathcal{O}(1)}(G_{s})+c\epsilon+c(m\epsilon)^{-1}.$$
Here, we set $T_{f_{s}^*\mathcal{O}(1)}(G_{s}):=\mathrm{ord}_{G_s}(F_{s})$.
\end{enumerate}
 \end{claim}

\begin{proof}[Proof of Claim \ref{lem-cl3}]
Let $\Sigma=\sum_jF_j$ and $(f\circ g)^*F_j=\sum a_{ij} G_i^{(j)}$ be the irreducible decompositions. 
We note that $G_i^{(j)}$ has only smooth and irreducible fibers over $S$ by the fifth condition of the diagram (\ref{diag2}).
Note also that we may shrink $S$ freely by the same reason as in the second paragraph of the proof of Proposition \ref{lem-delta-conv}.
We may further assume that $S$ is quasi-projective and hence there exists a very ample line bundle $\mathcal{A}$ on $\mathcal{Y}$. 
Set $d$ as the relative dimension of $f$ throughout this proof.

First, we deal with (\ref{claim-2-(1)}). 
We take $m'_1\in\mathbb{Z}_{>0}$ such that $m'_1\mathcal{A}-K_{\mathcal{Y}/S}-G'$ is ample, where $G'$ is an irreducible component of $(f\circ g)^*F_j$ or $(f\circ g)^*(S')$.
Since all smooth fibers of $f_{s}\circ g_{s}$ are linearly equivalent to $((f\circ g)^*(S'))|_{\mathcal{Y}_s}$, $m'_1\mathcal{A}_{s}-K_{\mathcal{Y}_{s}}-G_{s}$ is also ample for any irreducible component $G_{s}$ of a fiber of $f_{s}\circ g_{s}$ and for any $s\in S$.
On the other hand, we obtain $m''_1\in\mathbb{Z}_{>0}$ such that for any closed point $p\in\mathcal{Y}$ and the blow up $\mu:\tilde{\mathcal{Y}}\to\mathcal{Y}$ at $p$ with the exceptional divisor $E$, $\mu^*(m''_1\mathcal{A})-E$ is ample by applying \cite[1.4.14]{Laz} to a suitable projective compactification of $(\mathcal{Y},\mathcal{A})$.
Let $s=\pi\circ g(p)$ and $\mathcal{Z}$ be an irreducible component of $\widetilde{\mathcal{Y}_s}$ such that $\mu|_{\mathcal{Z}}:\mathcal{Z}\to\mathcal{Y}_s$ is the blow up at $p\in\mathcal{Y}_s$. Then $E\cap\mathcal{Z}$ is the $\mu|_{\mathcal{Z}}$-exceptional divisor.
Now we fix an irreducible component $G_{s}$ of a fiber of $f_{s}\circ g_{s}$ such that $p\in G_s$.
Restricting divisors to $\mathcal{Z}$, we can check that $m_1:=m_1'+(d-1)m_1''$ satisfies that for any closed points $s\in S$ and $p\in\mathcal{Y}_s$, $\mu|_{\mathcal{Z}}^*(m_1\mathcal{A}_s-K_{\mathcal{Y}_s}-G_s)-(d-1)(E\cap\mathcal{Z})$ and $\mu|_{\mathcal{Z}}^*(m_1\mathcal{A}_s)-(E\cap\mathcal{Z})$ are ample. 
Let $\widetilde{G_s}:=(\mu|_{\mathcal{Z}})_*^{-1}G_s$ and note that $\widetilde{G_s}$ is smooth.
Hence, there exists a positive integer $m_2>0$ depending only on $d$ such that $m_2(m_1\mu|_{\mathcal{Z}}^*\mathcal{A}_{s}-E\cap\mathcal{Z})|_{\widetilde{G_{s}}}$ is globally generated by applying the effective base point freeness \cite[Theorem 1.1]{kollar-eff-basepoint-free} to $\widetilde{G_{s}}$ since $$(\mu|_{\mathcal{Z}}^*(m_1\mathcal{A}_{s}-K_{\mathcal{Y}_{s}}-G_{s})-(d-1)(E\cap\mathcal{Z}))|_{\widetilde{G}_{s}}=(m_1\mu|_{\mathcal{Z}}^*\mathcal{A}_{s}-(E\cap\mathcal{Z}))|_{\widetilde{G_{s}}}-K_{\widetilde{G_{s}}}$$ is ample. 
Then for any $D$ effective $\mathbb{Q}$-divisor $\mathbb{Q}$-linearly equivalent to $\epsilon\mathcal{H}_{s}+f_{s}^*\mathcal{O}(1)$, let $\gamma':=\cap_{i=1}^{d-2}D_{i}$ and $\gamma:=(\mu|_{\mathcal{Z}})_*\gamma'$ for sufficiently general $D_i\in|m_2(m_1\mu|_{\mathcal{Z}}^*\mathcal{A}_{s}-(E\cap\mathcal{Z}))|_{\widetilde{G_{s}}}|$ such that $\gamma\not \subset(g_s^*D)_{\mathrm{hor}}$ and $\gamma'\not\subset E\cap\mathcal{Z}$.
Since $m_1\mu|_{\mathcal{Z}}^*\mathcal{A}_{s}-E\cap\mathcal{Z}$ is ample, $\gamma'$ intersects with $E\cap\mathcal{Z}$.
Hence, $\gamma$ passes through $p$.
 On the other hand, we see that
\begin{align*}
\gamma\cdot (g_{s}^*D)_{\mathrm{hor}}&=\gamma'\cdot(\mu|_{\mathcal{Z}})^*((g_{s}^*D)_{\mathrm{hor}})\\
&=(m_2(m_1(\mu|_{\mathcal{Z}})^*\mathcal{A}_{s}-(E\cap\mathcal{Z})))^{d-2}\cdot\widetilde{G_{s}}\cdot(\mu|_{\mathcal{Z}})^*((g_{s}^*D)_{\mathrm{hor}} )\\
&=(m_2m_1\mathcal{A}_{s})^{d-2}\cdot G_{s}\cdot (g_{s}^*D)_{\mathrm{hor}} \\
&\le(m_2m_1\mathcal{A}_{s})^{d-2}\cdot(T_{f_{s}^*\mathcal{O}(1)}(G_{s})g_{s}^*f^*_{s}\mathcal{O}(1))\cdot (g_{s}^*D)_{\mathrm{hor}} \\
&\le(m_2m_1\mathcal{A}_{s})^{d-2}\cdot(T_{f_{s}^*\mathcal{O}(1)}(G_{s})g_{s}^*f^*_{s}\mathcal{O}(1))\cdot g_{s}^*D\\
&=\epsilon(m_1m_2)^{d-2}T_{f_{s}^*\mathcal{O}(1)}(G_{s})(\mathcal{A}_{s}^{d-2}\cdot g_{s}^*\mathcal{H}_{s}\cdot g_{s}^*f^*_{s}\mathcal{O}(1)).
\end{align*}
Let $T=\max_{G_{s}}T_{f_{s}^*\mathcal{O}(1)}(G_{s})$. 
Then we see that $T$ is independent of $s\in S$ by the conditions of the diagram (\ref{diag2}).
Let $M:=(m_1m_2)^{d-2}T(\mathcal{A}_{s}^{d-2}\cdot g_{s}^*\mathcal{H}_{s}\cdot g_{s}^*f^*_{s}\mathcal{O}(1))>0$.
Then, we see that this $M$ is independent of $s\in S$ and that $\mathrm{mult}_p((g_{s}^*D)_{\mathrm{hor}})\le M\epsilon$ for any closed points $s\in S$ and $p\in \mathcal{Y}_{s}$, and any effective $\mathbb{Q}$-divisor $D$ that is $\mathbb{Q}$-linearly equivalent to $\epsilon\mathcal{H}_{s}+f_{s}^*\mathcal{O}(1)$. 
Indeed, we saw in the above argument that there exists a curve $\gamma\not\subset (g_{s}^*D)_{\mathrm{hor}}$ passing through $p$ such that $\gamma\cdot (g_{s}^*D)_{\mathrm{hor}}\le M\epsilon$.
Then, we have $$\mathrm{mult}_p((g_s^*D)_{\mathrm{hor}})\le\gamma'\cdot\mu^*( (g_{s}^*D)_{\mathrm{hor}})=\gamma\cdot (g_{s}^*D)_{\mathrm{hor}}\le M\epsilon
.$$ 
Thus, we obtain the assertion (\ref{claim-2-(1)}).
\vspace{3mm}

Next, we deal with (\ref{claim-2-(2)}). 
Fix $\tau>0$ so that $\mathcal{A}_s^{d-1}\cdot g_s^*(\mathcal{H}_s-\tau f_s^*\mathcal{O}(1))<0$ for some closed point $s\in S$. 
Then, $\mathcal{H}_s-\tau f_s^*\mathcal{O}(1)$ is not pseudoeffective for any $s\in S$. 
Furthermore, take $m_0\in\mathbb{Z}_{>0}$ such that $m_0\mathcal{H}-(K_{\mathcal{X}/S}+\mathcal{D})$ is $\pi$-ample.

We first treat the case when $G_s$ is a smooth fiber of $f_s\circ g_s$ for any closed point $s\in S$.
Then, it follows from \cite[Lemma 2.2]{FO} that
\begin{align*}
&mh^0(\mathcal{X}_s,m\epsilon \mathcal{H}_s+mf_s^*\mathcal{O}(1))S_{m,\epsilon\mathcal{H}_s+f_s^*\mathcal{O}(1)}(G_s)\\
&\le\sum_{k=1}^{\infty} h^0(\mathcal{Y}_s,g_s^*(m\epsilon \mathcal{H}_s+mf_s^*\mathcal{O}(1))-kG_s)\\
&\le \sum_{k=1}^mh^0(\mathcal{Y}_s,g_s^*(m\epsilon \mathcal{H}_s+(m-k)f_s^*\mathcal{O}(1)))+ \sum_{k=1}^{\lceil m\epsilon\tau\rceil}h^0(\mathcal{Y}_s,g_s^*(m\epsilon \mathcal{H}_s)).
\end{align*}
Recall that $m\epsilon\mathcal{H}-K_{\mathcal{X}/S}-\mathcal{D}$ is $\pi$-ample for any $m\in\mathbb{Z}_{>0}$ such that $m\epsilon\in\mathbb{Z}$ and $m\epsilon\ge m_0$. 
Then
\begin{equation}
h^0(\mathcal{Y}_s,g_s^*(m\epsilon \mathcal{H}_s+kf_s^*\mathcal{O}(1)))=\chi(\mathcal{X}_s,m\epsilon \mathcal{H}_s+kf_s^*\mathcal{O}(1))\nonumber
\end{equation}
by the Kawamata-Viehweg vanishing theorem for any $k\ge0$ and any closed point $s\in S$. 
Hence, $h^0(\mathcal{Y}_s,g_s^*(m\epsilon \mathcal{H}_s+kf_s^*\mathcal{O}(1)))$ is independent of $s\in S$.
We further have that 
\begin{equation*}
\chi(\mathcal{X}_s,m\epsilon \mathcal{H}_s+kf_s^*\mathcal{O}(1))=\chi(\mathcal{X}_s,m\epsilon \mathcal{H}_s)+k\chi(G_s,m\epsilon \mathcal{H}_s|_{G_s}).
\end{equation*}
Here, we note that $\chi(G_s,m\epsilon \mathcal{H}_s|_{G_s})$ is also independent of $s\in S$. Thus,
\begin{equation}
h^0(\mathcal{Y}_s,g_s^*(m\epsilon \mathcal{H}_s+mf_s^*\mathcal{O}(1)))=\frac{m^{d}\epsilon^{d-1}}{(d-1)!}((\mathcal{H}_s^{d-1}\cdot f_s^*\mathcal{O}(1))+O(\epsilon)+O((m\epsilon)^{-1})).
\label{eq-chi}
\end{equation}
Furthermore, we have by \cite[Theorem 1.36]{KM} that for sufficiently large $m\epsilon$ and $\epsilon^{-1}$,
\begin{align}
 &\sum_{k=1}^mh^0(\mathcal{Y}_s,g_s^*(m\epsilon \mathcal{H}_s+(m-k)f_s^*\mathcal{O}(1)))+ \sum_{k=1}^{\lceil m\epsilon\tau\rceil}h^0(\mathcal{Y}_s,g_s^*(m\epsilon \mathcal{H}_s))\label{eq-volume}\\ 
 &=\frac{m(m-1)}{2}\chi(G_s,m\epsilon \mathcal{H}_s|_{G_s})+(m+\lceil m\epsilon\tau\rceil)\chi(\mathcal{X}_s,m\epsilon \mathcal{H}_s)\nonumber\\
 &=\frac{m^{d+1}\epsilon^{d-1}}{2(d-1)!}((\mathcal{H}_s^{d-1}\cdot f_s^*\mathcal{O}(1))+O(\epsilon)+O((m\epsilon)^{-1})).\nonumber
\end{align}
By (\ref{eq-chi}) and (\ref{eq-volume}), we see that there exist positive real numbers $C_0$, $C'_0$ and $C''_0$ such that 
\begin{equation}
S_{m,\epsilon\mathcal{H}_s+f_s^*\mathcal{O}(1)}(G_s)\le\frac{1}{2}+C'_0\epsilon+C''_0(m\epsilon)^{-1}\label{eq-smooth-fib}
\end{equation}
for any rational number $0<\epsilon<C_0^{-1}$, positive integer $m$ such that $m\epsilon\in\mathbb{Z}$ and $m\epsilon>\max\{C_0,m_0\}$, closed point $s\in S$ and smooth fiber $G_s$ of $f_s\circ g_s$. 

Next, we deal with the case when $G_s$ is an irreducible component of a singular fiber of $f_s\circ g_s$ for three paragraphs. 
Recall that $(f\circ g)^*F_j=\sum_{i=0}^{r_j} a_{ij}G_i^{(j)}$ is the irreducible decomposition. 
Then, we see that $G_s=(G^{(j)}_i)|_{\mathcal{Y}_s}$ for some $j$ and $i$.
By renumbering $G^{(j)}_i$, we may assume that $G=G^{(j)}_0$.
We note that a matrix $(G_{k,s}^{(j)}\cdot G_{l,s}^{(j)}\cdot \mathcal{A}_s^{d-2})_{1\le k,l\le r_j}$ is negative definite (cf.~\cite[Lemma 1]{LX}). 
Thus, there exists a Cartier divisor $F'=\sum _{i=1}^{r_j}e_i G_{i}^{(j)}$ such that 
\begin{equation*}
G_{i,s}^{(j)}\cdot \mathcal{A}_s^{d-2}\cdot(g_s^*\mathcal{H}_s+F'_s)<0
\end{equation*}
for $i>0$ and for some closed point $s\in S$. 
For some $b\in\mathbb{Z}_{>0}$, $F'':=F'+b(f\circ g)^*F_j$ is effective. 
Then the inequality 
\begin{equation}\label{eq-hodge}
G_{i,s}^{(j)}\cdot \mathcal{A}_s^{d-2}\cdot(g_s^*\mathcal{H}_s+F''_s)<0
\end{equation}
also holds for $F''$ and for any closed point $s\in S$.
Fix a positive integer $a$ such that $a(f\circ g)^*F_j-F''$ is effective.
We claim that for any closed point $s\in S$, $m\in\mathbb{Z}_{>0}$ such that $m\epsilon\in\mathbb{Z}$ and $k\ge 0$,
\begin{align}
&h^0(\mathcal{Y}_s,m\epsilon F''_s+g_s^*(m\epsilon \mathcal{H}_s+mf_s^*\mathcal{O}(1))-ka_{0j}G_s)\label{eq-baseloci} \\ 
&=h^0(\mathcal{Y}_s,m\epsilon F''_s+g_s^*(m\epsilon \mathcal{H}_s+(m-k)f_s^*\mathcal{O}(1))).\nonumber
\end{align}
Indeed, note that 
$$m(f_s\circ g_s)^*F_{j,s}-ka_{0j}G_s= (m-k)(f_s\circ g_s)^*F_{j,s}+k\sum_{i>0}a_{ij}G^{(j)}_{i,s}$$
and we claim that $k\sum_{i>0}a_{ij}G^{(j)}_{i,s}$ is contained in the fixed part of the linear system $|m\epsilon F''_s+g_s^*(m\epsilon \mathcal{H}_s+mf_s^*\mathcal{O}(1))-ka_{0j}G_s|$.
For this, it suffices to show for any non-zero effective divisor $M=\sum_{i>0} m_iG_{i,s}^{(j)}$ and $l\in\mathbb{Z}_{\ge0}$ that the fixed part of the linear system $|m\epsilon F''_s+g_s^*(m\epsilon \mathcal{H}_s+lf_s^*\mathcal{O}(1))+M|$ contains some $G_{i,s}^{(j)}$ such that $m_i>0$.
Here, we see that
\[
G_{i,s}^{(j)}\cdot \mathcal{A}_s^{d-2}\cdot(m\epsilon F''_s+g_s^*(m\epsilon \mathcal{H}_s+lf_s^*\mathcal{O}(1))+M)<G_{i,s}^{(j)}\cdot \mathcal{A}_s^{d-2}\cdot M
\]
by (\ref{eq-hodge}), and $G_{i,s}^{(j)}\cdot \mathcal{A}_s^{d-2}\cdot M<0$ for some $i>0$ such that $m_i>0$ by \cite[Lemma 1]{LX}.
This means that $G_{i,s}^{(j)}$ is contained in the fixed part and thus we obtain the equality (\ref{eq-baseloci}).

By the equation (\ref{eq-baseloci}) and the fact that $a(f\circ g)^*F_j-F''$ is effective, we have
\begin{align*}
m&h^0(\mathcal{X}_s,m\epsilon \mathcal{H}_s+mf_s^*\mathcal{O}(1))S_{m,\epsilon\mathcal{H}_s+f_s^*\mathcal{O}(1)}(G_s)\\
\le&\sum_{k=1}^\infty\sum_{l=0}^{a_{0j}-1} h^0(\mathcal{Y}_s,g_s^*(m\epsilon \mathcal{H}_s+mf_s^*\mathcal{O}(1))-(ka_{0j}-l)G_s)\\
\le&a_{0j}\sum_{k=1}^\infty h^0(\mathcal{Y}_s,g_s^*(m\epsilon \mathcal{H}_s+mf_s^*\mathcal{O}(1))-ka_{0j}G_s)\\
\le& a_{0j}\sum_{k=0}^{\infty} h^0(\mathcal{Y}_s,m\epsilon F''_s+g_s^*(m\epsilon \mathcal{H}_s+(m-k)f_s^*\mathcal{O}(1)))\\
\le& a_{0j}\sum_{k=0}^mh^0(\mathcal{Y}_s,g_s^*(m\epsilon \mathcal{H}_s+(m-k+m\epsilon a)f_s^*\mathcal{O}(1)))\\
+& a_{0j}\sum_{k=1}^{\lceil m\epsilon(\tau+a)\rceil}h^0(\mathcal{Y}_s,g_s^*(m\epsilon (\mathcal{H}_s+af_s^*\mathcal{O}(1)))-kg_s^*f_s^*\mathcal{O}(1)).
\end{align*}
By (\ref{eq-chi}) and estimating the right hand side of the above inequality as (\ref{eq-volume}), we see that there exist positive real numbers $C_{G_0^{(j)}},C_{G_0^{(j)}}'$ and $C_{G_0^{(j)}}''>0$ such that
\begin{equation}
S_{m,\epsilon\mathcal{H}_s+f_s^*\mathcal{O}(1)}(G_{0,s}^{(j)})\le\frac{a_{0j}}{2}+C_{G_0^{(j)}}'\epsilon+C_{G_0^{(j)}}''(m\epsilon)^{-1}\label{eq-g_0,j-fib}
\end{equation}
for any rational number $0<\epsilon<C_{G_0^{(j)}}^{-1}$, positive integer $m$ such that $m\epsilon\in\mathbb{Z}$ and $m\epsilon>\max\{C_{G_0^{(j)}},m_0\}$ and closed point $s\in S$.

Since there exist only finitely many possibilities of $G_i^{(j)}$, we see by the inequalities (\ref{eq-smooth-fib}) and (\ref{eq-g_0,j-fib}) that there exist positive real numbers $\epsilon'_0$ and $c$ such that
$$S_{m,\epsilon\mathcal{H}_s+f_s^*\mathcal{O}(1)}(G_s)\le \frac{1}{2}T_{f_s^*\mathcal{O}(1)}(G_s)+c\epsilon+c(m\epsilon)^{-1}$$
for any rational number $0<\epsilon<\epsilon'_0$, closed point $s\in S$, irreducible component $G_s$ of a fiber of $f_s\circ g_s$ and $m\in\mathbb{Z}_{>0}$ such that $m\epsilon>\epsilon_0'^{-1}$ and $m\epsilon\in\mathbb{Z}$.
Thus, we obtain the assertion (\ref{claim-2-(2)}) by taking $m_{\epsilon}\in\mathbb{Z}$ for any rational number $0<\epsilon<\epsilon'_0$ such that $m_\epsilon\epsilon>\epsilon_0'^{-1}$ and $m_\epsilon\epsilon\in\mathbb{Z}$.
 \end{proof}
 
Finally, we show that Claim \ref{lem-cl3} implies Proposition \ref{lem-delta-conv}. 
Take an arbitrary constant $0<\delta_0<u$.
Let $c$ and $\epsilon_0'$ be as in Claim \ref{lem-cl3} and take $m_{\epsilon}\in\mathbb{Z}_{>0}$ for any rational number $0<\epsilon<\epsilon_0'$ as (\ref{claim-2-(2)}). 
Let $D_m$ be an $m$-basis type divisor of $\epsilon \mathcal{H}_s+f_s^*\mathcal{O}(1)$ for any $m\in\mathbb{Z}_{>0}$ such that $m_\epsilon|m$ and $\Gamma$ be a $\mathbb{Q}$-divisor such that $g_{s*}\Gamma=\mathcal{D}_s$ and
\[
K_{\mathcal{Y}_s}+\Gamma=g_s^*(K_{\mathcal{X}_s}+\mathcal{D}_s),
\]
 for any closed point $s\in S$. 
 Now, we have 
$$A_{(\mathcal{X}_s,\mathcal{D}_s)}(G_s)\ge \frac{u}{2}T_{f_s^*\mathcal{O}(1)}(G_s)$$
 by \cite[Theorem D]{Hat} for any irreducible component $G_s$ of any fiber $F_s$ of $f_s\circ g_s$. 
 We note that there exists a positive integer $r$ such that $r(K_{\mathcal{X}/S}+\mathcal{D})$ is Cartier.
 Then we see that for any geometric point $\bar{s}\in S$, $(\mathcal{X}_{\bar{s}},\mathcal{D}_{\bar{s}})$ is $\frac{1}{r}$-lc. 
 Let $\epsilon_0'':=\min\{\epsilon_0',\frac{\delta_0}{8c(u-\delta_0)}\}$ and $\delta_0':=\min\{\frac{\delta_0}{4},\frac{1}{r}\}$.
 Then, we claim that for any $0<\epsilon<\epsilon''_0$ and $m\in\mathbb{Z}_{>0}$ such that $m_\epsilon|m$ and $m\epsilon>\epsilon''^{-1}_0$, $(\mathcal{Y}_s,\Gamma+(u-\delta_0)(g_s^*D_m)_{\rm{vert}})$ is log smooth and $\delta_0'$-sublc. 
 Indeed, this follows from the conditions of the diagram (\ref{diag2}), \cite[Corollary 2.31]{KM} and 
 \begin{align*}
 &A_{\mathcal{Y}_s}(G_s)-\mathrm{ord}_{G_s}\left(\Gamma+\left(u-\delta_0\right)(g_s^*D_m)_{\rm{vert}}\right)\\
 &\ge A_{\mathcal{Y}_s}(G_s)-\mathrm{ord}_{G_s}(\Gamma)-\left(u-\delta_0\right)S_{m,\epsilon\mathcal{H}_s+f_s^*\mathcal{O}(1)}(G_s)
 \\ 
 &\ge A_{(\mathcal{X}_s,\mathcal{D}_s)}(G_s)-\left(u-\delta_0\right)\left(\frac{T_{f_s^*\mathcal{O}(1)}(G_s)}{2}+c\epsilon+c(m\epsilon)^{-1}\right)
 \\
 &\ge\frac{\delta_0}{2}T_{f_s^*\mathcal{O}(1)}(G_s)-\left(u-\delta_0\right)(c\epsilon+c(m\epsilon)^{-1})>\frac{\delta_0}{4}.
 \end{align*}
 Let $\Theta=(\Gamma+(g_s^*D_m)_{\mathrm{vert}})_{\mathrm{red}}$. 
 Since $(1-\delta_0')\Theta\ge\Gamma+(u-\delta_0)(g_s^*D_m)_{\rm{vert}}$ and $(\mathcal{Y}_s,\Theta)$ is a log smooth pair, we further have as the argument of \cite[Step 2 in Proof of 9.14]{BHJ} that 
 $$ A_{\mathcal{Y}_s}(E)-\mathrm{ord}_E\left(\Gamma+\left(u-\delta_0\right)(g_s^*D_m)_{\rm{vert}}\right)\ge \delta'_0\,A_{\mathcal{Y}_s}(E)$$ for any prime divisor $E$ over $\mathcal{X}_s$. 
 On the other hand, Claim \ref{lem-cl3} (\ref{claim-2-(1)}) and Skoda's theorem (cf., [{\it loc.cit}, Step 1.])~show that 
 $$c\epsilon\,A_{\mathcal{Y}_s}(E)\ge \mathrm{ord}_E((g_s^*D_m)_{\rm{hor}}).$$
 Let $\epsilon_0:=\min\{\frac{\delta'_0}{(u-\delta_0)c},\epsilon_0''\}$.
 Then we have for any prime divisor $E$ over $\mathcal{X}_s$ and $0<\epsilon<\epsilon_0$,
\begin{align*}
A_{(\mathcal{X}_s,\mathcal{D}_s+(u-\delta_0)D_m)}(E)&=A_{\mathcal{Y}_s}(E)-\mathrm{ord}_E\left(\Gamma+\left(u-\delta_0\right)((g_s^*D_m)_{\mathrm{vert}}+(g_s^*D_m)_{\mathrm{hor}})\right)\\
&\ge(\delta_0'-(u-\delta_0)c\epsilon)A_{\mathcal{Y}_s}(E)>0.
\end{align*}
In other words, we conclude that $(\mathcal{X}_s,\mathcal{D}_s+(u-\delta_0)D_m)$ is klt for any closed point $s\in S$, rational number $\epsilon\in(0,\epsilon_0)$, $m\in\mathbb{Z}$ such that $m_{\epsilon}|m$ and $m>(\epsilon\epsilon''_0)^{-1}$, and $m$-basis type divisor $D_m$ of $\epsilon\mathcal{H}_s+f_s^*\mathcal{O}(1)$. 
Thus, we obtain that 
\[
\delta_{(\mathcal{X}_s,\mathcal{D}_s)}(\epsilon\mathcal{H}_s+f_s^*\mathcal{O}(1))=\lim_{l\to\infty}\delta_{lm_\epsilon,(\mathcal{X}_s,\mathcal{D}_s)}(\epsilon\mathcal{H}_s+f_s^*\mathcal{O}(1))\ge u-\delta_0
\]
 for any $\epsilon\in\mathbb{Q}\cap(0,\epsilon_0)$ and closed point $s\in S$. 
 We complete the proof.
 \end{proof}

 Now, we are ready to show Theorem \ref{mainunif} for the case when $u<0$.
 
 \begin{thm}\label{unifKst}
 Let $\pi \colon  (\mathcal{X},\mathcal{D}) \to S$ be a log $\mathbb{Q}$-Gorenstein family and $f \colon \mathcal{X}\to \mathcal{P}$ a contraction over $S$, where $\mathcal{P}$ is a scheme that is projective and smooth over $S$. 
Let $\mathcal{H}$ be a $\pi$-ample $\mathbb{Q}$-divisor on $\mathcal{X}$ and $\mathcal{L}$ a Cartier divisor on $\mathcal{P}$. Suppose that there exists an $m\in\mathbb{Z}_{>0}$ such that $(\mathcal{P}_{\bar{s}},\mathcal{L}_{\bar{s}})=(\mathbb{P}^1,\mathcal{O}(m))$ for any geometric point $\bar{s}\in S$. 
Assume that $-(K_{\mathcal{X}/S}+\mathcal{D})\sim_{\mathbb{Q},S}\frac{u}{m}f^*\mathcal{L}$ for some $u\in\mathbb{Q}_{>0}$ and that $f_s \colon (\mathcal{X}_s,\mathcal{D}_s,\mathcal{H}_s)\to(\mathbb{P}^1,\mathcal{O}(1))$ is uniformly adiabatically K-stable for any closed point $s\in S$.
 
 Then there exists a positive rational number $\epsilon_0$ such that $(\mathcal{X}_s,\mathcal{D}_s,\epsilon\mathcal{H}_s+f_s^*\mathcal{O}(1))$ is specially K-stable for any closed point $s\in S$ and rational number $\epsilon\in (0,\epsilon_0)$.  Furthermore, there exists a positive rational number $\alpha$ such that
 \[
M^\mathrm{NA}_{\mathcal{D}_s}(\mathcal{Y},\mathcal{M})\ge\alpha (\mathcal{J}^{\epsilon\mathcal{H}_s+f_s^*\mathcal{O}(1)})^\mathrm{NA}(\mathcal{Y},\mathcal{M})
\]
 for any rational number $\epsilon\in (0,\epsilon_0)$, closed point $s\in S$, and normal semiample test configuration $(\mathcal{Y},\mathcal{M})$ for $(\mathcal{X}_s,\epsilon\mathcal{H}_s+f_s^*\mathcal{O}(1))$. 
 \end{thm}

\begin{proof}
As the argument in the second paragraph of the proof of Proposition \ref{lem-delta-conv}, we may freely shrink or replacing $S$ by $S'$ with an \'{e}tale morphism $S'\to S$.
Thus, by the argument as in the first paragraph of the proof of Theorem \ref{op}, we may assume that $\mathcal{P}=\mathbb{P}^1_S$, $\mathcal{L}=\mathcal{O}(1)$, and $m=1$.

By Theorem \ref{op} and \cite[Theorem 4.4]{Hat}, there exists a positive rational number $\delta_0$ such that $$\delta_{(\mathbb{P}^{1},B_{s})}(\mathcal{O}(1))=\lim_{\epsilon\to0}\delta_{(\mathcal{X}_s,\mathcal{D}_s)}(\epsilon\mathcal{H}_s+f_s^*\mathcal{O}(1))\ge u+\delta_0$$
for any closed point $s\in S$, where $B_s$ is the discriminant divisor of $f_s\colon(\mathcal{X}_s,\mathcal{D}_s)\to\mathbb{P}^1$.
From now on, we set $\mathscr{L}_{\epsilon, s}:=\epsilon\mathcal{H}_s+f_s^*\mathcal{O}(1)$ and $\delta(\epsilon,s):=\delta_{(\mathcal{X}_s,\mathcal{D}_s)}(\mathscr{L}_{\epsilon, s})$ for any closed point $s\in S$.
By Proposition \ref{lem-delta-conv}, there exists a positive rational number $\epsilon_0$, which is independent of $s\in S$, such that 
$$\delta(\epsilon,s)\ge u+\frac{\delta_0}{2}$$
for any $\epsilon\in (0,\epsilon_0)$. 
Then this $\epsilon_0$ satisfies the assertion of Theorem \ref{unifKst}.
Indeed, for any normal semiample test configuration $(\mathcal{Y},\mathcal{M})$ for $(\mathcal{X}_s,\mathscr{L}_{\epsilon, s})$, we have
 \[
 M^{\mathrm{NA}}_{\mathcal{D}_s}(\mathcal{Y},\mathcal{M})\ge (\mathcal{J}^{K_{\mathcal{X}_s}+\mathcal{D}_s+\delta(\epsilon,s)\mathscr{L}_{\epsilon,s}})^{\mathrm{NA}}(\mathcal{Y},\mathcal{M})
 \]
 by 
 Theorem \ref{A.13}. We also have
 \[
 K_{\mathcal{X}_s}+\mathcal{D}_s+\left(u+\frac{\delta_0}{2}\right)\mathscr{L}_{\epsilon,s}\sim_{\mathbb{Q}}\frac{\delta_0}{2}\mathscr{L}_{\epsilon,s}+u\epsilon\mathcal{H}_s.
 \]
 By  the ampleness of $\mathcal{H}_s$ and the argument in the proof of \cite[Lemma 4.5]{Hat}, we see that $(\mathcal{X}_s,\mathscr{L}_{\epsilon,s})$ is J$^{\mathcal{H}_s}$-semistable. 
 Thus, we obtain
  \[
(\mathcal{J}^{K_{\mathcal{X}_s}+\mathcal{D}_s+{\delta(\epsilon,s)\mathscr{L}_{\epsilon,s}}})^{\mathrm{NA}}(\mathcal{Y},\mathcal{M})\ge\frac{\delta_0}{2}(\mathcal{J}^{\mathscr{L}_{\epsilon,s}})^\mathrm{NA}(\mathcal{Y},\mathcal{M})
\]
for any rational number $\epsilon\in (0,\epsilon_0)$.
We complete the proof.
\end{proof}

\begin{proof}[Proof of Theorem \ref{mainunif}]
We first treat the case when $u\ne0$.
By  Lemma \ref{lem--Cartierindex}, Theorem \ref{thm--eff-veryample}, and Remark \ref{rem--ell-bound}, it is sufficient to discuss the assertion of Theorem \ref{mainunif} for $\mathfrak{G}_{d, n, m^{d-1}v, u,m^d w}^{\mathrm{Car}}$ for some $n$ and $m$ instead of $\mathfrak{G}_{d, \Theta, v, u,w}$. 
In this situation, the assertion of Theorem \ref{mainunif} immediately follows from Theorem \ref{op}, Proposition \ref{bou6}, and Theorems \ref{unif+} and \ref{unifKst}. 

Next, we deal with the case when $u=0$.
As in the previous paragraph, it is enough to show the assertion for $\mathfrak{G}_{d, \Theta, v,0,w}^{\mathrm{Car}}$.
There exist $I\in\mathbb{Z}_{>0}$ and $r\in\mathbb{Z}_{>0}$ in Corollary \ref{cor--hilbertpolynomial} and Remark \ref{rem--ell-bound} respectively, such that for any element $f\colon(X,\Delta,A)\to C\in\mathfrak{G}_{d, \Theta, v,0,w}^{\mathrm{Car}}$ and any line bundle $H$ on $C$ of degree one, $r(K_X+\Delta)$ is Cartier and $I(A+f^{*}H)$ is very ample.
In particular, $(X,\Delta)$ is $\frac{1}{r}$-lc.
Taking a small $\mathbb{Q}$-factorialization of $X$ and applying the length of extremal rays, we see that $K_{X}+3dI(A+f^{*}H)$ is the pushdown of a big divisor (cf.~\cite[Lemma 2.46]{birkar-compl}). 
Because we have
$$3dI(A+f^{*}H)-\Delta\sim_{\mathbb{Q}}3dI(A+f^{*}H)+K_{X},$$
we see that $3dI(A+f^{*}H)-\Delta$ is the pushdown of a big divisor. 
We also have 
$${\rm vol}(3dI(A+f^{*}H))\le(3dI)^{d}(w+dv).$$ 
Thus, we may apply \cite[Theorem 1.8]{birkar-bab} to $(X,\Delta)$ and $3dI(A+f^{*}H)$, and we obtain $\delta_0>0$, depending only on $d, \, \Theta, \, v$, and $w$, such that 
$$\alpha_{(X,\Delta)}(A+f^{*}H)\ge \delta_0.$$
Here, we use the fact that $\alpha_{(X,\Delta)}(3dI(A+f^{*}H))=(3dI)^{-1}\alpha_{(X,\Delta)}(A+f^{*}H)$.
By the inequality (\ref{eq-alpha-calculate}) in the proof of Theorem \ref{unif+}, we obtain 
\[
\delta_{(X,\Delta)}(\epsilon A+f^{*}H)\ge\frac{d+1}{d}\delta_0
\]
for any $0<\epsilon<1$.
By Theorem \ref{A.13} and the fact that $K_X+\Delta\sim_{\mathbb{Q}}0$, for any element $f\colon(X,\Delta,A)\to C\in\mathfrak{G}_{d, \Theta, v,0,w}^{\mathrm{Car}}$,  $(X,\Delta,\epsilon A+f^*H)$ is specially K-stable and
\[
M^{\mathrm{NA}}_{\Delta}(\mathcal{X},\mathcal{M})\ge\frac{d+1}{d}\delta_0(\mathcal{J}^{\epsilon A+f^*H})^{\mathrm{NA}}(\mathcal{X},\mathcal{M})
\]
for any normal semiample test configuration $(\mathcal{X},\mathcal{M})$ for $(X,\epsilon A+f^*H)$, line bundle $H$ on $C$ of degree one, and rational number $\epsilon \in (0,1)$.
We complete the proof.
\end{proof}

\begin{proof}[Proof of Corollary \ref{cor-main}]
By Theorem \ref{mainbound} (\ref{bound-(2)}), there exists a positive integer $w'$, depending only on $d$, $v$, and $u$, such that for any $f\colon(X,\Delta,A)\to C$ as assumption and the general fiber $F$ of $f$, $A+t'_{A}F$ is ample and ${\rm vol}(A+t'_{A}F) \leq w'$ for some integer $t'_A$. 
For some positive rational number $t''_A$, we have ${\rm vol}(A+(t'_{A}+t''_A)F) = w'$.
We set 
$$t_A:=t'_{A}+t''_A\in\mathbb{Q}_{>0}.$$
Note that $A+t'_AF$ is a $\mathbb{Q}$-Cartier Weil divisor.
Then Theorem \ref{mainunif} shows that there exists $\epsilon_0>0$, depending only on $d$, $v$, and $u$, such that $(X,\Delta,\frac{\epsilon}{1+\epsilon t''_A} (A+t'_{A}F)+F)$ is specially K-stable for any $0<\epsilon\le\epsilon_0$ and any $f$.
Then so is $(X,\Delta,A+(t_A+\epsilon^{-1})F)$. 
Here, we use the fact that $$\frac{\epsilon}{1+\epsilon t''_A} (A+t'_{A}F)+F=\frac{\epsilon}{1+\epsilon t''_A} (A+(t_A+\epsilon^{-1})F).$$
Since
 $$w:=\mathrm{vol}\left(A+(t_A+\epsilon_0^{-1})F\right)=w'+d\epsilon_0^{-1}v$$ is independent of $f$, it follows that $(X,\Delta,A+tF)$ is specially K-stable if $\mathrm{vol}(A+tF)\ge w$.
 
As noted in the proof of \cite[Theorem 4.6]{Hat}, the last assertion follows from the special K-stability, \cite[Theorem 1.1]{G}, \cite[Corollary 5.2]{Z}, and \cite[Theorem 1.3]{Ch}.
\end{proof}

Finally, we remark that Corollary \ref{cor-main} states that $\mathscr{M}_{d,v,u,r}$ has a family whose geometric fibers are specially K-stable in the following sense.

\begin{rem}\label{rem-final}
As \cite[\S5]{deligne-mumford}, we can construct the {\it universal family} $(\mathscr{U},\mathscr{A}_{\mathscr{U}})$ over $\mathscr{M}_{d,v,u,w,r}$ as follows.
In the proof of Theorem \ref{mod2}, $\mathscr{M}_{d,v,u,w,r}$ is the disjoint union of finitely many $\mathscr{M}_{d_1,d_2,d_3,h}$.
Thus, we can construct a (Deligne-Mumford) stack
\[
\mathscr{U}:=\bigsqcup_{d_1,d_2,d_3,h}[\mathcal{U}_{N_{d_1,d_2,d_3,h}}/PGL(d_1)\times PGL(d_2)\times PGL(d_3)],
\]
where $\mathcal{U}_{N_{d_1,d_2,d_3,h}}$ is the universal family of $N_{d_1,d_2,d_3,h}$.
Let $\tilde{\mathscr{A}}_{N_{d_1,d_2,d_3,h}}$ be the line bundle on $\mathcal{U}_{N_{d_1,d_2,d_3,h}}$ as in the proof of Claim \ref{cl2}. 
Note that $\tilde{\mathscr{A}}_{N_{d_1,d_2,d_3,h}}$ is ample over $N_{d_1,d_2,d_3,h}$.
By construction, we see that $\tilde{\mathscr{A}}_{N_{d_1,d_2,d_3,h}}$ is $PGL(d_1)\times PGL(d_2)\times PGL(d_3)$-equivariant.
By Example \ref{ex--quotient-stack}, we can construct $\mathscr{A}_{\mathscr{U}}$ as the line bundle on $\mathscr{U}$ whose restriction to $[\mathcal{U}_{N_{d_1,d_2,d_3,h}}/PGL(d_1)\times PGL(d_2)\times PGL(d_3)]$ corresponds to $\tilde{\mathscr{A}}_{N_{d_1,d_2,d_3,h}}$.

    In the proof of Theorem \ref{quesmain}, we have proved that $\mathscr{M}_{d,v,u,r}$ is an open and closed substack of $\mathscr{M}_{d,v,u,w'+dv,r}$ for some $w'>0$.
    Using this fact, we obtain a family 
    $$(\mathscr{U}',\mathscr{A}_{\mathscr{U}'}):=(\mathscr{U}\times_{\mathscr{M}_{d,v,u,w'+dv,r}}\mathscr{M}_{d,v,u,r},\mathscr{A}_{\mathscr{U}}|_{\mathscr{U}'})$$ 
    over $\mathscr{M}_{d,v,u,r}$.
    If we take $w'$ larger than $w$ in Corollary \ref{cor-main}, then all geometric fibers of $(\mathscr{U}',\mathscr{A}_{\mathscr{U}'})$ over $\mathscr{M}_{d,v,u,r}$ are specially K-stable. 
\end{rem}


\end{document}